\documentclass[11pt]{article}
\usepackage{graphicx}
\usepackage{epsfig}
\usepackage{amssymb, amsmath}
\usepackage{amsthm}
\usepackage{setspace}
\usepackage[top = 1in, bottom = 1in, left = 1.2in, right =
1.2in]{geometry}

\newtheorem{theorem}{Theorem}[section]

\newtheorem{claim}[theorem]{Claim}

\newtheorem{corollary}[theorem]{Corollary}

\newtheorem{lemma}[theorem]{Lemma}

\newtheorem{proposition}[theorem]{Proposition}
\newtheorem{remark}[theorem]{Remark}

\renewcommand{\L}{\mathcal L}
\newcommand{\A}{\mathcal A}
\newcommand{\eps}{\epsilon}

\begin{document}

\title{Well-posedness for compressible Euler equations
 with physical vacuum singularity}
\author{Juhi Jang and Nader Masmoudi}
\singlespacing
\maketitle

\begin{abstract}

An important problem in the theory of compressible gas flows is to
understand the singular behavior of vacuum states. The main
difficulty lies in the fact that the system becomes degenerate at
the vacuum boundary, where the characteristics coincide and have
unbounded  derivative.
 In
this paper, we overcome this difficulty by presenting a new
formulation and   new energy spaces. We establish the local in
time well-posedness of one-dimensional compressible Euler equations
for isentropic flows with the physical vacuum singularity in some
spaces adapted to the singularity.
\end{abstract}

\section{Introduction}

One-dimensional compressible Euler equations for the isentropic flow
with damping in Eulerian coordinates read as follows:
\begin{equation}\label{EDE}
 \begin{split}
  \rho_t+(\rho u)_x&=0,\\
\rho u_t+\rho uu_x +p_x&=-\rho u,
 \end{split}
\end{equation}
with initial data $\rho(0,x)=\rho_0(x)\geq 0$ and $u(0,x)=u_0(x)$
prescribed. Here $\rho$, $u$, and $p$ denote respectively the
density, velocity, and pressure. We consider the polytropic gases:
the equation of state is given by $p=A\rho^\gamma$, where $A$ is an
entropy constant and $\gamma>1$ is the  adiabatic gas exponent. When
the initial density function contains vacuum, the vacuum boundary
$\Gamma$ is defined as
\[
 \Gamma=cl\{(t,x):\rho(t,x)>0\}\cap cl\{(t,x):
\rho(t,x)=0\}\,
\]
where $cl$ denotes the closure.
A vacuum boundary is called \textbf{physical} if
\begin{equation}\label{pvb}
0<|\frac{\partial c^2}{\partial x}|<\infty
\end{equation}
in a small neighborhood of the boundary, where
$c=\sqrt{\frac{d}{d\rho}p(\rho)}$ is the sound speed. This physical
vacuum behavior can be realized by some self-similar solutions and
stationary solutions for different physical systems such as Euler
equations with damping, Navier-Stokes equations or Euler-Poisson
equations for gaseous stars. For more details and the physical
background regarding this concept of the physical vacuum boundary,
we refer to \cite{L2,LY2,Y}.

Despite its physical importance, even the local existence theory of
smooth solutions featuring the physical vacuum boundary has not been
completed yet. This is because the hyperbolic system becomes
degenerate at the vacuum boundary and in particular if the physical
vacuum boundary condition \eqref{pvb} is assumed, the classical
theory of  hyperbolic systems can not be applied \cite{LY2}: the
characteristic speeds of Euler equations are $u\pm c$, thus they
become singular with infinite spatial derivatives at the vacuum
boundary and this singularity creates a severe analytical
difficulty. To our knowledge, there has been no satisfactory theory
to treat this kind of singularity. The purpose of this article is to
investigate the local existence and uniqueness  theory of regular
solutions (in a sense that will be made precise later and which is
adapted to the singularity of the problem) to compressible Euler
equations featuring this physical vacuum boundary.

Before we formulate our problem, we briefly review some existence
theories of compressible flows with vacuum states from various
aspects. We will not attempt to address exhaustive references in
this paper.  First, in the absence of vacuum, namely if the density
is bounded below from zero, then one can use the theory of symmetric
hyperbolic systems; for instance, see \cite{Majda84}. In particular,
the author in \cite{Sideris85} gave a sufficient condition for
non-global existence when the density is bounded below from zero.
When the data is compactly supported, there are two ways of looking
at the problem. The first consists in solving  the Euler equations
in the whole space and requiring that the equations hold for all $x$
and $t \in (0,T)$. The second way is to require the Euler equations
to hold on the set $\{(t,x):\rho(t,x)>0\}$ and write an equation for
the vacuum boundary $\Gamma$ which is a free boundary.

Of course in the first way, there is no need of knowing exactly the
position of the vacuum boundary. The authors in \cite{MUK86} wrote
the system in a symmetric hyperbolic form which allows the density
to vanish. The system they get is not equivalent to the Euler
equations when the density vanishes. This was also used for the
Euler-Poisson system. As noted by the authors \cite{Makino92,MU87},
the requirement that $\rho^{\gamma-1 \over 2}$ is continuously
differentiable excludes many interesting solutions such as the
stationary solutions of the Euler-Poisson which have a behavior of
the type  $\rho \sim |x|^{1 \over \gamma -1}$ at the vacuum
boundary. This formulation was also used in \cite{Chemin90} to prove
the  local existence of regular solutions in the sense that $
\rho^{\gamma-1 \over 2}, u \in C([0,T); H^m(\mathbb{R}^d))  $ for
some $m > 1 + d/2$ and $d$ is the space dimension (see also
\cite{S}, for some global existence result of classical solutions
under some special conditions on the initial data, by extracting a
dispersive effect after some invariant transformation, and
\cite{Grassin98}).

For the  second way  and when the singularity is mild, some
existence results of smooth solutions are available, based on the
adaptation of the theory of symmetric hyperbolic system. In
\cite{LY1}, the local in time solutions to Euler equations with
damping \eqref{EDE} were constructed when $c^\alpha$, $0<\alpha\leq
1$ is smooth across $\Gamma$, by using the energy method and
characteristic method. They also prove that $C^1$ solutions cross
$\Gamma$ can not be global. However, with or without damping, the
methods developed therein are not applicable to the local
well-posedness theory of the physical vacuum boundary. We only
mention a result in \cite{XY05} for perturbation of a planar wave.
 For other interesting aspects of vacuum states and
related problems, we refer to \cite{LY2,Y}.

As the above results indicate, there is an interesting distinction
between flows with damping and without damping, when the long time
behavior is considered. Indeed, without damping, it was shown in
\cite{LS} that the shock waves vanish at the vacuum and the singular
behavior is similar to the behavior of the centered rarefaction
waves corresponding to the case when $c$ is regular \cite{LY2}.  On
the other hand, with damping, it was conjectured in \cite{L2} that
time asymptotically, Euler equations with damping \eqref{EDE} should
behave like the porus media equation, where the canonical boundary
is characterized by the physical vacuum boundary condition
\eqref{pvb}. This conjecture was established in \cite{HMP} in the
entropy solution framework where the method of compensated
compactness yields a global weak solution in $L^\infty$. We point
out that the difficulty coming from the resonance due to vacuum in
there is very different from the difficulty that we are facing,
since we want to have some regularity so that the vacuum boundary is
well-defined and the evolution of the vacuum boundary can be
realized.

In order to understand the physical vacuum boundary behavior, the
study of regular solutions is very important and fundamental;
the evolution of the vacuum boundary should be considered as
the free boundary/interface generated by vacuum. In the presence of
viscosity, there are some existence theories available with
the physical vacuum boundary:  The vacuum interface behavior as
well as the regularity to one-dimensional Navier-Stokes free
boundary problems were investigated in \cite{LXY}. And the local
in time well-posedness of Navier-Stokes-Poisson equations in
three dimensions with radial symmetry featuring the physical
vacuum boundary was established in \cite{J}. On the other hand,
 the free surface boundary problem was studied in \cite{L1}
in the motion of a compressible liquid with vacuum by using
Nash-Moser iteration; the physical
boundary was treated in a sense that the pressure vanishes on the
boundary and the pressure gradient is bounded away from zero, but
the density has to be bounded away from vacuum, and thus the
analysis is not applicable for the motion of gas with vacuum.

In the next section, we formulate the problem and we state the main
result: The vacuum free boundary problem is studied in Lagrangian
coordinates so that the free boundary becomes fixed. By change of
variables, the equations can be written as the first order system
with non-degenerate propagation speeds that have different behaviors
inside the domain and on the vacuum boundary. In order to cope with
these nonlinear coefficients, which give rise to the main analytical
difficulty,
 the new operators $V,V^\ast$ are
introduced. Our theorem is stated in  $V,V^\ast$ framework.

\section{Formulation and Main result}

We study the initial boundary value problem to one-dimensional Euler
equations with or without damping
for isentropic flows \eqref{EDE}. First, we
impose the fixed boundary condition on one boundary $x=b\,:$ $u(t,
b)=0$. The class of the initial data $\rho_0$, $u_0$ of our interest
 is characterized as follows: for $a\leq x\leq b$, where
 $-\infty<a<b\leq\infty$
\[
\begin{split}
(i)&\,\,\rho_0(a)=0\,,\;
 0<\frac{d}{dx}\rho_0^{\gamma-1}|_{x=a}<\infty\,;\\
(ii)&
\,\;\rho_0(x)> 0 \text{ for } a<x\leq b\,;\\
(iii)&\,\int_a^b \rho_0(x)dx<\infty\,;
(iv)\,\,u_0(b)=0\,.
\end{split}
\]
The condition $(i)$ implies that the initial vacuum is physical,
$(ii)$ means that $x=a$ is the only vacuum, $(iii)$ represents the
finite total mass of gas, $(iv)$ is the compatibility condition with the
 boundary condition at $x=b$.
We seek $\rho(t,x)$, $u(t,x)$, and $a(t)$ for
 $t\in [0,T]$, $T>0$ and $x\in [a(t),b]$,
 so that for such $t$ and $x$,
\begin{equation*}
 \begin{split}
& \rho(t,x) \text{ and } u(t,x) \text{ satisfy }
\eqref{EDE}\,;\\
& \rho(t,a(t))=0\,;\;u(t,b)=0\,; \\
&0<\frac{\partial}{\partial x}\rho^{\gamma-1}|_{x=a(t)}<\infty\,.
 \end{split}
\end{equation*}
For regular solutions, the vacuum boundary $a(t)$ is the particle
path through $x=a$. In one-dimensional gas dynamics,
 there is a natural Lagrangian coordinates
transformation where all the particle paths are straight lines:
\[
y\equiv \int_{a(t)}^x \rho(t,z)dz,\;\;  a(t)\leq x\leq b\,.
\]
Note that $0\leq y\leq M$, where $M$ is the total mass of the gas.
Under this transformation, the vacuum free boundary $x=a(t)$
corresponds to $y=0$, and $x=b$ to $y=M$; thus both boundaries are
fixed in $(t,y)$. By this change of variables,
\[
\partial_t=\partial_t-
u\partial_y,\;\;\partial_x=\rho {\partial_y}
\]
the system \eqref{EDE} takes the following form in Lagrangian
coordinates $(t,y)$: for $t\geq 0$ and $0\leq y\leq M$,
\begin{equation}
\begin{split}
\rho_t+\rho^2u_y&=0\\
u_t+p_y&=-u\label{ed}
\end{split}
\end{equation}
where $p=A\rho^\gamma$ with $\gamma>1$. The boundary conditions are
given by $\rho(t,0)=0$ and $u(t,M)=0$. The physical singularity
\eqref{pvb} in Eulerian coordinates corresponds to $0< |p_y|<\infty$
in Lagrangian coordinates and thus the physical vacuum boundary
condition at $y=0$ can be realized as $$\rho\sim
y^{\frac{1}{\gamma}}\;\text{ for }\;y\sim 0\,.$$ Euler equations
(\ref{ed}) can be rewritten as a symmetric hyperbolic system
\begin{equation}
\begin{split}
\phi_t+\mu u_y&=0\,,\\
u_t+\mu\phi_y&=-u\,,\label{eds}
\end{split}
\end{equation}
in the variables $$\phi=\frac{2\sqrt{A\gamma}}{\gamma-1}
\rho^{\frac{\gamma-1}{2}}\text{ and }\mu=\sqrt{A\gamma}\rho
^{\frac{\gamma+1}{2}}\,.$$ Note that the propagation speed $\mu$
becomes degenerate and the degeneracy for the physical singularity
is given by  $\mu\sim y^{\frac{\gamma+1}{2\gamma}}$. In order to get
around this difficulty, we  introduce  the following change of
variables:
\[
\xi\equiv
\frac{2\gamma}{\gamma-1}y^{\frac{\gamma-1}{2\gamma}}\;\text{ such
that }\;{\partial_y}=y^{-\frac{\gamma+1}{2\gamma}} {\partial_\xi}\,.
\]
We normalize $A$ and $M$ appropriately such that the equations
(\ref{eds}) take the form:
\begin{equation*}
 \begin{split}
  \phi_t+(\frac{\phi}{\xi})^{\frac{\gamma+1}{\gamma-1}}u_{\xi}&=0\,,\\
u_t+(\frac{\phi}{\xi})^{\frac{\gamma+1}{\gamma-1}}\phi_{\xi}&=-u\,,
 \end{split}
\end{equation*}
for $t\geq 0$ and $0\leq \xi\leq 1$. The physical singularity
condition $0<|p_y|<\infty$ is written as $0<|\phi_\xi| <\infty$.
Thus we expect $\phi$ to be more or less $\xi$ for short time near
$0$. Since the damping is not important for the local theory, for
simplicity, we consider the pure Euler equations. Letting
\[
 k=k_\gamma\equiv \frac{1}{2}\frac{\gamma+1}{\gamma-1}\,,
\]
the Euler equations read in $(t,\xi)$ as follows: for $t\geq 0$ and
$0\leq \xi\leq 1$,
\begin{equation}
 \begin{split}
  \phi_t+(\frac{\phi}{\xi})^{2k}u_{\xi}&=0\,,\\
u_t+(\frac{\phi}{\xi})^{2k}\phi_{\xi}&=0\,.\label{euler}
 \end{split}
\end{equation}
The range of $k$ is $\frac{1}{2}< k<\infty$, since $\gamma>1$. When
$\gamma\rightarrow 1$, $k\rightarrow \infty$ and when $\gamma=3$, we
get $k=1$. Note that the propagation speed is now non-degenerate.
However, its behavior is quite different in the interior and on the
boundary, since $\lim_{\xi\searrow 0}\frac{\phi}{\xi}=\phi_\xi(0)$,
but $\phi\leq\frac{\phi}{\xi}\leq c\phi$ if $\xi\geq \frac{1}{c}>0$.
This makes it hard to apply any standard energy method to construct
solutions in the current formulation.

We will propose a new formulation to \eqref{euler} such that some
energy estimates can be closed in the appropriate energy space. As a
preparation, we first define the operators $V$ and $V^{\ast}$
associated to \eqref{euler} as follows:
\begin{equation*}
\begin{split}
V(f)\equiv \frac{1}{\xi^{k}}\partial_\xi
[\frac{\phi^{2k}}{\xi^{k}}f],\;\;\;
V^{\ast}(g)\equiv-\frac{\phi^{2k}}{\xi^{k}}\partial_\xi[
\frac{1}{\xi^{k}} g]\,.
\end{split}
\end{equation*}
for $f,\;g\in L_\xi^2$, where we have denoted  $L_\xi^2[0,1]$ by
$L_\xi^2$. We can think of $V$ and $V^{\ast}$ as modified first
order spatial derivatives.  We also incorporate the boundary
condition at $\xi = 1$ in the domain of
 of $V$ and $V^\ast$, namely  $V$ and $V^\ast$, are given as follows:
% the set of
% admissible functions of $V$ and $V^\ast$, are given as follows:
\begin{equation}\label{domain}
\begin{split}
&\mathcal{D}(V)=\{f\in L_\xi^2: V(f)\in L_\xi^2\}\\
&\mathcal{D}(V^\ast)=\{g\in L_\xi^2: V^\ast (g)\in L_\xi^2, \ g(\xi=1) = 0
. \}
\end{split}
\end{equation}
We also introduce the higher order operators $(V)^i$ and
$(V^\ast)^i$: for $f\in \mathcal{D}(V)$ and $g\in
\mathcal{D}(V^\ast)$,
\begin{equation} \label{Vi}
(V)^i(f)\equiv
\begin{cases}
        (V^\ast V)^j(f) & \text{if }i=2j\\
        V(V^\ast V)^j(f) & \text{if }i=2j+1
  \end{cases}
\end{equation}
\begin{equation} \label{V*i}
(V^\ast)^i(g)\equiv
\begin{cases}
        (VV^\ast )^j(g) & \text{if }i=2j\\
        V^\ast(VV^\ast )^j(g) & \text{if }i=2j+1
  \end{cases}
\end{equation}
and the associated function spaces ${X}^{k,s}$ and ${Y}^{k,s}$ for
$s$ a given nonnegative integer:
\begin{equation}\label{XY}
\begin{split}
{X}^{k,s}&\equiv \{f\in L_\xi^2: (V)^i(f)\in L_\xi^2,\; 0\leq i\leq s\}\\
{Y}^{k,s}&\equiv \{g\in L_\xi^2: (V^\ast)^i (g)\in L_\xi^2,\;0\leq
i\leq s\}
\end{split}
\end{equation}
equipped with the following norms
\[
||f||^2_{{X}^{k,s}}\equiv \sum_{i=0}^s
||(V)^i(f)||^2_{L^2_\xi}\;\text{ and }\;||g||^2_{{Y}^{k,s}}\equiv
\sum_{i=0}^s ||(V^\ast)^i(g)||^2_{L^2_\xi}\,.
\]
In order to emphasize the dependence of $k$, equivalently $\gamma$,
we keep $k$ in the above definitions.

In terms of $V$ and $V^{\ast}$,  the Euler equations (\ref{euler})
can be rewritten as follows:
\begin{equation}
\begin{split}
&\partial_t(\xi^{k}\phi)
-V^{\ast}(\xi^{k}u)=0\,,\\
&\partial_t(\xi^{k}u) +\frac{1}{2k+1}V(\xi^{k}\phi)=0\,, \label{VVk}
\end{split}
\end{equation}
with the boundary conditions
\begin{equation}\label{BC}
 \phi(t,0)=0\:\text{ and }\:u(t,1)=0\,.
\end{equation}
In this new $V,V^\ast$ formulation, the system  is akin to the
symmetric hyperbolic system with respect to $V,V^\ast$. In
particular, the zeroth energy estimates assert that this $V,V^\ast$
formulation retains the energy conservation property, which is
equivalent to the physical energy in Eulerian coordinates: It is
well known that the energy of Euler equations without damping is
conserved for regular solutions:
\[
\frac{d}{dt}\{\int_{a(t)}^b \frac{1}{2}\rho u^2 + \frac{p}{\gamma-1}dx\} =0\,,
\]
and in turn, since $dy=\rho dx$, it is written as, in $y$ variable,
\[
 \frac{d}{dt}\{\int_{0}^M \frac{1}{2} u^2 +\frac{A}{\gamma-1}\rho^{\gamma-1}dy\}
 =0\,.
\]
It is routine to check that this is equivalent to
\[
\frac{1}{2} \frac{d}{dt}\{\int_0^1
\xi^{\frac{\gamma+1}{\gamma-1}}u^2+\frac{\gamma-1}{2\gamma}
\xi^{\frac{\gamma+1}{\gamma-1}}\phi^2d\xi\}
 =0\,,
\]
which is exactly the zeroth energy estimates of \eqref{VVk} with respect to
$V,V^\ast$:
\begin{equation}\label{0}
 \frac{1}{2} \frac{d}{dt}\{\int_0^1
\frac{1}{2k+1}|\xi^k\phi|^2+ |\xi^k u|^2 d\xi\}
 =0\,,
\end{equation}
This verifies  that this $V,V^\ast$ formulation does not destroy the
underlying structure of the Euler equations. Indeed,
it also captures the precise structure of the
singularity caused by the physical vacuum boundary, and furthermore,
it enables
us to perform $V,V^\ast$ energy estimates yielding the a priori
estimates and the well-posedness of the system.

In order to state the main result of this article precisely, we
now define the energy functional $\mathcal{E}^{k,s}(\phi,u)$ by
\begin{equation}\label{energy}
\begin{split}
\mathcal{E}^{k,s}(\phi,u)&\equiv
\frac{1}{2k+1}||\xi^k\phi||^2_{L^2_\xi} +||\xi^k
u||^2_{L^2_\xi}\\
&\;+\frac{1}{(2k+1)^2}||V(\xi^k\phi)||^2_{{Y}^{k,s-1}} +||V^\ast(\xi^k
u)||^2_{{X}^{k,s-1}}\,.
\end{split}
\end{equation}
To close the energy estimates, $s$ will be chosen  as $s= \lceil k \rceil +3$,
where $\lceil k \rceil$ is a
ceiling function, namely $\lceil k \rceil = \min \{n\in \mathbb{Z}:
k\leq n\}.$
We are now ready to state the main results.

\begin{theorem}\label{thm} Fix $k$, where $\frac{1}{2}<k<\infty$.
Suppose initial data $\phi_{0}$ and $u_0$ satisfy the following
conditions:
\[
\begin{split}
(\text{i})\;\mathcal{E}^{k,\lceil k \rceil +3}(\phi_0,u_0)<\infty;\;
(\text{ii})\;\frac{1}{C_0}\leq \frac{\phi_0}{\xi}\leq C_0\text{ for
some }C_0>1
\end{split}
\]
There exist a time $T>0$ only depending on $\mathcal{E}^{k,\lceil k
\rceil +3}(\phi_0,u_0)$ and $C_0$, and a unique solution $(\phi,u)$
to the reformulated Euler equations \eqref{VVk} with the boundary
conditions \eqref{BC} on the time interval $[0,T]$ satisfying
\begin{equation*}
\mathcal{E}^{k,\lceil k \rceil +3}(\phi,u)\leq 2\mathcal{E}^{k,\lceil k
\rceil +3}(\phi_0,u_0)\,,
\end{equation*}
and moreover, the vacuum boundary behavior of $\phi$ is preserved on
that time interval:
\begin{equation*}
\frac{1}{2C_0}\leq \frac{\phi}{\xi}\leq 2C_0\,.
\end{equation*}
\end{theorem}\

\begin{remark}
 The evolution of the vacuum boundary $x=a(t)$ is given by
$\dot{a}(t)=u(t,\xi=0)$. By Theorem \ref{thm},
one can derive that  $u_\xi(t,\xi)$ is bounded and
continuous in $(t,x)\in [0,T]\times [0,1]$, and since $u(t,0)=\int_1^0
u_\xi(t,\xi)d\xi$, we deduce that the vacuum interface is well-defined and
within short time $t\leq T$, the vacuum boundary
at time $t$ stays close to the
initial position with $|a(t)-a|\leq CT$ for some constant $C$
depending on the initial energy in  Theorem \ref{thm}.
\end{remark}

\begin{remark} Different constants in the energy functional \eqref{energy}
are due to the nonlinearity of \eqref{euler} or \eqref{VVk}.
The structure of the equations are
to be systematic after applying $V,V^\ast$ to \eqref{VVk} as in
\eqref{VVVk} and thereafter. Since we are dealing with the local
existence theory,  one may
 work with the energy functional
$\mathcal{E}^{k,s}=||\xi^k\phi||^2_{{X}^{k,s}} +||\xi^k u||^2_{{Y}^{k,s}}$,
which is
equivalent to \eqref{energy}.
\end{remark}

Our work is a fundamental step in understanding the long time
behavior of regular solutions and the vacuum boundary of Euler
equations with the physical singularity rigorously. With or without
damping case, it would be interesting to study whether our solution
exists globally in time. In particular, for the damping case, the
physical singularity is expected to be canonical as in the porus
media equation, it would be also interesting to investigate the
asymptotic relationship between our solution and regular solutions
to the porus media equation.

Parallel to the recent progress in free surface boundary problems with
geometry involved,
we expect that our result can be generalized to multidimensional case,
since the difficulty of the physical singularity lies in how
the solution behaves with respect to the normal direction to the boundary.
We leave this for future study.
The physical vacuum boundary also naturally appears
in Euler-Poisson equations for gaseous stars. It would be
very interesting to study the behavior of
solutions and the vacuum boundary under the influence of the gravitation.

The method of the proof is based on a careful study of $V,V^\ast$
operators and $V,V^\ast$ energy estimates. The first key
ingredient is to establish the
relationship between the multiplication with $\frac{1}{\xi}$, a
common operation embedded in
the equations \eqref{euler} or \eqref{VVk}, and $V,V^\ast$ by using
the underlying functional analytic properties of $V,V^\ast$ which can be
obtained from a thorough speculation on the behavior of  $V,V^\ast$ near
the vacuum boundary $\xi=0$. Another essential idea  is to find
the right form of spatial derivatives of $\partial_t\phi$, which is critical
to cope with strong nonlinearity, in particular of the second and
third order equations and to close the energy estimates in the end.
For large $k$, this can be done by introducing the representation formula
of $(V^\ast)^i(\xi^ku)$. In the similar vein,
due to the strong nonlinearity, the approximate
scheme starts from the third order equation and lower order terms including
$\phi$ and $u$ are recovered by integrals and boundary conditions. Lastly,
we point out that it is not trivial to build the well-posedness
of linear approximate systems and the duality argument
 is employed for that as in \cite{AM}.

The rest of the paper is organized as follows: In Section \ref{3}, we study
the operators $V$ and $V^\ast$ built in the reformulated Euler
equations \eqref{euler} or \eqref{VVk}.
In Section \ref{4}, we establish the a priori
 estimates in $V, V^\ast$ formulation.
 Based on the a priori estimates, in Section \ref{5}, we implement the
approximate scheme and prove that each approximate system is well-posed.
In Section \ref{6}, we finish the proof of Theorem \ref{thm}. In Section
\ref{7}, the duality argument is proved.

\section{Preliminaries}\label{3}

Throughout this section, we assume that $\phi$ is a given
nonnegative, smooth function of $t$ and $\xi$, and moreover,
$\frac{\phi}{\xi}$ and $\partial_\xi\phi$ are bounded from below and
above near $\xi=0$.

\subsection{Basic properties of $V$ and $V^\ast$}

In this subsection, we study the operators $V$ and $V^\ast$.
Let us denote $C_0^\infty((0,1)) $  (respectively  $C_0^\infty((0,1]) $)
the set of $C^\infty$ functions with compact support in
$(0,1)$ (respectively  $(0,1]$).

\begin{lemma}
\begin{equation} \label{density}
\begin{split}
  \overline{C_0^\infty((0,1])}^{X^{k,1}}  & = \mathcal{D}(V)=
  X^{k,1}\:;\;\;  \\
%   \overline{C_0^\infty((0,1))}^{X^{k,1}}  = \mathcal{D}(V^{\ast adj} );  \\
%  \overline{C_0^\infty((0,1])}^{Y^{k,1}}  & = \mathcal{D}(V^\ast)= Y^{k,1};
\;\, \overline{C_0^\infty((0,1))}^{Y^{k,1}} &  = \mathcal{D}(V^\ast)= Y^{k,1};
% = \mathcal{D}(V^{adj} ) ; \\
\end{split}
\end{equation}
 \end{lemma}

\begin{proof}
Take $f \in X^{k,1}$. Hence,
\begin{equation*}
\|f \|_{X^{k,1}}^2 =
 \int_0^1 \frac1{\xi^{2k}} |\partial_\xi (\frac{\phi^{2k}}{\xi^k}  f)|^2 +
 |f|^2  \ d\xi <  \infty.
\end{equation*}
We make the change of variable $y = \xi^{2k+1}$ and $F(y) = \frac{\phi^{2k}}{\xi^k} f $.
Hence
\begin{equation*}
\|f \|_{X^{k,1}}^2 =
 \int_0^1  (2k+1)|\partial_y F|^2 + \frac1{2k+1} \frac{\xi^{4k}}{\phi^{4k}}
  \frac1{y^{4k \over  2k + 1}} |F|^2 \ dy  <  \infty.
\end{equation*}
Since, $2k > 1$, we deduce that ${4k \over  2k + 1} > 1 $ and
hence, necessary $F(0) = 0$.
Applying Hardy inequality to $F$, we deduce that
 \begin{equation*}
\int_0^1  \frac{F^2}{y^2} dy \leq  C  \int_0^1 |\partial_y F|^2 .
\end{equation*}
Hence,  going back to the original variables, we deduce that for
$f \in X^{k,1}$, we have
\begin{equation} \label{Hardy}
 \int_0^1   \frac{f^2}{\xi^2} \  d\xi \leq  C \| \frac{\xi}{\phi}
 \|^{4k}_{L^\infty_\xi}
  \| V(f) \|_{L_\xi^2}^2\,.
\end{equation}

Now, consider a  cut-off function $\chi \in C^\infty (\mathbb{R})$ given by
$ \chi(\xi) = 0 $ for $\xi \leq  1/2 $ and $\chi (\xi ) = 1 $ for $\xi \geq 1$.
We also define $\chi_n (\xi) = \chi(n \xi)$.
For $f \in X^{k,1}$,
we define  $f_n(\xi)  = \chi_n( \xi) f (\xi) $.
Hence, $f_n \in  X^{k,1}$ and it is clear that $f_n$ goes to $f$ in $L^2$.
Moreover,
\begin{equation*}
 V(f -f_n) =
% \frac1{\xi^k} \partial_\xi ( \xi^k (f-f_n )   ) =
 (1-\chi_n) V(f) -  n \chi'(n \xi) \frac{\phi^{2k}}{\xi^{2k}}  f\,.
\end{equation*}
The first term  on the right hand side goes to zero in
$L^2_\xi$ when $n$ goes to infinity. For the second term, we use the
fact that
\begin{equation*}
\int_0^1   |n \chi'(n \xi)\frac{\phi^{2k}}{\xi^{2k}} f|^2 d\xi
\leq  \| \frac{\phi}{\xi}  \|_{L^\infty_\xi}^{4k}
 \int_0^1 \frac{f^2}{\xi^2} 1_{\frac1{2n} \leq \xi \leq
\frac1n } d\xi
\end{equation*}
which goes to zero when $n$ goes to infinity in view of \eqref{Hardy}.
Hence, we deduce that $f_n $ goes to $f$ in $X^{k,1}$. Now, it is clear that
$f_n$ can be approximated in $X^{k,1}$ by functions which are  in
$C_0^\infty((0,1])$.
Indeed, one can just convolute $f_n$  by some mollifier.
Hence,  the first equality of \eqref{density} holds.

To prove a similar result for $V^\ast$, we take $g \in Y^{k,1}$
and fix some  $c \in (0,1]$.
For $\xi \in (0,c]$  , we have
 \begin{equation*}
\begin{split}
|\frac{g(\xi)}{\xi^k} |  &= |- \int_\xi^c  \partial_\xi(\frac{g(\xi')}{\xi^{'k}} )
d\xi' + \frac{g(c)}{c^k}|  \\
  & \leq  \left(\int_\xi^c  \frac{\phi(\xi')^{4k}}{ |\xi'|^{2k }}
|\partial_\xi(\frac{g(\xi')}{\xi^{'k}} )|^2  d\xi'
  \int_\xi^c \frac{|\xi'|^{2k }}{ \phi(\xi')^{4k} }   d\xi' \right)^{1/2}
  + |\frac{g(c)}{c^k}| \\
 & \leq \  \eps  \xi^{1-2k \over 2} + |\frac{g(c)}{c^k}|
\end{split}
\end{equation*}
  where $\eps = C  \int_0^c \frac{\phi(\xi')^{4k}}{ |\xi'|^{2k }}
|\partial_\xi(\frac{g(\xi')}{\xi^{'k}} )|^2    d\xi'$. By choosing $c$ small enough,
we can make $\eps$ small. Hence, we deduce that for all $\eps > 0$, there exists
a constant $C_\eps(g) =  |\frac{g(c)}{c^k} | $ such that
for all $\xi \in (0,1]$, we have
\begin{equation*}
|g(\xi)| \leq \eps \  \sqrt{\xi} + C_\eps \xi^k.
\end{equation*}
We denote $g_n = \chi_n g$, hence it is clear that $g_n$ goes to
$g$ in $L_\xi^2$.  Moreover,
\begin{equation*}
 V^\ast(g -g_n) =
% \frac1{\xi^k} \partial_\xi ( \xi^k (f-f_n )   ) =
 (1-\chi_n) V^\ast (g) -  n \chi'(n \xi)\frac{\phi^{2k}}{\xi^{2k}} g.
\end{equation*}
The second term on the right hand side satisfies
\begin{equation*}
\int_0^1   |n \chi'(n \xi)\frac{\phi^{2k}}{\xi^{2k}} g|^2 d\xi \leq  C
\int_0^1 ( \frac{\eps}{\xi} +
 C_\eps \xi^{2k-1}) 1_{\frac1{2n} \leq \xi \leq
\frac1n } d\xi
\end{equation*}
Hence, since $\eps> 0$ is arbitrary,  it goes to zero when $n$ goes
to infinity. Therefore, we deduce that $g_n$ goes to $g$ in
$Y^{k,1}$. Now, it is clear that $g_n$ can be approximated by
functions which are  in $C_0^\infty((0,1])$ and the second  equality
of \eqref{density} follows.
\end{proof}

\begin{lemma}
(1) For $f\in
X^{k,1},\;g\in Y^{k,1}$
% with $f\cdot g|_{\xi=1}=0$
\[
\int V(f)\cdot g d\xi =\int f\cdot V^{\ast}(g) d\xi
\]

(2) $\ker V=\{0\}$ and $\ker V^\ast  = \{0\} $
%  is one-dimensional vector space
% generated by $\xi^k$.

(3) $V_t$ and $V_t^\ast$ are commutators of $V,V^\ast$ and
$\partial_t$:
\[
\partial_t V(f)=V(\partial_t f)+V_t(f),\;\;
\partial_t V^{\ast}(g)=V^{\ast}(\partial_t g)+V_t^{\ast}(g)
\]
where
\[
V_t(f) \equiv2k
\frac{1}{\xi^{k}}\partial_\xi
[\frac{\phi^{2k-1}\partial_t\phi}{\xi^{k}}f],\;\;V_t^{\ast}(g)
\equiv-2k
\frac{\phi^{2k-1}\partial_t\phi}{\xi^{k}}\partial_\xi[
\frac{1}{\xi^{k}} g]
\]
In addition,
\[
V_t^{\ast}(g)=
-2k\frac{\partial_t\phi}{\phi}V^\ast(g)
\text{ and }\;
VV_t^\ast (g)=V_tV^\ast (g).
\]
\end{lemma}

\begin{proof}
The proof of this lemma directly follows from  the density property proved in
the previous lemma. It will be used frequently
 during the energy estimates.
\end{proof}

The following lemma displays a key ingredient
of the main estimates. Dividing by $\xi$ is a common operation
embedded in the equations \eqref{euler} and \eqref{VVk}. The lemma
claims that the operation, which acts like derivatives when $\xi$
approach $0$, is completely controlled by modified derivatives $V$
and $V^\ast$.

\begin{lemma}\label{key}
(1) If  $f\in X^{k,1}$ and $g\in Y^{k,1}$, then $\frac{f}{\xi}$ and
$\frac{g}{\xi}$ are bounded in $L^2_\xi$ and we obtain the following
inequalities:
\begin{equation}\label{gxi}
\begin{split}
||\frac{f}{\xi}||_{L^2_\xi}\leq C
||\frac{\xi}{\phi}||_{L_\xi^\infty}^{2k} ||V
f||_{L^2_\xi},\;\;||\frac{g}{\xi}||_{L_\xi^2}\leq
 C \{||\frac{\xi}{\phi}||_{L_\xi^\infty}^{2k} ||V^\ast g||_{L_\xi^2}      \}
\end{split}
\end{equation}

(2) More generally, if $f\in L^2_\xi$  satisfies $\frac{V(f)}{\xi^{m-1}} \in L^2_\xi $
for some nonnegative real  number $m$, $m+k > \frac32$, then
\begin{equation*}
||\frac{f}{\xi^m}||_{L_\xi^2}\leq C
||\frac{\xi}{\phi}||_{L_\xi^\infty}^{2k} ||\frac{V
f}{\xi^{m-1}}||_{L_\xi^2}\,.
\end{equation*}

 (3) Also,  if  $g$ satisfies   $
\frac{g}{\xi^{m-1}} \in L^2_\xi $ and $\frac{V^\ast(g) }{\xi^{m-1}}
\in L^2_\xi $ for some nonnegative real number $m$,  $ m <  k + \frac12 $, then
\begin{equation} \label{gxim}
||\frac{g}{\xi^m}||_{L_\xi^2}\leq C
||\frac{\xi}{\phi}||_{L_\xi^\infty}^{2k} ||\frac{V^\ast
g}{\xi^{m-1}}||_{L_\xi^2} + \|  \frac{g}{\xi^{m-1}} \|_{L^2_\xi}\,.
\end{equation}
Here,  $m$ is not necessarily an integer. In practice, $m$ will be
chosen  as $\frac{1}{2}<m\leq k$.
\end{lemma}

 \begin{proof}
The point  (1)  was already proved for $f$  in the previous lemma by Hardy
inequality (see \eqref{Hardy}).
For the second inequality, consider first $g \in C_0^\infty((0,1))$, hence
\[
\int V^\ast g \cdot \frac{\xi^{2k}}
{\phi^{2k}} \frac{g }{\xi}d\xi=-\int
\partial_\xi(\frac{g}{\xi^{k}})\cdot
\xi^{k}\frac{g}{\xi}d\xi=\frac{2k-1}{2}
\int |\frac{g}{\xi}|^2d\xi + \frac12 |g(1)|^2\,.
\]
and $g(1) =0$. Since $k\neq \frac{1}{2}$ and
\[
\int V^\ast g \cdot \frac{\xi^{2k}}
{\phi^{2k}} \frac{g }{\xi}d\xi\leq
||\frac{\xi}{\phi}||_{L_\xi^\infty}^{2k}||V^\ast
g||_{L_\xi^2}||\frac{g}{\xi}||_{L_\xi^2},
\]
we obtain the desired result. The case where $g \in Y^{k,1}$
follows by density.

Now, we concentrate on (2).
For $\tilde \phi$ satisfying the same type of bounds as $\phi$, we define
$\tilde V_\alpha (f)  =
 \frac1{\xi^\alpha}  \partial_\xi ( \frac{\tilde
 \phi^{2\alpha}}{\xi^\alpha} f  ) .  $
For $f$ as in (2), we use that
\begin{equation*}
\frac{V(f)}{\xi^{m-1}} = \frac1{\xi^{k+m-1}}\partial_\xi
( \frac{\phi^{2k}}{\xi^{k-m+1}}
\frac{f}{\xi^{m-1}}  ) = \tilde V_{k+m-1} ( \frac{f}{\xi^{m-1}} )
\end{equation*}
where $\tilde \phi $ is given by $\tilde \phi^{2(k+m-1)} = \phi^{2k}
\xi^{2(m-1)} $. Since, $k+m-1> \frac12$,  we can apply the estimate
\eqref{Hardy} for $V$ replaced by $\tilde  V_{k+m-1}$ and $f$
replaced by $\frac{f}{\xi^{m-1}}  $. This gives the desired bound.

For (3), we write
\begin{equation*}
\frac{V^\ast(g)}{\xi^{m-1}} = - \frac{\phi^{2k}}{\xi^{2k}}
\tilde V_{m-1-k} (\frac{g}{\xi^{m-1}})
\end{equation*}
with $\tilde \phi =\xi$. Since, $m-1-k < -\frac12$, $\tilde V_{m-1-k} $
satisfies the same estimates as $V^\ast$, in particular  the second
estimate of \eqref{gxi} holds for $\tilde V_{m-1-k} $, hence
\eqref{gxim} holds.
\end{proof}

\begin{remark} We note that the boundary conditions on $f$ and $g$  at $\xi=0$
in Lemma \ref{key} are imbedded in $X^{k,s}$ and $Y^{k,1}$, namely $L_\xi^2$
boundedness of $Vf$ and $V^\ast g $ forces  $f$ and $g$  to vanish at $\xi=0$.
Actually, it forces them to be less than $\xi^{1/2}$.
\end{remark}

As a direct result of Lemma \ref{key}, we obtain the following
$L_\xi^2$ estimates for $\frac{f}{\xi^m}$ and $\frac{g}{\xi^m}$:

\begin{corollary}\label{cor}  For given nonnegative real number $m$, $m<k+\frac32$,
if $f\in X^{k,\lceil m\rceil}$, then there exists $C_1$ only
depending on $||\frac{\xi}{\phi}||_{L_\xi^\infty}$ so that
\[
||\frac{f}{\xi^m}||_{L_\xi^2}\leq C_1 ||f||_{X^{k, \lceil
m\rceil}}\,.
\]
Also,  for given nonnegative real number $m$, $m<k+\frac12$, if $g\in
Y^{k,\lceil m\rceil}$, then there exists $C_2$ only depending on
$||\frac{\xi}{\phi}||_{L_\xi^\infty}$ so that
\[
||\frac{g}{\xi^m}||_{L_\xi^2}\leq C_2 ||g||_{Y^{k, \lceil
m\rceil}}\,.
\]
\end{corollary}

\begin{proof}
 We apply (2) and (3) of Lemma \ref{key}  alternatingly
 until negative powers of $\xi$ disappear. Note that the
 conditions on $m$ come from the one in (3) of Lemma \ref{key}.
\end{proof}

We next prove the
 Sobolev imbedding inequalities of $V,V^\ast$ version, which will be
useful tools to control nonlinear terms.

\begin{lemma}\label{sup}
If $f\in X^{k, \lceil k\rceil +1}$ and $g\in Y^{k, \lceil k\rceil
+1}$,
 then there exist constants $C_3$ and  $C_4$
 only depending on
$||\frac{\xi}{\phi}||_{L_\xi^\infty}$ so that
\begin{equation*}
||\frac{f}{\xi^k}||_{L_\xi^\infty}\leq C_3 ||f||_{X^{k, \lceil
k\rceil +1}}\text{ and }||\frac{ g}{\xi^k}||_{L_\xi^\infty}\leq C_4
||g||_{Y^{k, \lceil k\rceil +1}}\,.
\end{equation*}
\end{lemma}

\begin{proof} We start with $g$ part.
By the definition of $V$ and $V^\ast$, one finds that
\[
\partial_\xi[ \frac{g}{\xi^k} ] =  -
\frac{\xi^{2k}}{\phi^{2k}} \frac{V^\ast g}{\xi^k}
\]
Thus, by Sobolev embedding theorem in one dimension,  it suffices to
show that
\[
 \frac{g}{\xi^k},\;\;\frac{V^\ast g}{\xi^k}\in L_\xi^2\,.
\]
This follows from Corollary \ref{cor}. Hence, we conclude that
$L^\infty_\xi$ bound of $\frac{g}{\xi^k}$ is controlled by
$\|g\|_{Y^{k, \lceil k\rceil +1}}$. For $f$ part, we show that
$\|\frac{\phi^{2k}}{\xi^{2k}}\frac{f}{\xi^k}\|_{L^\infty_\xi}$ is
bounded by $\|f\|_{X^{k, \lceil k\rceil +1}}$. Note that
\[
 \partial_\xi[\frac{\phi^{2k}}{\xi^{2k}}\frac{f}{\xi^k} ] =
\frac{Vf}{\xi^k}-2k\frac{\phi^{2k}}{\xi^{2k}}\frac{f}{\xi^{k+1}}\,.
\]
Thus,  by applying Corollary \ref{cor},  we obtain the desired
conclusion.
\end{proof}

More generally, we obtain the following:

\begin{lemma}\label{infty}  Let $0\leq j <  k-\frac12$ be a given
nonnegative number. If $f\in X^{k,\lceil j\rceil+1}$ and $g\in Y^{k,
\lceil j\rceil+1}$, then there exist constants $C_5$ and $C_6$
 only depending on
$||\frac{\xi}{\phi}||_{L_\xi^\infty}$ so that
\begin{equation*}
||\frac{f}{\xi^j}||_{L_\xi^\infty}\leq C_5 ||f||_{X^{k,\lceil
j\rceil+1}}\text{ and }||\frac{ g}{\xi^j}||_{L_\xi^\infty}\leq C_6
||g||_{Y^{k, \lceil j\rceil+1}}\,.
\end{equation*}
\end{lemma}

\begin{proof} We only treat $\frac{g}{\xi^j}$. Note that
\[
\partial_\xi[\frac{g}{\xi^j}]=-\frac{\xi^{2k}}{\phi^{2k}}
\frac{V^\ast g}{\xi^j}+(k-j)\frac{g}{\xi^{j+1}}
\]
Hence, by the Sobolev embedding theorem, it suffices to show that
$\frac{g}{\xi^j}  $,    $ \frac{V^\ast g}{\xi^j}\text{ and
}\frac{g}{\xi^{j+1}}$ are in $ L_\xi^2$. This  follows from
Corollary  \ref{cor}. Note that $j$ has to be less than $k-\frac12$.
\end{proof}

Next we present the product rule for the operators $V,V^\ast$.

\begin{lemma}  \label{product rule}
Let $f,\,g\in \mathcal{D} V\cap\mathcal{D} V^\ast$ be given. Let $h$ be a
given smooth function. The following identities hold:
\[
\begin{split}
\bullet&\;V^\ast
f=-Vf+2k\frac{\phi^{2k}}{\xi^{2k}}\partial_\xi\phi\frac{f}{\phi}
=-Vf+\frac{2k}{2k+1}\frac{V(\xi^k \phi)}{\xi^k}\frac{f}{\phi}\\
\bullet&\;Vg=-V^\ast
g+2k\frac{\phi^{2k}}{\xi^{2k}}\partial_\xi\phi\frac{g}{\phi}
=-V^\ast g+\frac{2k}{2k+1}\frac{V(\xi^k \phi)}{\xi^k}\frac{g}{\phi}\\
\bullet&\; V(fh)=V(f)h+ f\frac{\phi^{2k}}{\xi^{2k}}\partial_\xi
h,\;\;
V^\ast (gh)=V^\ast (g)h- g\frac{\phi^{2k}}{\xi^{2k}}\partial_\xi h\\
\bullet&\;\frac{\phi^{2k}}{\xi^{2k}}\partial_\xi
h=Vh+k\frac{\phi^{2k}}{\xi^{2k}}
(\frac{1}{\xi}-2\frac{\partial_\xi\phi}{\phi}) {h}=-V^\ast
h+k\frac{\phi^{2k}}{\xi^{2k}}
\frac{h}{\xi}\\
\bullet&\;\frac{\phi^{2k}}{\xi^{2k}}\partial_\xi
[fg]=V(f)g+\frac{\phi^{4k}}{\xi^{3k}} f\partial_\xi[\frac{\xi^k
}{\phi^{2k}}g] =V(f) g-f V^\ast (g)+
2k\frac{\phi^{2k}}{\xi^{2k}}(\frac{1}{\xi}-\frac{\partial_\xi\phi}{\phi})fg
\end{split}
\]
\end{lemma}

The lemma  tells that when $V$ or $V^\ast$ act on a function $h$,
depending on that function, they can yield $\frac{h}{\xi}$ or $
\frac{V(\xi^k\phi)}{\xi^k}\frac{h}{\phi}$ besides $Vh$ or $V^\ast
h$.

\subsection{Homogeneous operators $\overline{V},\overline{V}^\ast$}

Next we introduce homogeneous linear operators $\overline{V}$ and
$\overline{V}^\ast$ of $V$ and $V^\ast$ as follows:
\begin{equation}\label{ho}
\begin{split}
\overline{V}(f)\equiv \frac{1}{\xi^{k}}\partial_\xi [\xi^kf],\;\;\;
\overline{V}^\ast(g)\equiv-\xi^{k}\partial_\xi[ \frac{g}{\xi^{k}}]
\ \hbox{and} \ g(\xi=1) = 0.
\end{split}
\end{equation}
These homogeneous operators are special cases of $V$ and $V^\ast$
for which $\phi$ is simply taken as $\xi$. Function spaces
${\overline{X}}^{k,s}$ and ${\overline{Y}}^{k,s}$ for
$s$ a given nonnegative integer are given as follows:
\begin{equation}\label{XY1}
\begin{split}
{\overline{X}}^{k,s}&\equiv \{f\in L_\xi^2:
(\overline{V})^i(f)\in L_\xi^2,\; 0\leq i\leq s\}\\
{\overline{Y}}^{k,s}&\equiv \{g\in L_\xi^2: (\overline{V}^\ast)^i
(g)\in L_\xi^2,\;0\leq i\leq s\}
\end{split}
\end{equation}
where $ (\overline{V})^i$ and $(\overline{V}^\ast)^i $ are defined as
in \eqref{Vi} and \eqref{V*i}. These  spaces are
equipped with the following norms
\[
||f||^2_{{\overline{X}}^{k,s}}\equiv \sum_{i=0}^s
||(\overline{V})^i(f)||^2_{L^2_\xi}\;\text{ and
}\;||g||^2_{{\overline{Y}}^{k,s}}\equiv \sum_{i=0}^s
||(\overline{V}^\ast)^i(g)||^2_{L^2_\xi}\,.
\]
Indeed,  $\overline{V}$, $\overline{V}^\ast$ and
 $V$, $V^\ast$ share many good properties;
 we summarize the
analog of Lemma \ref{key}, \ref{sup}, \ref{infty} and
\ref{product rule} for $\overline{V}$, $\overline{V}^\ast$ in the
following.

\begin{lemma}\label{linear-prop} (1) If $f\in L^2_\xi$ satisfies
$\frac{\overline{V}f}{\xi^{m-1}}\in L^2_\xi$ for some real number $m$,
$m+k>\frac32$, then
\[
 ||\frac{f}{\xi^m}||_{L_\xi^2}^2\leq ||\frac{\overline{V}f}{\xi^{m-1}}
||_{L_\xi^2}^2
\]
Also, if $g$ satisfies $\frac{ g}{\xi^{m-1}}\in L^2_\xi$ and
 $\frac{\overline{V}^\ast g}{\xi^{m-1}}\in L^2_\xi$ for some real number
 $m$, $m<k+\frac12$, then
\[
||\frac{g}{\xi^m}||_{L_\xi^2}^2\leq
||\frac{\overline{V}^\ast g}{\xi^{m-1}}
||_{L_\xi^2}^2+||\frac{g}{\xi^{m-1}}||_{L_\xi^2}^2\,.
\]

(2) If $f\in \overline{X}^{k, \lceil k\rceil +1}$ and $g\in
\overline{Y}^{k, \lceil k\rceil +1}$,
 then we obtain
\begin{equation*}
||\frac{f}{\xi^k}||_{L_\xi^\infty}\leq C ||f||_{\overline{X}^{k,
\lceil k\rceil +1}}\text{ and }||\frac{
g}{\xi^k}||_{L_\xi^\infty}\leq C ||g||_{\overline{Y}^{k, \lceil
k\rceil +1}}\,.
\end{equation*}

(3)  Let $0\leq j < k-\frac12$ be a given nonnegative number. If
$f\in \overline{X}^{k,\lceil j\rceil+1}$ and $g\in \overline{Y}^{k,
\lceil j\rceil+1}$, then we obtain
\begin{equation*}
||\frac{f}{\xi^j}||_{L_\xi^\infty}\leq C
||f||_{\overline{X}^{k,\lceil j\rceil+1}}\text{ and
}||\frac{g}{\xi^j}||_{L_\xi^\infty}\leq C ||g||_{\overline{Y}^{k,
\lceil j\rceil+1}}\,.
\end{equation*}

(4) Product rule for $\overline{V}$, $\overline{V}^\ast$:
\[
  \overline{V}^\ast f=-\overline{V}f+2k\frac{f}{\xi},\;
\overline{V}(fh)=\overline{V}(f) h+f\partial_\xi h,\;
\overline{V}^\ast(gh)=\overline{V}^\ast(g) h-g\partial_\xi h.
\]
\end{lemma}

Now we show that two norms induced by $V,V^\ast$ and
$\overline{V},\overline{V}^\ast$ are equivalent.

\begin{proposition}\label{equiv e} Let $\phi$ be given so that
$||\xi^k\phi||_{X^{k,\lceil k\rceil +2}} \leq A$ and
$\frac{1}{C}<\frac{\xi}{\phi}<C$ for positive constants $A,\,C$.
 Then for any $f\in X^{k,\lceil
k\rceil }
 $ and $g\in Y^{k,\lceil k\rceil}$, there exists a constant
$M$  depending only on $A,\,C$ so that
\begin{equation}\label{equiv}
\begin{split}
\frac{1}{M}|| f||_{\overline{X}^{k,\lceil k\rceil }}^2\leq ||
f||_{X^{k,\lceil k\rceil }}^2\leq M || f||_{\overline{X}^{k,\lceil
k\rceil }}^2\,,\;
\frac{1}{M}|| g||_{\overline{Y}^{k,\lceil k\rceil }}^2\leq ||
g||_{Y^{k,\lceil k\rceil }}^2\leq M|| g||_{\overline{Y}^{k,\lceil
k\rceil }}^2\,.
\end{split}
\end{equation}
\end{proposition}

\begin{proof} We will only prove the second inequality ($\leq$). The other
one can be shown in the same way.
 Let us start with $Vf$ and $V^\ast g$.
\[
\begin{split}
&Vf = \frac{1}{\xi^k}\partial_\xi[\frac{\phi^{2k}}{\xi^k}f]=
\frac{\phi^{2k}}{\xi^{2k}}\overline{V}f+\partial_\xi
[\frac{\phi^{2k}}{\xi^{2k}}]f=\frac{\phi^{2k}}{\xi^{2k}}\overline{V}f+
\{\frac{2k}{2k+1}\frac{V(\xi^k\phi)}{\xi^k\phi}-2k\frac{\phi^{2k}}{\xi^{2k+1}} \}f\\
& V^\ast g=-\frac{\phi^{2k}}{\xi^k}\partial_\xi[\frac{g}{\xi^k}]
=\frac{\phi^{2k}}{\xi^{2k}}\overline{V}^\ast g
\end{split}
\]
Next, note that by the product rule  of  Lemma \ref{product
rule} and  \ref{linear-prop}, the higher order terms $(V)^if$ and
$(V^\ast)^ig$ can be expanded into the following form in terms of
$(\overline{V})^jf$ and $(\overline{V}^\ast)^jg$ for $j\leq i$:
\[
(V)^if=\sum_{j=0}^i
\Psi^1_j(\xi^k\phi)\cdot(\overline{V})^{i-j}f,\;\;
(V^\ast)^ig=\sum_{j=0}^{i-1}
\Psi^2_j(\xi^k\phi)\cdot(\overline{V}^\ast)^{i-j}g
\]
where for $s=1$ or $2$
\[
\Psi_j^s(\xi^k\phi)\equiv\sum_{r=0}^j \{C_{rs}\frac{1}{\xi^r}\cdot
 \sum_{\substack{l_1+\cdots+l_p=j-r\\l_1,\dots,
l_p\geq 1}}{C_{l_1\cdots l_ps}} \prod_{q=1}^{p}
\frac{(V)^{l_q}(\xi^k\phi)}{\xi^k\phi}\}
\]
for some functions $C_{rs}$, $C_{l_1\cdots l_ps}$ which may only
depend on $k$, $\frac{\phi}{\xi}$ and $\frac{\xi}{\phi}$ and
therefore $C_{rs}$, $C_{l_1\cdots l_ps}$ are bounded by some power
function of $C$. In order to show \eqref{equiv}, first we rewrite
$(V)^if$ and $(V^\ast)^ig$ as
\[
(V)^if=\sum_{j=0}^i
\xi^j\Psi^1_j(\xi^k\phi)\cdot\frac{(\overline{V})^{i-j}f}{\xi^j},\;\;
(V^\ast)^ig=\sum_{j=0}^{i-1}
\xi^j\Psi^2_j(\xi^k\phi)\cdot\frac{(\overline{V}^\ast)^{i-j}g}{\xi^j}
\]
Since $||\frac{(\overline{V})^{i-j}f}{\xi^j}||_{L^2_\xi}$ and
$||\frac{(\overline{V}^\ast)^{i-j}g}{\xi^j}||_{L^2_\xi}$ are bounded
by $||f||_{\overline{X}^{k,j}}$ and $||g||_{\overline{Y}^{k,j}}$
respectively and moreover, by Lemma \ref{infty},
$||\xi^j\Psi^s_j||_{L^\infty_\xi}$ is bounded by
$||\xi^k\phi||_{X^{k,j+2}}$ for $j\leq \lceil k\rceil$, by adding
all the inequalities over $i\leq\lceil k\rceil$ we obtain the
desired inequality as well as the desired bound $M$ as a function of $A$ and $C$.
\end{proof}

The above equivalence dictates the linear character of higher order
energy. Note that we cannot deduce the same result for the full
energy $\mathcal{E}^{k,\lceil k\rceil +3}$ due to the nonlinearity.

\section{$V,V^\ast$ a priori energy estimates}\label{4}

This section is devoted to  $V,V^\ast$ energy estimates to get the
 following a
priori estimates, a key to construct strong solutions.

\begin{proposition}\label{apriori}
Suppose $\phi$ and
$u$ solve \eqref{VVk} and \eqref{BC}
 with $\mathcal{E}^{k,\lceil k\rceil+3}(\phi,u)
<\infty$.
 If we further assume that
\begin{equation}
\frac{1}{C}\leq \frac{\phi}{\xi}\leq C\text{ for some
}C>1\,,\label{AA}
\end{equation}
we obtain the following a priori estimates:
\begin{equation*}
\frac{d}{dt} \mathcal{E}^{k,\lceil k\rceil+3}(\phi,u)\leq
\mathcal{C}(\mathcal{E}^{k,\lceil k\rceil+3}(\phi,u))
\end{equation*}
where $\mathcal{C}(\mathcal{E}^{k,\lceil k\rceil+3}(\phi,u))$ is a continuous
function of $\mathcal{E}^{k,\lceil k\rceil+3}(\phi,u)$ and $C$. Moreover, the a priori
assumption \eqref{AA} can be justified:  the boundedness of
$\mathcal{E}^{k,\lceil k \rceil +3}(\phi,u)$ imply the boundedness
of $\frac{\phi}{\xi}$.
\end{proposition}

 In order to
illustrate the idea of of the proof, we start with the
simplest case when $k=1$ of which corresponding $\gamma$ is 3. In the
next subsections, we generalize it to arbitrary $k>\frac{1}{2}$.

\subsection{A priori estimates for $k=1$ $(\gamma=3)$}

When $k=1$, the Euler equations read as follows:
\begin{equation}
 \begin{split}
  \phi_t+(\frac{\phi}{\xi})^{2}u_{\xi}&=0\\
u_t+(\frac{\phi}{\xi})^{2}\phi_{\xi}&=0\label{euler1}
 \end{split}
\end{equation}
The operators $V$ and $V^{\ast}$ take the form:
\begin{equation*}
\begin{split}
V(f)\equiv \frac{1}{\xi}\partial_\xi (\frac{\phi^{2}}{\xi}f),\;\;\;
V^{\ast}(g)\equiv-\frac{\phi^{2}}{\xi}\partial_\xi( \frac{1}{\xi} g)
\end{split}
\end{equation*}
In terms of $V$ and $V^{\ast}$, (\ref{euler1}) can be rewritten as
follows:
\begin{equation}
\begin{split}
&\partial_t(\xi\phi)
-V^{\ast}(\xi u)=0\\
&\partial_t(\xi u) +\frac{1}{3}V(\xi\phi)=0 \label{VV}
\end{split}
\end{equation}
The energy functional $\mathcal{E}^{1,4}(\phi,u)$ reads as the following:
\begin{equation}
 \mathcal{E}^{1,4}(\phi,u)\equiv \int\frac{1}{3}|\xi\phi|^2
+|\xi u|^2d\xi+
\sum_{i=1}^4\int\frac{1}{9}|(V)^i(\xi\phi)|^2
+|(V^\ast)^i(\xi u)|^2d\xi \label{ef}
\end{equation}
Before carrying out the energy estimates, we verify the assumption
\eqref{AA}. In order to do so, we examine each term in the energy
functional \eqref{ef}. Let us start with $V$, $V^{\ast}V$,
$VV^{\ast}V$, $V^\ast VV^{\ast}V$ of $\xi \phi$.
\begin{align*}
\begin{split}
&\bullet\; V(\xi\phi)= \frac{1}{\xi}\partial_\xi [\phi^{3}]=3
\frac{\phi^{2}}
{\xi}\partial_\xi\phi\\
 &\bullet\;
V^{\ast}V(\xi\phi)=-3 \frac{\phi^2} {\xi}\partial_\xi
[\frac{\phi^2}{\xi^2}\partial_\xi\phi] \\
&\bullet\; VV^{\ast}V(\xi\phi)=- \frac{3}{\xi}
\partial_\xi[
\frac{\phi^{4}} {\xi^{2}}\partial_\xi
[\frac{\phi^2}{\xi^{2}}\partial_\xi \phi]]\\
 &\bullet\;
V^{\ast}VV^{\ast}V(\xi\phi)= 3\frac{\phi^{2}} {\xi}\partial_\xi[
\frac{1}{\xi^2}
\partial_\xi[
\frac{\phi^{4}} {\xi^{2}}\partial_\xi
[\frac{\phi^2}{\xi^{2}}\partial_\xi\phi]]]
\end{split}
\end{align*}
One advantage of $V$ and $V^\ast$ formulation is that $L_\xi^\infty$
control of $\phi$ is cheap to get: since
\[
\int \xi\phi\cdot V(\xi\phi) d\xi=\int
\phi\partial_\xi[\phi^{3}]d\xi= \frac{3}{4}\phi^{4},
\]
we deduce that
\begin{equation}
||\phi||_{L_\xi^\infty}^4\leq ||\xi\phi||_{L_\xi^2}^2+
||V(\xi\phi)||_{L_\xi^2}^2\leq \mathcal{E}^{1,4}(\phi,u)\label{supphi}
\end{equation}
The boundedness of $\partial_\xi\phi$ also follows from the boundedness of
$\mathcal{E}^{1,4}(\phi,u)$:  first note that
\[
\frac{\phi^2}{\xi^2}\partial_\xi\phi=\frac{1}{3}\frac{V(\xi\phi)}{\xi}
\]
and then by applying Lemma \ref{sup} to $g=V(\xi\phi)$ when $k=1$,
we deduce that  $\frac{\phi^2}{\xi^2}\partial_\xi\phi$ is bounded
and continuous, and in turn $\partial_\xi\phi$ is in $L_\xi^\infty$
under the assumption \eqref{AA}. In the same way, by applying Lemma
\ref{sup} to  $f=V^\ast V(\xi\phi)$ when $k=1$, we can derive that
$\partial_\xi[\frac{\phi^2}{\xi^2}\partial_\xi\phi]$ is bounded.
However, we remark that this does not imply that
$\partial_\xi^2\phi$ is bounded in our energy space, since it is not
clear how to control $\partial_\xi[\frac{\phi}{\xi}]$. Thus we keep
the form as they are rather than try to go back to the standard
Sobolev space.

Next we turn to $u$ variable.  We list out $V^\ast$, $VV^{\ast}$,
$V^\ast VV^{\ast}$, $VV^\ast VV^{\ast}$ of $\xi u$.
\begin{align}
\begin{split}
&\bullet\; V^\ast (\xi u)=-\frac{\phi^{2}}
{\xi}\partial_\xi u\\
%;\;\; V_t^\ast(\xi u)=-2\frac{\phi\partial_t\phi}
%{\xi}\partial_\xi u\\
 &\bullet\;
VV^{\ast}(\xi u)=-\frac{1}{\xi}\partial_\xi[\frac{\phi^4}{\xi^2}
\partial_\xi u]\\
&\bullet\; V^{\ast}VV^\ast (\xi u)=\frac{\phi^2}{\xi}\partial_\xi
[\frac{1}{\xi^2}\partial_\xi[\frac{\phi^4}{\xi^2}
\partial_\xi u]]\\
&\bullet\; VV^{\ast}VV^\ast (\xi
u)=\frac{1}{\xi}\partial_\xi[\frac{\phi^4}{\xi^2}\partial_\xi
[\frac{1}{\xi^2}\partial_\xi[\frac{\phi^4}{\xi^2}
\partial_\xi u]]]
\end{split}\label{u}
\end{align}
Note that $\frac{\partial_t\phi}{\xi}$ can be estimated in terms of
$\partial_\xi u$ via the equation \eqref{euler1}:
\begin{equation*}
 {\partial_t\phi}=-\frac{\phi^{2}}
{\xi^2}\partial_\xi u=\frac{V^\ast (\xi u)}{\xi}
\end{equation*}
 Apply Lemma \ref{sup}  to deduce that
 $\frac{VV^\ast(\xi
 u)}{\xi}=\frac{1}{\xi^2}\partial_\xi[\phi^2\partial_t\phi]$
is bounded and continuous if $\mathcal{E}^{1,4}(\phi,u)$ is bounded.
Letting $h$ be $\frac{1}{\xi^2}\partial_\xi[\phi^2\partial_t\phi]$,
we can write $\phi^2{\partial_t\phi}=\int_0^\xi\xi'^2h d\xi$, and
therefore we conclude $\frac{\partial_t\phi}{\xi}$ is also bounded by
$\mathcal{E}^{1,4}(\phi,u)$ and $C$.

Writing $\frac{\phi}{\xi}$ as
\[
 \frac{\phi}{\xi}(t,\xi)=\frac{\phi}{\xi}(0,\xi)+\int_0^t
\frac{\partial_t\phi}{\xi}(\tau,\xi) d\tau\,,
\]
we conclude that for a short time, the boundary behavior of
 $\frac{\phi}{\xi}$ is preserved and
 in particular, this justifies  the assumption \eqref{AA}.

 We now perform the energy estimates. The zeroth order energy energy
is conserved as given by \eqref{0}.
 Apply $V$ and $V^{\ast}$ to \eqref{VV} and use  $V_t(\xi\phi)=\frac{2}{3}
\partial_tV(\xi\phi)$
\begin{equation}
\begin{split}
&\partial_tV(\xi\phi)
-3VV^{\ast}(\xi u)=0\\
&\partial_tV^{\ast}(\xi u) +\frac{1}{3}V^{\ast}V(\xi
\phi)=V_t^{\ast}(\xi u) \label{VVV}
\end{split}
\end{equation}
Multiply by $\frac{1}{9}V(\xi\phi)$ and $V^{\ast}(\xi u)$ and
integrate to get
\begin{equation}
\frac{1}{2}\frac{d}{dt}\int\frac{1}{9} |V(\xi\phi)|^2 +|V^{\ast}(\xi
u)|^2d\xi=\int V_t^{\ast}(\xi u) \cdot V^{\ast}(\xi u)d\xi\label{e1}
\end{equation}
Apply $V^{\ast}$ and $V$ to \eqref{VVV} to get
\begin{equation}\label{V2}
\begin{split}
&\partial_tV^{\ast}V(\xi\phi) -3V^{\ast}VV^{\ast}(\xi u)=V_t^{\ast}
V(\xi\phi)\\
&\partial_tVV^{\ast}(\xi u) +\frac{1}{3}VV^{\ast}V(\xi
\phi)=VV_t^{\ast}(\xi u)+V_t V^{\ast}(\xi u)=2V_t V^{\ast}(\xi u)
\end{split}
\end{equation}
Multiply by $\frac{1}{9}V^\ast V(\xi\phi)$ and $VV^{\ast}(\xi u)$
and integrate to get
\begin{equation}
\begin{split}
\frac{1}{2}\frac{d}{dt}\int\frac{1}{9} |V^\ast V(\xi\phi)|^2
+|VV^{\ast}(\xi u)|^2d\xi=\int \frac{1}{9} V_t^{\ast} V(\xi\phi)
\cdot V^{\ast}V(\xi\phi)d\xi\\+\int 2V_t V^{\ast}(\xi u) \cdot
VV^{\ast}(\xi u)d\xi\label{e2}
\end{split}
\end{equation}
Apply $V^{\ast}$ and $V$ to \eqref{V2} to get
\begin{equation}\label{V3}
\begin{split}
&\partial_tVV^{\ast}V(\xi\phi) -3VV^{\ast}VV^{\ast}(\xi
u)=VV_t^{\ast}
V(\xi\phi)+V_tV^\ast V(\xi\phi)=2V_tV^\ast V(\xi\phi)\\
&\partial_tV^{\ast}VV^{\ast}(\xi u)
+\frac{1}{3}V^{\ast}VV^{\ast}V(\xi \phi)=2V^{\ast}V_t V^{\ast}(\xi
u)+V_t^\ast VV^\ast (\xi u)
\end{split}
\end{equation}
Multiply by $\frac{1}{9}VV^\ast V(\xi\phi)$ and $V^\ast
VV^{\ast}(\xi u)$ and integrate to get
\begin{equation}
\begin{split}
\frac{1}{2}\frac{d}{dt}\int\frac{1}{9} |VV^\ast V(\xi\phi)|^2
+|V^\ast VV^{\ast}(\xi
u)|^2d\xi=\int \frac{2}{9} V_tV^\ast V(\xi\phi) \cdot VV^{\ast}V(\xi\phi)d\xi\\
+\int 2V^{\ast}V_t V^{\ast}(\xi u) \cdot V^{\ast}VV^{\ast}(\xi
u)d\xi+\int V_t^\ast VV^\ast (\xi u)\cdot V^{\ast}VV^{\ast}(\xi
u)d\xi
\end{split}\label{e3}
\end{equation}
Apply $V^{\ast}$ and $V$ to \eqref{V3} to get
\begin{equation}\label{V4}
\begin{split}
&\partial_tV^\ast VV^{\ast}V(\xi\phi) -3V^\ast
VV^{\ast}VV^{\ast}(\xi u)=
2V^\ast V_tV^\ast V(\xi\phi)+V_t^\ast VV^{\ast}V(\xi\phi)\\
&\partial_tVV^{\ast}VV^{\ast}(\xi u)
+\frac{1}{3}VV^{\ast}VV^{\ast}V(\xi \phi)=2VV^{\ast}V_t V^{\ast}(\xi
u)+2V_tV^\ast VV^\ast (\xi u)
\end{split}
\end{equation}
Multiply by $\frac{1}{9}V^\ast VV^\ast V(\xi\phi)$ and $VV^\ast
VV^{\ast}(\xi u)$ and integrate to get
\begin{equation}
\begin{split}
&\frac{1}{2}\frac{d}{dt}\int\frac{1}{9} |V^\ast VV^\ast
V(\xi\phi)|^2
+|VV^\ast VV^{\ast}(\xi u)|^2d\xi\\
&=\int \frac{2}{9} V^\ast V_tV^\ast V(\xi\phi) \cdot V^\ast
VV^{\ast}V(\xi\phi)d\xi +\int \frac{1}{9}V_t^\ast
VV^{\ast}V(\xi\phi)\cdot V^\ast VV^{\ast}V(\xi\phi)d\xi\\&\;\;+\int
2VV^{\ast}V_t V^{\ast}(\xi u)\cdot VV^{\ast}VV^{\ast}(\xi u)d\xi
+\int 2V_tV^\ast VV^\ast (\xi u) \cdot VV^{\ast}VV^{\ast}(\xi u)d\xi
\end{split}\label{e4}
\end{equation}
In order to prove Proposition \ref{apriori}, it now remains to
estimate the following nonlinear terms coming from the energy
estimates in terms of the energy functional \eqref{ef}:
\begin{equation}
\begin{split}
\int 2V_t V^{\ast}(\xi u) \cdot VV^{\ast}(\xi u)d\xi,\;
\int \frac{2}{9} V_tV^\ast V(\xi\phi) \cdot VV^{\ast}V(\xi\phi)d\xi,\\
\int 2V^{\ast}V_t V^{\ast}(\xi u) \cdot V^{\ast}VV^{\ast}(\xi
u)d\xi,\;\int \frac{2}{9} V^\ast V_tV^\ast V(\xi\phi) \cdot V^\ast
VV^{\ast}V(\xi\phi)d\xi, \\
\int 2VV^{\ast}V_t V^{\ast}(\xi u)\cdot VV^{\ast}VV^{\ast}(\xi
u)d\xi, \;\int 2V_tV^\ast VV^\ast (\xi u) \cdot
VV^{\ast}VV^{\ast}(\xi u)d\xi
\end{split}\label{mixed}
\end{equation}
Our goal is to control these terms by our energy functional
\eqref{ef}. All of them contain $\partial_t\phi$ and its derivatives
with suitable weights. The estimates of $\partial_t\phi$ related
terms can be obtained through the equation \eqref{euler1} by
estimating $V^\ast,\;VV^\ast,\; V^\ast VV^\ast,\; VV^\ast VV^\ast$
of $\xi u$ in terms of $\partial_t\phi$. Before going any further,
let us try to get the better understanding of them. First,  we
rewrite $V^\ast,\;VV^\ast,\; V^\ast VV^\ast,\; VV^\ast VV^\ast$ of
$\xi u$, namely \eqref{u}, in terms of $\partial_t\phi$:
\begin{equation}
\begin{split}
&\bullet\; V^\ast (\xi u)=-\frac{\phi^2}{\xi}\partial_\xi u=\xi \partial_t\phi\\
%;\;\; V_t^\ast(\xi u)=-2\frac{\phi\partial_t\phi}
%{\xi}\partial_\xi u\\
 &\bullet\;
VV^{\ast}(\xi u)=\frac{1}{\xi}\partial_\xi[\phi^2\partial_t
\phi]=\frac{\phi^3}{\xi}\partial_\xi[\frac{\partial_t\phi}{\phi}]
+3\frac{\phi}{\xi}\partial_\xi\phi\partial_t\phi=
\boxed{\frac{\phi^3}{\xi}\partial_\xi[\frac{\partial_t\phi}{\phi}]}+V(\xi\phi)
\frac{\partial_t\phi}{\phi}\\
&\bullet\; V^{\ast}VV^\ast (\xi u)=-\frac{\phi^2}{\xi}\partial_\xi
[\frac{\phi^3}{\xi^2}\partial_\xi[\frac{\partial_t\phi}{\phi}]]-3
\frac{\phi^2}{\xi}\partial_\xi[\frac{\phi^2}{\xi^2}\partial_\xi\phi]
\frac{\partial_t\phi}{\phi}-3
\frac{\phi^2}{\xi}\frac{\phi^2}{\xi^2}\partial_\xi\phi \partial_\xi[
\frac{\partial_t\phi}{\phi}]\\
&\;\;\;\;\;\;\;\;\;\;\;\;\;\;\;\;\;\;\;\;\;\;\;=
-\frac{1}{\xi\phi}\partial_\xi
[\frac{\phi^6}{\xi^2}\partial_\xi[\frac{\partial_t\phi}{\phi}]]-3
\frac{\phi^2}{\xi}\partial_\xi[\frac{\phi^2}{\xi^2}\partial_\xi\phi]
\frac{\partial_t\phi}{\phi}\\
&\;\;\;\;\;\;\;\;\;\;\;\;\;\;\;\;\;\;\;\;\;\;\;=
\boxed{-\frac{1}{\xi\phi}\partial_\xi
[\frac{\phi^6}{\xi^2}\partial_\xi[\frac{\partial_t\phi}{\phi}]]}
+V^\ast V(\xi\phi)\frac{\partial_t\phi}{\phi}\label{tphi}\\
&\bullet\; VV^{\ast}VV^\ast (\xi
u)=-\frac{1}{\xi}\partial_\xi[\frac{\phi}{\xi^2}\partial_\xi
[\frac{\phi^6}{\xi^2}\partial_\xi[\frac{\partial_t\phi}{\phi}]]]-3\frac{1}{\xi}
\partial_\xi[\frac{\phi^4}{\xi^2}\partial_\xi[\frac{\phi^2}{\xi^2}\partial_\xi\phi]
\frac{\partial_t\phi}{\phi}]\\
&\;\;\;\;\;\;\;\;\;\;\;\;\;\;\;\;\;\;\;\;\;\;\;\;\;\;
=\boxed{-\frac{1}{\xi}\partial_\xi[\frac{\phi}{\xi^2}\partial_\xi
[\frac{\phi^6}{\xi^2}\partial_\xi[\frac{\partial_t\phi}{\phi}]]]}
+VV^\ast V(\xi\phi)\frac{\partial_t\phi}{\phi}+\frac{V^\ast
V(\xi\phi)}{\xi}\cdot\frac{\phi^2}{\xi}\partial_\xi
[\frac{\partial_t\phi}{\phi}]
\end{split}
\end{equation}
As we can see in the above, the term $\frac{\partial_t\phi}{\phi}$
and its derivatives naturally appear and there are many ways to
write them.
 The key idea is not to separate them randomly when
distributing spatial derivatives, but to find the right form of each
term. The boxed terms in \eqref{tphi} have been chosen in such a way
that the remaining terms in the right hand sides have the better or
the same integrability as the left hand sides.

We analyze the most intriguing full derivative terms $V^\ast
V_tV^\ast V(\xi\phi)$ and $VV^\ast V_tV^\ast (\xi u)$. Other terms
can be handled in a rather direct way.

\begin{claim} $$||V^\ast V_tV^\ast V(\xi\phi)||_{L^2_\xi}^2\leq
\mathcal{C}_1(\mathcal{E}^{1,4}(\phi,u))$$ where
$\mathcal{C}_1(\mathcal{E}^{1,4}(\phi,u))$ is a continuous function of
$\mathcal{E}^{1,4}(\phi,u)$.
\end{claim}

\begin{proof}
\[
\begin{split}
\frac{1}{6}V^\ast V_tV^\ast
V(\xi\phi)&=\frac{\phi^2}{\xi}\partial_\xi [\frac{1}{\xi^2}
\partial_\xi[\frac{\partial_t\phi}{\phi}\frac{\phi^4}{\xi^2}
\partial_\xi[\frac{\phi^2}{\xi^2}
\partial_\xi\phi]]]\\
&={\frac{\phi^2}{\xi}\partial_\xi
[\frac{1}{\xi^2}\frac{\partial_t\phi}{\phi}
\partial_\xi[\frac{\phi^4}{\xi^2}
\partial_\xi[\frac{\phi^2}{\xi^2}
\partial_\xi\phi]]]}+
{\frac{\phi^2}{\xi}\partial_\xi [\frac{1}{\xi^2}
\partial_\xi[\frac{\partial_t\phi}{\phi}]\frac{\phi^4}{\xi^2}
\partial_\xi[\frac{\phi^2}{\xi^2}
\partial_\xi\phi]]}\\
&=\frac{\partial_t\phi}{\phi}\frac{\phi^2}{\xi}\partial_\xi
[\frac{1}{\xi^2}
\partial_\xi[\frac{\phi^4}{\xi^2}
\partial_\xi[\frac{\phi^2}{\xi^2}
\partial_\xi\phi]]]+2\frac{\phi^2}{\xi}\partial_\xi
[\frac{\partial_t\phi}{\phi}] \cdot\frac{1}{\xi^2}
\partial_\xi[\frac{\phi^4}{\xi^2}
\partial_\xi[\frac{\phi^2}{\xi^2}
\partial_\xi\phi]]\\
&\;\;\;+\frac{\phi^2}{\xi}\partial_\xi[\frac{1}{\xi^2}
\partial_\xi[\frac{\partial_t\phi}{\phi}]]\cdot\frac{\phi^4}{\xi^2}
\partial_\xi[\frac{\phi^2}{\xi^2}
\partial_\xi\phi]\\
&\equiv (I)+(II)+(III)
\end{split}
\]
Since $(I)= \frac{\partial_t\phi}{\phi}V^\ast VV^\ast V(\xi\phi)$,
$(I)$ is controllable:
\[
 ||(I)||_{L^2_\xi}\leq ||\frac{\partial_t\phi}{\phi}||_{L^\infty_\xi}
||V^\ast VV^\ast V(\xi\phi)||_{L^2_\xi}
\]
 The second term is written as
\[
 (II)= -\frac{2}{3}\frac{\phi^2}{\xi}\partial_\xi
[\frac{\partial_t\phi}{\phi}]\cdot \frac{VV^\ast V(\xi\phi)}{\xi}
\]
From \eqref{tphi}, we get
\[
 \frac{\phi^2}{\xi}\partial_\xi
[\frac{\partial_t\phi}{\phi}]=\frac{VV^\ast (\xi u)}{\phi}
-\frac{V(\xi\phi)}{\phi}\frac{\partial_t\phi}{\phi}\;\in\;
 L_\xi^\infty.
\]
And since $ ||\frac{VV^\ast V(\xi\phi)}{\xi}||_{L^2_\xi}\leq
||V^\ast VV^\ast V(\xi\phi)||_{L^2_\xi}$,  $(II)$ is also controlled by the energy.
Now we turn into $(III)$. First,
we note that by Lemma \ref{key} and Lemma \ref{sup},
\[
\frac{1}{\xi}\partial_\xi[\frac{\phi^2}{\xi^2}
\partial_\xi\phi]\;\in\;L_\xi^2 \;\text{ and }\;\partial_\xi[\frac{\phi^2}{\xi^2}
\partial_\xi\phi]\;\in\;
 L_\xi^\infty.
\]
Next let us take a look at the other factor and rewrite it by using
boxed terms in \eqref{tphi}:
\[
\begin{split}
\frac{\phi^6}{\xi^2}\partial_\xi[\frac{1}{\xi^2}
\partial_\xi[\frac{\partial_t\phi}{\phi}]]=
\frac{\phi^6}{\xi^2}\partial_\xi[\frac{1}{\phi^6}\frac{\phi^6}{\xi^2}
\partial_\xi[\frac{\partial_t\phi}{\phi}]]=
\frac{1}{\xi^2}\partial_\xi[\frac{\phi^6}{\xi^2}
\partial_\xi[\frac{\partial_t\phi}{\phi}]]-6\frac{\phi^2}{\xi^2}
\partial_\xi\phi\cdot\frac{\phi^3}{\xi^2}
\partial_\xi[\frac{\partial_t\phi}{\phi}]\\
=-\frac{\phi}{\xi} V^\ast VV^\ast(\xi u)+ \frac{V^\ast
V(\xi\phi)}{\xi} \frac{V^\ast (\xi u)}{\xi}
-2\frac{V(\xi\phi)}{\xi}\frac{VV^\ast (\xi
u)}{\xi}+2\frac{\xi}{\phi} |\frac{V(\xi\phi)}{\xi}|^2\frac{V^\ast
(\xi u)}{\xi^2}
\end{split}
\]
Thus $(III)$ can be rewritten as follows:
\[
\begin{split}
(III)=&-\frac{\phi}{\xi} \frac{V^\ast VV^\ast(\xi u)}{\xi}\cdot
\partial_\xi[\frac{\phi^2}{\xi^2}
\partial_\xi\phi]-2\frac{V(\xi\phi)}{\xi}\frac{VV^\ast (\xi u)}{\xi}
\cdot\frac{1}{\xi}\partial_\xi[\frac{\phi^2}{\xi^2}
\partial_\xi\phi]\\
&+\frac{V^\ast V(\xi\phi)}{\xi} \frac{V^\ast (\xi
u)}{\xi}\cdot\frac{1}{\xi}\partial_\xi[\frac{\phi^2}{\xi^2}
\partial_\xi\phi]+2\frac{\xi}{\phi}
|\frac{V(\xi\phi)}{\xi}|^2\frac{V^\ast (\xi
u)}{\xi^2}\cdot\frac{1}{\xi}\partial_\xi[\frac{\phi^2}{\xi^2}
\partial_\xi\phi]
\end{split}
\]
Hence $(III)$ can be controlled by $\mathcal{E}^{1,4}(\phi, u)$.
\end{proof}

 Now let
us move onto $VV^\ast V_tV^\ast (\xi u)$. The treatment of this term
contains another flavor.

\begin{claim}
$$|| VV^\ast V_tV^\ast (\xi u)||_{L_\xi^2}^2\leq
\mathcal{C}_2(\mathcal{E}^{1,4}(\phi,u))$$ where
$\mathcal{C}_2(\mathcal{E}^{1,4}(\phi,u))$ is a continuous function of
$\mathcal{E}^{1,4}(\phi,u)$.
\end{claim}

\begin{proof}
 We use the continuity equation in \eqref{euler1} first
to deal with $V_tV^\ast (\xi u)$.  Since $V^\ast (\xi
u)=\xi\partial_t\phi$ and
$V(\xi\partial_t\phi)=\frac{1}{\xi}\partial_\xi
[\phi^2\partial_t\phi]=\frac{1}{\xi}[\phi^2\partial_\xi\partial_t\phi
+2\phi\partial_\xi\phi\partial_t\phi]$, we can write $V_t V^\ast (\xi u)$
as the following:
\[
\frac{1}{2}V_t V^\ast (\xi u)=\frac{1}{\xi}\partial_\xi [\phi
|\partial_t\phi|^2]=\frac{1}{\xi}[\partial_\xi\phi|\partial_t\phi|^2
+2\phi\partial_t\partial_\xi\phi\partial_t\phi]=2\frac{\partial_t\phi}{\phi}
V(\xi\partial_t\phi)-3\frac{\partial_\xi\phi
|\partial_t\phi|^2}{\xi}
\]
Apply $V^\ast$:
\[
\begin{split}
&\frac{1}{2}V^\ast V_t V^\ast (\xi
u)=-\frac{\phi^2}{\xi}\partial_\xi [2\frac{\partial_t\phi}{\phi}
\frac{VV^\ast (\xi
u)}{\xi}-3\frac{\partial_\xi\phi |\partial_t\phi|^2}{\xi^2}]\\
&=2\frac{\partial_t\phi}{\phi}V^\ast V V^\ast (\xi u)-2
\frac{\phi^2}{\xi}\partial_\xi[\frac{\partial_t\phi}{\phi}]\cdot\frac{VV^\ast
(\xi u)}{\xi}+3\frac{\phi^2}{\xi}\partial_\xi
[\frac{\phi^2}{\xi^2}\partial_\xi\phi
|\frac{\partial_t\phi}{\phi}|^2]\\
&=2\frac{\partial_t\phi}{\phi}V^\ast V V^\ast (\xi u)-2
\frac{\phi^2}{\xi}\partial_\xi[\frac{\partial_t\phi}{\phi}]\underbrace{\{
\frac{VV^\ast (\xi
u)}{\xi}-3\frac{\partial_t\phi}{\phi}\frac{\phi^2}{\xi^2}\partial_\xi\phi\}}
_{(\star)}
+3\frac{\phi^2}{\xi}\partial_\xi[\frac{\phi^2}{\xi^2}\partial_\xi\phi]
|\frac{\partial_t\phi}{\phi}|^2
\end{split}
\]
Note that $(\star)$ reduces to:
\[
(\star)=\frac{1}{\xi^2}\partial_\xi[\phi^3\frac{\partial_t\phi}{\phi}]-
3\frac{\partial_t\phi}{\phi}\frac{\phi^2}{\xi^2}\partial_\xi\phi
=\frac{\phi^3}{\xi^2}\partial_\xi[\frac{\partial_t\phi}{\phi}]
\]
Apply $V$:
\[
\begin{split}
&\frac{1}{2}VV^\ast V_t V^\ast (\xi u)\\
&=2\frac{1}{\xi}\partial_\xi
[\frac{\phi^2}{\xi}\frac{\partial_t\phi}{\phi}V^\ast V V^\ast (\xi
u)]-2\frac{1}{\xi}\partial_\xi
[\frac{1}{\phi^5}|\frac{\phi^6}{\xi^2}\partial_\xi
[\frac{\partial_t\phi}{\phi}]|^2]+3\frac{1}{\xi}\partial_\xi
[\frac{\phi^4}{\xi^2}\partial_\xi[\frac{\phi^2}{\xi^2}\partial_\xi\phi]
|\frac{\partial_t\phi}{\phi}|^2]\\
&=(I)+(II)+(III)
\end{split}
\]
We rewrite $(I),\;(II),\;(III)$ as follows:
\[
 \begin{split}
  (I)&=2\frac{\partial_t\phi}{\phi}VV^\ast VV^\ast (\xi u)+
2\frac{\phi^2}{\xi}\partial_\xi
[\frac{\partial_t\phi}{\phi}]\frac{V^\ast VV^\ast (\xi u)}{\xi}\\
(II)&=10\frac{\partial_\xi\phi}{\xi\phi^6}|\frac{\phi^6}{\xi^2}\partial_\xi
[\frac{\partial_t\phi}{\phi}]|^2
-4\frac{1}{\xi\phi^5}\frac{\phi^6}{\xi^2}\partial_\xi
[\frac{\partial_t\phi}{\phi}]\cdot
\partial_\xi[\frac{\phi^6}{\xi^2}\partial_\xi
[\frac{\partial_t\phi}{\phi}]]\\
(III)&= 3\frac{1}{\xi}\partial_\xi
[\frac{\phi^4}{\xi^2}\partial_\xi[\frac{\phi^2}{\xi^2}\partial_\xi\phi]]
|\frac{\partial_t\phi}{\phi}|^2+6\frac{1}{\xi}
\frac{\phi^4}{\xi^2}\partial_\xi[\frac{\phi^2}{\xi^2}\partial_\xi\phi]
\cdot\frac{\partial_t\phi}{\phi}\partial_\xi[\frac{\partial_t\phi}{\phi}]
 \end{split}
\]
It is easy to see that $(I)\text{ and }(III)$ can be controlled by
the energy functional. On the other hand, in order to take care of
$(II)$, a special attention is needed. Since $\frac{V^\ast VV^\ast(\xi u)}{\xi}$ is
bounded by $VV^\ast VV^\ast(\xi u)$ in $L_\xi^2$, we obtain
\begin{equation}
h\equiv \frac{1}{\xi^3}\partial_\xi[\frac{\phi^6}{\xi^2}\partial_\xi
[\frac{\partial_t\phi}{\phi}]]\in L_\xi^2\label{h}\,.
\end{equation}
Thus the second term in $(II)$ is bounded by the energy functional.
To prove that the first term is also bounded, we claim
\[
 \frac{1}{\xi^7}|\frac{\phi^6}{\xi^2}\partial_\xi
[\frac{\partial_t\phi}{\phi}]|^2\in L_\xi^2\,.
\]
By using \eqref{h}, rewrite $\frac{\phi^6}{\xi^2}\partial_\xi
[\frac{\partial_t\phi}{\phi}]$ as follows:
\[
 \frac{\phi^6}{\xi^2}\partial_\xi
[\frac{\partial_t\phi}{\phi}]=\int_0^\xi\zeta^3 h d\zeta
\]
Applying H$\ddot{\text{o}}$lder inequality, we observe that
\[
 |\frac{\phi^6}{\xi^2}\partial_\xi
[\frac{\partial_t\phi}{\phi}]|^2\leq \xi^7 ||h||_{L_\xi^2}^2\,.
\]
Hence we get
\[
 ||\frac{1}{\xi^7}|\frac{\phi^6}{\xi^2}\partial_\xi
[\frac{\partial_t\phi}{\phi}]|^2||_{L_\xi^2}\leq ||h||_{L_\xi^2}^2\,.
\]
This finishes the a priori estimates for $k=1$ as well as the claim.
\end{proof}

\subsection{The case when $\frac{1}{2}<k<1$ ($\gamma>3$)}

In this subsection, we prove Proposition \ref{apriori} for the case
$\frac{1}{2}<k<1$ and $s=4$.  First, by Lemma \ref{sup}, one finds
that $\partial_\xi\phi$, $\frac{\xi}{\phi}$, and
$\frac{\partial_t\phi}{\xi}$ are bounded by the energy functional
$\mathcal{E}^{k,4}(\phi,u)$. We recall the reformulated Euler
equations \eqref{VVk}. One can apply $V,V^\ast$ alternatingly as in
the case $k=1$, and integrate to get
\[
\frac{1}{2}\frac{d}{dt}\sum_{i=1}^4\int\frac{1}{(2k+1)^2}
|(V)^i(\xi^k\phi)|^2+|(V^\ast)^i(\xi^ku)|^2 d\xi\leq \text{mixed
nonlinear terms}\,,
\]
where the mixed terms have the same form in \eqref{mixed}. Thus it
suffices to show that those nonlinear terms are bounded by
$\mathcal{E}^{k,4}(\phi,u)$. We focus on the intriguing term
$VV^\ast V_tV^\ast(\xi^ku)$ as well as $V^\ast V_tV^\ast(\xi^ku)$, $
V_tV^\ast(\xi^ku)$. As we saw in the case $k=1$, in order to treat
the mixed terms, the careful analysis of $\partial_t\phi$ terms is
required. First, take a look at $V^\ast(\xi^ku)$ and
$VV^{\ast}(\xi^k u)$.
\begin{equation*}
\begin{split}
&\bullet\; V^\ast (\xi^k u)=-\frac{\phi^{2k}} {\xi^k}\partial_\xi
u=\xi^k\partial_t\phi\\
 &\bullet\;
VV^{\ast}(\xi^k
u)=\frac{1}{\xi^k}\partial_\xi[\phi^{2k}\partial_t\phi]
=\frac{\phi^{2k+1}}{\xi^k}\partial_\xi[\frac{\partial_t\phi}{\phi}]
+(2k+1)\frac{\phi^{2k}}{\xi^k}\partial_\xi\phi\frac{\partial_t\phi}
{\phi}\\
&\;\;\;\;\;\;\;\;\;\;\;\;\;\;\;\;\;\;\;\;=
\boxed{\frac{\phi^{2k+1}}{\xi^k}\partial_\xi[\frac{\partial_t\phi}{\phi}]}
+V(\xi^k\phi) \frac{\partial_t\phi}{\phi}
\end{split}
\end{equation*}
Because $||\frac{V(\xi^k\phi)}{\xi}||_{L^2_\xi}$ and
$||\frac{VV^{\ast}(\xi^k u)}{\xi}||_{L^2_\xi}$ are bounded by
$\mathcal{E}^{k,4}$ as an application of Lemma \ref{key}, we deduce
that
$||\phi^k\partial_\xi[\frac{\partial_t\phi}{\phi}]||_{L^2_\xi}$
is also bounded by $\mathcal{E}^{k,4}$. And by Lemma \ref{sup}, we obtain
$\phi\partial_\xi[\frac{\partial_t\phi}{\phi}]\in L^\infty_\xi$.
  Now let us compute
$V_tV^\ast(\xi^k u)$ and write in terms of the energy:
\begin{equation*}
\begin{split}
\frac{1}{2k}V_tV^\ast (\xi^k u)&=\frac{1}{\xi^k}\partial_\xi
[\phi^{2k+1}|\frac{\partial_t\phi}{\phi}|^2]=\frac{\partial_t\phi}{\phi}\{2
\frac{\phi^{2k+1}}{\xi^k}\partial_\xi[\frac{\partial_t\phi}{\phi}]+(2k+1)
\frac{\phi^{2k}}{\xi^k}\partial_\xi\phi\frac{\partial_t\phi}{\phi}\}\\
&=2\frac{\partial_t\phi}{\phi}VV^\ast(\xi^k
u)-|\frac{\partial_t\phi}{\phi}|^2V(\xi^k\phi)
\end{split}
\end{equation*}
It is clear that this term is bounded by the energy functional. Next
we write out $V^{\ast}VV^\ast (\xi^k u)$.
\begin{equation}\label{box}
\begin{split}
&\bullet\; V^{\ast}VV^\ast (\xi^k
u)=-\frac{\phi^{2k}}{\xi^k}\partial_\xi
[\frac{\phi^{2k+1}}{\xi^{2k}}\partial_\xi[\frac{\partial_t\phi}{\phi}]]\\
&\;\;\;\;\;\;\;\;\;\;\;\;\;\;\;\;\;\;\;\;\;\;\;\;\;\;\;\;\,-(2k+1)\{
\frac{\phi^{2k}}{\xi^k}\partial_\xi[\frac{\phi^{2k}}{\xi^{2k}}\partial_\xi\phi]
\frac{\partial_t\phi}{\phi}+
\frac{\phi^{4k}}{\xi^{3k}}\partial_\xi\phi
\partial_\xi[
\frac{\partial_t\phi}{\phi}]\}\\
&\;\;\;\;\;\;\;\;\;\;\;\;\;\;\;\;\;\;\;\;\;\;\;\;=
-\frac{1}{\xi^k\phi}\partial_\xi
[\frac{\phi^{4k+2}}{\xi^{2k}}\partial_\xi[\frac{\partial_t\phi}{\phi}]]-(2k+1)
\frac{\phi^{2k}}{\xi^k}\partial_\xi[\frac{\phi^{2k}}{\xi^{2k}}\partial_\xi\phi]
\frac{\partial_t\phi}{\phi}
\\
&\;\;\;\;\;\;\;\;\;\;\;\;\;\;\;\;\;\;\;\;\;\;\;\;=
\boxed{-\frac{1}{\xi^k\phi}\partial_\xi
[\frac{\phi^{4k+2}}{\xi^{2k}}\partial_\xi[\frac{\partial_t\phi}{\phi}]]}
+V^\ast V(\xi^k\phi)\frac{\partial_t\phi}{\phi}
\end{split}
\end{equation}
Thus we note that the boxed term $\frac{1}{\xi^k\phi}\partial_\xi
[\frac{\phi^{4k+2}}{\xi^{2k}}\partial_\xi[\frac{\partial_t\phi}{\phi}]]$
is the right form of the second derivative of $\partial_t\phi$ of
which structure we do not want to destroy. Here is $V^\ast
V_tV^\ast(\xi^k u)$:
\begin{equation*}
\begin{split}
&\frac{1}{2k}V^\ast V_tV^\ast (\xi^k
u)=-\frac{\phi^{2k}}{\xi^k}\partial_\xi[\frac{1}{\xi^k}\frac{\partial_t\phi}{\phi}\{2
\frac{\phi^{2k+1}}{\xi^k}\partial_\xi[\frac{\partial_t\phi}{\phi}]+(2k+1)
\frac{\phi^{2k}}{\xi^k}\partial_\xi\phi\frac{\partial_t\phi}{\phi}\}]\\
&=-2\frac{\partial_t\phi}{\phi}\frac{1}{\xi^k\phi}\partial_\xi
[\frac{\phi^{4k+2}}{\xi^{2k}}\partial_\xi[\frac{\partial_t\phi}{\phi}]]
-2\frac{\phi^{4k+1}}{\xi^{3k}}|\partial_\xi[\frac{\partial_t\phi}{\phi}]|^2
-(2k+1)|\frac{\partial_t\phi}{\phi}|^2\cdot
\frac{\phi^{2k}}{\xi^k}\partial_\xi[\frac{\phi^{2k}}{\xi^{2k}}\partial_\xi
\phi]
\end{split}
\end{equation*}
It is easy to see that the first and third terms are bounded by the
energy functional. For the second term, writing it as
$$-2\frac{\phi^{4k+1}}{\xi^{3k}}
|\partial_\xi[\frac{\partial_t\phi}{\phi}]|^2
=\underbrace{-2\frac{\phi^{3k+1}}{\xi^{3k}}\partial_\xi[\frac{\partial_t\phi}{\phi}]}_{L^\infty_\xi}
\cdot \underbrace{\phi^k\partial_\xi[\frac{\partial_t\phi}{\phi}]}
_{L^2_\xi}\,,$$
we deduce
that it is also controlled by $\mathcal{E}^{k,4}$, and in result, we
conclude that $V^\ast V_tV^\ast (\xi^k u)$ is bounded by
$\mathcal{E}^{k,4}$. Next $VV^{\ast}VV^\ast (\xi^k u)$:
\begin{equation*}
\begin{split}
\bullet\; VV^{\ast}VV^\ast (\xi^k
u)=-\frac{1}{\xi^k}\partial_\xi[\frac{\phi^{2k-1}}{\xi^{2k}}\partial_\xi
[\frac{\phi^{4k+2}}{\xi^{2k}}
\partial_\xi[\frac{\partial_t\phi}{\phi}]]]-(2k+1)\frac{1}{\xi^k}
\partial_\xi[\frac{\phi^{4k}}{\xi^{2k}}
\partial_\xi[\frac{\phi^{2k}}{\xi^{2k}}\partial_\xi\phi]
\frac{\partial_t\phi}{\phi}]\\
=\boxed{-\frac{1}{\xi^k}\partial_\xi[\frac{\phi^{2k-1}}{\xi^{2k}}\partial_\xi
[\frac{\phi^{4k+2}}{\xi^{2k}}
\partial_\xi[\frac{\partial_t\phi}{\phi}]]]}
+VV^\ast
V(\xi^k\phi)\frac{\partial_t\phi}{\phi}+\frac{\phi^{2k}}{\xi^{2k}}
{V^\ast V(\xi^k\phi)}\cdot\partial_\xi [\frac{\partial_t\phi}{\phi}]
\end{split}
\end{equation*}
Note that the boxed term is bounded in $L^2_\xi$ by the energy
functional. Now we write $VV^\ast V_tV^\ast (\xi^k u)$ in terms of
the boxed terms as well as the energy:
\begin{equation*}
\begin{split}
\frac{1}{2k}&VV^\ast V_tV^\ast (\xi^k
u)=-2\frac{\partial_t\phi}{\phi}\cdot \frac{1}{\xi^k}\partial_\xi[
\frac{\phi^{2k-1}}{\xi^{2k}}\partial_\xi
[\frac{\phi^{4k+2}}{\xi^{2k}}\partial_\xi[\frac{\partial_t\phi}{\phi}]]]\\
-&6\frac{\phi^{2k+1}}{\xi^{2k}}\partial_\xi[\frac{\partial_t\phi}{\phi}]\cdot
\frac{1}{\xi^k\phi^2}\partial_\xi
[\frac{\phi^{4k+2}}{\xi^{2k}}\partial_\xi[\frac{\partial_t\phi}{\phi}]]
+(4k+6)
\underline{\partial_\xi\phi\frac{\phi^{6k}}{\xi^{5k}}|\partial_\xi
[\frac{\partial_t\phi}{\phi}]|^2}_{(\star)}\\
-&(2k+1)|\frac{\partial_t\phi}{\phi}|^2\cdot\frac{1}{\xi^k}\partial_\xi
[\frac{\phi^{4k}}{\xi^{2k}}\partial_\xi
[\frac{\phi^{2k}}{\xi^{2k}}\partial_\xi\phi]]-(4k+2)
\frac{\partial_t\phi}{\phi}
\frac{\phi^{3k}}{\xi^{3k}}\partial_\xi[\frac{\phi^{2k}}{\xi^{2k}}\partial_\xi
\phi]\cdot\phi^k\partial_\xi[\frac{\partial_t\phi}{\phi}]
\end{split}
\end{equation*}
It is clear that  except $(\star)$, the first factor of each term in the right hand side
is bounded in $L^\infty_\xi$ and the second factor is bounded in
$L^2_\xi$ by $\mathcal{E}^{k,4}$. In order to show that
$(\star)$ is bounded in $L^2_\xi$, first note that by Lemma \ref{key},
$||\frac{V^{\ast}VV^\ast (\xi^k u)}{\xi}||_{L^2_\xi}$ and
$||\frac{V^{\ast}V (\xi^k \phi)}{\xi}||_{L^2_\xi}$ are bounded by
$\mathcal{E}^{k,4}(\phi,u)$ due to the fact $k> \frac{1}{2}$, from
the relation \eqref{box}, thus we get
\[ h\equiv\frac{1}{\xi^{k+2}}\partial_\xi
[\frac{\phi^{4k+2}}{\xi^{2k}}\partial_\xi[\frac{\partial_t\phi}{\phi}]]\in
L_\xi^2\,,
\]
and in turn we obtain
\[
\begin{split}
\frac{\phi^{4k+2}}{\xi^{2k}}\partial_\xi[\frac{\partial_t\phi}{\phi}]
&=\int_0^\xi
|\xi'|^{k+2}hd\xi'\leq \xi^{k+\frac{5}{2}}\|h\|_{L^2_\xi}\\
\Longrightarrow
||\xi^k |\partial_\xi[\frac{\partial_t\phi}{\phi}]|^2||_{L_\xi^2}^2&\leq
||h||^4_{L_\xi^2}\int_0^1 \xi^{2k+(4k+10)-(8k+8)}d\xi
\end{split}
\]
Note that the last integral is bounded. Therefore we
conclude that $||VV^\ast V_tV^\ast (\xi^k u)||_{L^2_\xi}$ is bounded
by $\mathcal{E}^{k,4}$. Similarly, we deduce the same conclusion for
other mixed terms and it finishes the proof of Proposition
\ref{apriori} for $\frac{1}{2}<k<1$.

\subsection{The case when $\lceil k\rceil\geq 2$ $(1<\gamma<3)$}

We now turn into the general $k$. The spirit is the same as the case
when $k=1$: we need to carry out $L^\infty_\xi$ estimates and
nonlinear estimates. For large $k$, however, the number of mixed nonlinear
terms increases accordingly and it is not an easy task to work term
by term.
 We will present a systematic way of treating those terms
involving derivatives of $\partial_t\phi$.

The following lemma is a direct result from Lemma \ref{sup} and
\ref{infty}, and it is
useful to
justify the assumption \eqref{AA} in $\mathcal{E}^{k,\lceil k\rceil
+3}(\phi,u)$.

\begin{lemma}\label{supk}
(1) We obtain the following $L_\xi^\infty$ estimates:
\[
 ||\frac{V(\xi^k\phi)}{\xi^k}||_{L^\infty_\xi},\;
 ||\frac{V^\ast V(\xi^k\phi)}{\xi^k}||_{L^\infty_\xi},\;
||\frac{V^\ast(\xi^k u)}{\xi^k}||_{L^\infty_\xi},\;||
\frac{VV^\ast(\xi^k u)}{\xi^k}||_{L^\infty_\xi}\leq
\mathcal{C}_3(\mathcal{E}^{k,\lceil k\rceil +3}(\phi,u))
\]
where $\mathcal{C}_3(\mathcal{E}^{k,\lceil k\rceil +3}(\phi,u))$ is
a continuous function of
$||\frac{\xi}{\phi}||_{L_\xi^\infty}$ and $\mathcal{E}^{k,\lceil
k\rceil
+3}(\phi,u)$. \\
(2) For $3\leq i\leq \lceil k\rceil +1$
\[
 ||\frac{(V)^i(\xi^k\phi)}{\xi^{\lceil k\rceil +2-i}}||_{L^\infty_\xi},\;
||\frac{(V^\ast)^i(\xi^k u)}{\xi^{\lceil k\rceil
+2-i}}||_{L^\infty_\xi}\leq \mathcal{C}_4(\mathcal{E}^{k,\lceil
k\rceil +3}(\phi,u))
\]
where $\mathcal{C}_4(\mathcal{E}^{k,\lceil k\rceil +3}(\phi,u))$ is
a continuous function of
$||\frac{\xi}{\phi}||_{L_\xi^\infty}$ and $\mathcal{E}^{k,\lceil
k\rceil +3}(\phi,u)$.
\end{lemma}

We start with $L^\infty_\xi$ estimate of $\phi$. We list out $V$,
$V^{\ast}V$, $VV^{\ast}V$ of $\xi^{k}\phi$ for references.
\begin{align*}
\begin{split}
&\bullet\; V(\xi^{k}\phi)=
\frac{1}{\xi^{k}}\partial_\xi
[\phi^{2k+1}]=(2k+1)
\frac{\phi^{2k}}
{\xi^{k}}\partial_\xi\phi\\
 &\bullet\;
V^{\ast}V(\xi^{k}\phi)=-
(2k+1) \frac{\phi^{2k}}
{\xi^{k}}\partial_\xi
[\frac{\phi^{2k}}
{\xi^{2k}}\partial_\xi\phi]\\
&\bullet\;
VV^{\ast}V(\xi^{k}\phi)=-(2k+1)
\frac{1}{\xi^{k}}
\partial_\xi[
\frac{\phi^{4k}}
{\xi^{2k}}\partial_\xi
[\frac{\phi^{2k}}
{\xi^{2k}}\partial_\xi\phi]]
\end{split}
\end{align*}
One advantage of $V$ and $V^\ast$ formulation is that $L_\xi^\infty$
control of $\phi$ is cheap to get:
\[
\int \xi^{k}\phi\cdot
V(\xi^{k}\phi) d\xi=\int
\phi\partial_\xi[\phi^{2k+1}]d\xi=
\frac{2k+1}{2k+2}\phi^{2k+2}
\]
Applying Lemma \ref{supk}, we obtain that $\partial_\xi\phi$ is bounded and continuous if the
assumption \eqref{AA} holds and $\mathcal{E}^{k,\lceil
k\rceil+3}(\phi, u)$ is bounded.

Next we turn to $u$ variable.
\begin{align}
\begin{split}
&\bullet\; V^\ast (\xi^k u)=-\frac{\phi^{2k}}
{\xi^k}\partial_\xi u\\
 &\bullet\;
VV^{\ast}(\xi^k
u)=-\frac{1}{\xi^k}\partial_\xi[\frac{\phi^{4k}}{\xi^{2k}}
\partial_\xi u]\\
&\bullet\; V^{\ast}VV^\ast (\xi^k
u)=\frac{\phi^{2k}}{\xi^k}\partial_\xi
[\frac{1}{\xi^{2k}}\partial_\xi[\frac{\phi^{4k}}{\xi^{2k}}
\partial_\xi u]]
\end{split}\label{uk}
\end{align}
Note that $\frac{\partial_t\phi}{\xi}$ can be estimated in terms of
$\partial_\xi u$ via the equation \eqref{euler}:
\begin{equation*}
 {\partial_t\phi}=-\frac{\phi^{2k}}
{\xi^{2k}}\partial_\xi u=\frac{V^\ast (\xi^k u)}{\xi^k}
\end{equation*}
By Lemma \ref{supk}, we  deduce that
 $\frac{VV^\ast(\xi^k
 u)}{\xi^k}=\frac{1}{\xi^{2k}}\partial_\xi[\phi^{2k}\partial_t\phi]$
is bounded and continuous if $\mathcal{E}^{k,\lceil
k\rceil+3}(\phi,u)$ is bounded. Letting $h$ be
$\frac{1}{\xi^{2k}}\partial_\xi[\phi^{2k}\partial_t\phi]$, we can
write $\phi^{2k}{\partial_t\phi}=\int_0^\xi\xi'^{2k}h d\xi$, and
therefore we conclude $\frac{\partial_t\phi}{\xi}$ is also bounded.
 By the same  continuity  argument
as in $k=1$ case, we can verify the same boundary behavior for a short
time and the assumption \eqref{AA}.

We now perform the energy estimates. From \eqref{VVk},
we have the conservation
of the zeroth order
energy:
\[
\frac{1}{2}\frac{d}{dt}\int \frac{1}{2k+1} |\xi^{k}\phi|^2+
|\xi^{k}u|^2 d\xi=0
\]
Apply $V$ and $V^{\ast}$ to \eqref{VVk} and use
$V_t(\xi^{k}\phi)=\frac{2k}{2k+1}
\partial_tV(\xi^{k}\phi)$ to get
\begin{equation}
\begin{split}
&\partial_tV(\xi^{k}\phi)
-(2k+1)VV^{\ast}(\xi^{k}u)=0\\
&\partial_tV^{\ast}(\xi^{k}u) +\frac{1}{2k+1}V^{\ast}V(\xi^{k}
\phi)=V_t^{\ast}(\xi^{k}u) \label{VVVk}
\end{split}
\end{equation}
Apply $V^{\ast}$ and $V$ to \eqref{VVVk} and use $VV_t^\ast=V_tV^\ast$
to get
\begin{equation}\label{i=2}
\begin{split}
&\partial_tV^{\ast}V(\xi^{k}\phi)
-(2k+1)V^{\ast}VV^{\ast}(\xi^{k}u)=V_t^{\ast}
V(\xi^{k}\phi)\\
&\partial_tVV^{\ast}(\xi^{k}u)
+\frac{1}{2k+1}VV^{\ast}V(\xi^{k}
\phi)=2V_t
V^{\ast}(\xi^{k}u)
\end{split}
\end{equation}
By keeping taking $V$ and $V^{\ast}$ alternatingly, one obtains higher
order
equations for $(V)^i(\xi^k\phi)$ and $(V^\ast)^i(\xi^k u)$ for any $i$.
Indeed, the mixed terms in the right hand sides can be written in a
systematic way. For $i=2j+1$ where $j\geq 1$:
\begin{equation}
\begin{split}
\partial_t(V)^{2j+1}(\xi^{k}\phi)
-(2k+1)(V^{\ast})^{2j+2}(\xi^{k}u)&=
2\sum_{l=0}^{j-1}(V^\ast)^{2j-2l-2}V_tV^\ast(V)^{2l+1}(\xi^{k}\phi)\\
\partial_t(V^{\ast})^{2j+1}(\xi^{k}u) +\frac{1}{2k+1}(V)^{2j+2}(\xi^{k}
\phi)&=2\sum_{l=0}^{j-1}(V^\ast)^{2j-2l-1}V_tV^\ast(V^\ast)^{2l}(\xi^{k}u)\\
&\;\;\;+V_t^\ast (V^\ast)^{2j}(\xi^ku) \label{i=2j+1}
\end{split}
\end{equation}
For $i=2j$ where $j\geq 2$:
\begin{equation}
\begin{split}
\partial_t(V)^{2j}(\xi^{k}\phi)
-(2k+1)(V^{\ast})^{2j+1}(\xi^{k}u)&=2\sum_{l=0}^{j-2} (V^\ast)^{2j-2l-3}
V_tV^\ast(V)^{2l+1}(\xi^k\phi)\\
&\;\;\;+V_t^\ast (V)^{2j-1}(\xi^k\phi)\\
\partial_t(V^{\ast})^{2j}(\xi^{k}u) +\frac{1}{2k+1}(V)^{2j+1}(\xi^{k}
\phi)&=2\sum_{l=0}^{j-1}(V^\ast)^{2j-2l-2}V_tV^\ast(V^\ast)^{2l}(\xi^{k}u)
\label{i=2j}
\end{split}
\end{equation}
Our main goal is to estimate the mixed terms in the right hand sides
of \eqref{i=2j+1} and \eqref{i=2j}. Note that the most intriguing
cases seem to be when $l=0$, where the most spatial derivatives of
$\frac{\partial_t\phi}{\phi}$ are present. Before getting into the
estimates, let us get the better understanding of the effect of the
operator $V_t$. First, recall $(V^\ast)^i(\xi^k u)$'s. They give
rise to $\partial_t\phi$ terms via the equations \eqref{euler}
 and furthermore they predict the
right form of spatial derivatives of $\partial_t\phi$ terms. We
borrow the computations from the previous section for $i\leq 4$.
\begin{equation}
\begin{split}
&\bullet\; V^\ast (\xi^k u)=-\frac{\phi^{2k}} {\xi^k}\partial_\xi
u=\xi^k\partial_t\phi\\
 &\bullet\;
VV^{\ast}(\xi^k
u)=\frac{1}{\xi^k}\partial_\xi[\phi^{2k}\partial_t\phi]
=\frac{\phi^{2k+1}}{\xi^k}\partial_\xi[\frac{\partial_t\phi}{\phi}]
+(2k+1)\frac{\phi^{2k}}{\xi^k}\partial_\xi\phi\frac{\partial_t\phi}
{\phi}\\
&\;\;\;\;\;\;\;\;\;\;\;\;\;\;\;\;\;\;\;\;=
\boxed{\frac{\phi^{2k+1}}{\xi^k}\partial_\xi[\frac{\partial_t\phi}{\phi}]}
+V(\xi^k\phi) \frac{\partial_t\phi}{\phi}\label{u2}
\end{split}
\end{equation}
Now let us compute $V_tV^\ast(\xi^k u)$ and write in terms of the
energy:
\begin{equation}
\begin{split}
\frac{1}{2k}V_tV^\ast (\xi^k u)&=\frac{1}{\xi^k}\partial_\xi
[\phi^{2k+1}|\frac{\partial_t\phi}{\phi}|^2]=\frac{\partial_t\phi}{\phi}\{2
\frac{\phi^{2k+1}}{\xi^k}\partial_\xi[\frac{\partial_t\phi}{\phi}]+(2k+1)
\frac{\phi^{2k}}{\xi^k}\partial_\xi\phi\frac{\partial_t\phi}{\phi}\}\\
&=2\frac{\partial_t\phi}{\phi}VV^\ast(\xi^k
u)-|\frac{\partial_t\phi}{\phi}|^2V(\xi^k\phi)\label{v_tvu}
\end{split}
\end{equation}
Next we write out $V^{\ast}VV^\ast (\xi^k u)$.
\begin{equation*}
\begin{split}
&\bullet\; V^{\ast}VV^\ast (\xi^k
u)=-\frac{\phi^{2k}}{\xi^k}\partial_\xi
[\frac{\phi^{2k+1}}{\xi^{2k}}\partial_\xi[\frac{\partial_t\phi}{\phi}]]-(2k+1)\{
\frac{\phi^{2k}}{\xi^k}\partial_\xi[\frac{\phi^{2k}}{\xi^{2k}}\partial_\xi\phi]
\frac{\partial_t\phi}{\phi}+
\frac{\phi^{4k}}{\xi^{3k}}\partial_\xi\phi
\partial_\xi[
\frac{\partial_t\phi}{\phi}]\}\\
&\;\;\;\;\;\;\;\;\;\;\;\;\;\;\;\;\;\;\;\;\;\;\;\;=
-\frac{1}{\xi^k\phi}\partial_\xi
[\frac{\phi^{4k+2}}{\xi^{2k}}\partial_\xi[\frac{\partial_t\phi}{\phi}]]-(2k+1)
\frac{\phi^{2k}}{\xi^k}\partial_\xi[\frac{\phi^{2k}}{\xi^{2k}}\partial_\xi\phi]
\frac{\partial_t\phi}{\phi}
\\
&\;\;\;\;\;\;\;\;\;\;\;\;\;\;\;\;\;\;\;\;\;\;\;\;=
\boxed{-\frac{1}{\xi^k\phi}\partial_\xi
[\frac{\phi^{4k+2}}{\xi^{2k}}\partial_\xi[\frac{\partial_t\phi}{\phi}]]}
+V^\ast V(\xi^k\phi)\frac{\partial_t\phi}{\phi}
\end{split}
\end{equation*}
Thus we note that $\frac{1}{\xi^k\phi}\partial_\xi
[\frac{\phi^{4k+2}}{\xi^{2k}}\partial_\xi[\frac{\partial_t\phi}{\phi}]]$
is the right form of the second derivative of $\partial_t\phi$ that
we do not want to destroy. Here is $V^\ast V_tV^\ast(\xi^k u)$.
\begin{equation}
\begin{split}
&\frac{1}{2k}V^\ast V_tV^\ast (\xi^k
u)=-\frac{\phi^{2k}}{\xi^k}\partial_\xi[\frac{1}{\xi^k}\frac{\partial_t\phi}{\phi}\{2
\frac{\phi^{2k+1}}{\xi^k}\partial_\xi[\frac{\partial_t\phi}{\phi}]+(2k+1)
\frac{\phi^{2k}}{\xi^k}\partial_\xi\phi\frac{\partial_t\phi}{\phi}\}]\\
&=-2\frac{\partial_t\phi}{\phi}\frac{1}{\xi^k\phi}\partial_\xi
[\frac{\phi^{4k+2}}{\xi^{2k}}\partial_\xi[\frac{\partial_t\phi}{\phi}]]
-2\frac{\phi^{4k+1}}{\xi^{3k}}|\partial_\xi[\frac{\partial_t\phi}{\phi}]|^2
-(2k+1)\frac{\phi^{2k}}{\xi^k}\partial_\xi[\frac{\phi^{2k}}{\xi^{2k}}\partial_\xi
\phi]|\frac{\partial_t\phi}{\phi}|^2\label{vv_tvu}
\end{split}
\end{equation}
Next $VV^{\ast}VV^\ast (\xi^k u)$:
\begin{equation}
\begin{split}
\bullet\; VV^{\ast}VV^\ast (\xi^k
u)=-\frac{1}{\xi^k}\partial_\xi[\frac{\phi^{2k-1}}{\xi^{2k}}\partial_\xi
[\frac{\phi^{4k+2}}{\xi^{2k}}
\partial_\xi[\frac{\partial_t\phi}{\phi}]]]-(2k+1)\frac{1}{\xi^k}
\partial_\xi[\frac{\phi^{4k}}{\xi^{2k}}
\partial_\xi[\frac{\phi^{2k}}{\xi^{2k}}\partial_\xi\phi]
\frac{\partial_t\phi}{\phi}]\\
=\boxed{-\frac{1}{\xi^k}\partial_\xi[\frac{\phi^{2k-1}}{\xi^{2k}}\partial_\xi
[\frac{\phi^{4k+2}}{\xi^{2k}}
\partial_\xi[\frac{\partial_t\phi}{\phi}]]]}
+VV^\ast
V(\xi^k\phi)\frac{\partial_t\phi}{\phi}+\frac{\phi^{2k}}{\xi^{2k}}
{V^\ast V(\xi^k\phi)}\cdot\partial_\xi [\frac{\partial_t\phi}{\phi}]
\end{split}\label{V^4u}
\end{equation}
In turn, $VV^\ast V_tV^\ast (\xi^k u)$:
\begin{equation*}
\begin{split}
\frac{1}{2k}&VV^\ast V_tV^\ast (\xi^k
u)=-2\frac{\partial_t\phi}{\phi}\frac{1}{\xi^k}\partial_\xi[
\frac{\phi^{2k-1}}{\xi^{2k}}\partial_\xi
[\frac{\phi^{4k+2}}{\xi^{2k}}\partial_\xi[\frac{\partial_t\phi}{\phi}]]]\\
&-6\frac{\phi^{2k+1}}{\xi^{2k}}\partial_\xi[\frac{\partial_t\phi}{\phi}]\cdot
\frac{1}{\xi^k\phi^2}\partial_\xi
[\frac{\phi^{4k+2}}{\xi^{2k}}\partial_\xi[\frac{\partial_t\phi}{\phi}]]
+(4k+6)\partial_\xi\phi\frac{\phi^{6k}}{\xi^{5k}}|\partial_\xi
[\frac{\partial_t\phi}{\phi}]|^2\\&-(2k+1)|\frac{\partial_t\phi}{\phi}|^2
\cdot\frac{1}{\xi^k}\partial_\xi
[\frac{\phi^{4k}}{\xi^{2k}}\partial_\xi[\frac{\phi^{2k}}{\xi^{2k}}\partial_\xi\phi]]
-(4k+2)\frac{\partial_t\phi}{\phi}
\frac{\phi^{3k}}{\xi^{3k}}\partial_\xi[\frac{\phi^{2k}}{\xi^{2k}}\partial_\xi
\phi]\cdot\phi^k\partial_\xi[\frac{\partial_t\phi}{\phi}]
\end{split}
\end{equation*}
As in the case $\frac{1}{2}<k\leq 1$, we chose the above specific
expansion of $V^\ast VV^\ast(\xi^ku)$ in order to find the right
expression of the second derivative of $\partial_t\phi$ and then
estimate the mixed term $V^\ast V_tV^\ast(\xi^ku)$. This is because
the nonlinearity of \eqref{euler} is prevalent at this level in
$V,V^\ast$ formulation and an arbitrary expansion may destroy the
structure of \eqref{euler}. For higher order terms, we employ a
rather crude expansion in a systematic way. In the below, we present
a representation for $(V^\ast)^i(\xi^k u)$, get the information of
spatial derivatives of $\partial_t\phi$, which we denote by $T_i$'s,
and
derive the estimates of $(V^\ast)^{i-2} V_tV^\ast(\xi^ku)$. \\

\noindent{\textbf{Representation of $(V^\ast)^i(\xi^k u)$ and $T_i$
for $i\geq 2$}:} First, we define $T_i$ for $2\leq i\leq =\lceil
k\rceil+3$:
\begin{equation}\label{T_i}
 T_2\equiv
 \frac{\phi^{2k+1}}{\xi^{k}}\partial_\xi[\frac{\partial_t\phi}{\phi}];
\;\;T_3\equiv -\frac{1}{\xi^k\phi}\partial_\xi
[\frac{\phi^{4k+2}}{\xi^{2k}}\partial_\xi[\frac{\partial_t\phi}{\phi}]];\;\;T_i
\equiv (V)^{i-3}T_3\text{ for }i\geq 4
\end{equation}
 We chose $T_2$ and $T_3$ so that
\[
 (V^\ast)^2(\xi^k u)=T_2+ V(\xi^k\phi)\frac{\partial_t\phi}{\phi},\;\;
(V^\ast)^3(\xi^k u)=T_3+
(V)^2(\xi^k\phi)\frac{\partial_t\phi}{\phi}\,.
\]
In other words, they are boxed terms in the above as well as in the
previous sections.
 Note that by Lemma \ref{supk},
\begin{equation}
 \frac{T_2}{\xi^k}\in L_\xi^\infty,\;\;\frac{T_3}{\xi^{\lceil k\rceil
-1}}\in L_\xi^\infty\,.
\end{equation}
The validity of $T_i$ will follow from the construction of $T_i$ in the below.
Now we claim for any $i\geq 4$, $(V^\ast)^i(\xi^k u)$ has the following
representation which extends the case $i=2,3$:
\begin{equation}\label{expu}
 (V^\ast)^i(\xi^k u)=T_i+ (V)^{i-1}(\xi^k\phi)\frac{\partial_t\phi}{\phi}
+ \sum_{j=2}^{i-2}\Phi_{i-j} T_j
\end{equation}
where
\begin{equation}\label{Phi}
 \Phi_{i-j}\equiv \sum_{r=0}^{i-j-2} \{C_{r}
 \frac{1}{\xi^{r}}
 \sum_{\substack{l_1+\cdots+l_p=i-j-r\\l_1,\dots,
l_p\geq 1}} C_{l_1\cdots l_p} \prod_{q=1}^{p}
\frac{(V)^{l_q}(\xi^k\phi)}{\xi^k\phi}\}
\end{equation}
for some functions $C_{r}$, $C_{l_1\cdots l_p}$ which may only
depend on $\frac{\phi}{\xi}$ and $\frac{\xi}{\phi}$ and therefore
$C_{r}$, $C_{l_1\cdots l_p}$ are bounded by
$||\frac{\phi}{\xi}||_{L_\xi^\infty}$ and
$||\frac{\xi}{\phi}||_{L_\xi^\infty}$.
 Furthermore, $T_i$ and the last
term in \eqref{expu} have the following property: for each $4\leq
i\leq \lceil k\rceil +3$,
\begin{equation}
||\frac{1}{\xi^{\lceil k\rceil+3-i}}\sum_{j=2}^{i-2}\Phi_{i-j}
T_j||_{L_\xi^2}\text{ and }||\frac{T_i}{\xi^{\lceil
k\rceil+3-i}}||_{L_\xi^2} \text{ are bounded by
}\mathcal{E}^{k,\lceil k\rceil+3}(\phi,u)\,, \label{L^2prop}
\end{equation}
and for $4\leq i\leq \lceil k\rceil+1$,
\begin{equation}
||\frac{1}{\xi^{\lceil k\rceil+2-i}}\sum_{j=2}^{i-2}\Phi_{i-j}
T_j||_{L_\xi^\infty}\text{ and } ||\frac{T_i}{\xi^{\lceil
k\rceil+2-i}}||_{L_\xi^\infty}\text{ are bounded by
}\mathcal{E}^{k,\lceil k\rceil+3}(\phi,u).\label{prop}
\end{equation}
In particular, each $T_i$ is well-defined. We are now ready to prove
the representation formula \eqref{expu}.
 From the previous computation \eqref{V^4u},
\begin{equation}\label{v^4u2}
(V^\ast)^4(\xi^k u)=T_4+(V)^3(\xi^k\phi)\frac{\partial_t\phi}{\phi}
+\frac{(V)^2(\xi^k\phi)}{\xi^k\phi} T_2\,.
\end{equation}
Setting $\Phi_{4-2}=\frac{(V)^2(\xi^k\phi)}{\xi^k\phi}$, the formula \eqref{expu}
holds for $i=4$. Moreover, by Lemma \ref{supk}, the properties
\eqref{L^2prop} and \eqref{prop} are satisfied for $i=4$. For $i\geq
5$, based on the induction on $i$, we will show that $V$ or $V^\ast$
of each term in \eqref{expu} can be decomposed in the same fashion.
With the aid of Lemma \ref{product rule}, applying $V$ or $V^\ast$
of the second term $(V)^{i-1}(\xi^k\phi)\frac{\partial_t\phi}{\phi}$
in the right hand side of \eqref{expu}, we obtain
\[
\begin{split}
\bullet\;
&V((V)^{2j}(\xi^k\phi)\frac{\partial_t\phi}{\phi})=(V)^{2j+1}(\xi^k\phi)
\frac{\partial_t\phi}{\phi}+\frac{(V)^{2j}(\xi^k\phi)}{\xi^k\phi}T_2\\
\bullet\;&V^\ast((V)^{2j+1}(\xi^k\phi)\frac{\partial_t\phi}{\phi})=
(V)^{2j+2}(\xi^k\phi)
\frac{\partial_t\phi}{\phi}-\frac{(V)^{2j+1}(\xi^k\phi)}{\xi^k\phi}{T_2}
\end{split}
\]
Note that the right hand sides are of the right form. In the same
spirit, we can apply $V$ or $V^\ast$ to the last term in
\eqref{expu}. Here is the estimate of the simplest case when $j=2,\;
r=0,\; l_1=i-2$. Consider $i=2j+2$ or $i=2j+3$.
\[
\begin{split}
\bullet\; &V^\ast(\frac{(V)^{2j}(\xi^k\phi)}{\xi^k\phi}{T_2})=2
\frac{V(\xi^k\phi)}{\xi^k\phi}\frac{V^{2j}(\xi^k\phi)}{\xi^k\phi}{T_2}
-\frac{V^{2j+1}(\xi^k\phi)}{\xi^k\phi}{T_2}+
\frac{V^{2j}(\xi^k\phi)}{\xi^k\phi}{T_3}\\
\bullet\;
&V(\frac{(V)^{2j+1}(\xi^k\phi)}{\xi^k\phi}{T_2})=-\frac{2}{2k+1}
\frac{V(\xi^k\phi)}{\xi^k\phi}\frac{V^{2j+1}(\xi^k\phi)}{\xi^k\phi}{T_2}
-\frac{V^{2j+2}(\xi^k\phi)}{\xi^k\phi}{T_2}+
\frac{V^{2j+1}(\xi^k\phi)}{\xi^k\phi}{T_3}
\end{split}
\]
Each term in the right hand sides has a desirable form. From Lemma
\ref{product rule}, we deduce that when $V$ or $V^\ast$ act on a
function $h$, depending on that function, they can yield
$\frac{h}{\xi}$ or $ \frac{V(\xi^k\phi)}{\xi^k}\frac{h}{\phi}$ or
$Vh$ or $V^\ast h$. Therefore, $V$ or $V^\ast$ of other cases of
${\Phi_{i-j}T_j}$ falls into the right form of the case when $i+1$.
Next we verify \eqref{L^2prop} and \eqref{prop}. Let $s\geq 4$ be
given.
 Assume that they hold
for $i\leq s$ and we first claim that
\[
||\frac{1}{\xi^{\lceil k\rceil+3-(s+1)}}\sum_{j=2}^{(s+1)-2}
\Phi_{(s+1)-j} T_j||_{L_\xi^2}\text{ is bounded by }\mathcal{E}^{k,
\lceil k\rceil+3}(\phi,u)\,.
\]
This can be justified by counting derivatives and $\frac{1}{\xi}$ and
distributing $\frac{1}{\xi}$ factors in a right way. The highest derivative
term with appropriate factor of $\frac{1}{\xi}$ will be taken as the main
$L_\xi^2$ term and others will be bounded by taking the sup.
 Let
$z\equiv \max\{j,l_1,\dots,l_p\}$  for each term. If $z=j$, we consider
$\frac{T_j}{\xi^{\lceil k\rceil+3-j}}$ whose $L_\xi^2$ norm is bounded by $\mathcal{E}^{k,\lceil k\rceil+3}
(\phi,u)$ from the induction hypothesis, since $j\leq s-1$,
as an $L_\xi^2$ term. Now it remains to show that the rest of factors
\[
 \frac{\xi^{\lceil
 k\rceil+3-j}}{\xi^{\lceil k\rceil+3-(s+1)}}
\sum_{r=0}^{(s+1)-j-2}
 \frac{1}{\xi^{r}}
 \sum_{\substack{l_1+\cdots+l_p=(s+1)-j-r\\l_1,\dots,
l_p\geq 1}} \prod_{q=1}^{p} \frac{(V)^{l_q}(\xi^k\phi)}{\xi^k\phi}
\]
 are bounded.
  It is enough to look at the following: for each $r\leq s-j-1$
\[
 \frac{\xi^{s+1-j}}{\xi^{r}}
 \sum_{\substack{l_1+\cdots+l_p=s+1-j-r\\l_1,\dots,
 l_p\geq 1}} \prod_{q=1}^{p}
 \frac{(V)^{l_q}(\xi^k\phi)}{\xi^k\phi}
\]
which is clearly bounded by $\mathcal{E}^{k,\lceil
k\rceil+3}(\phi,u)$ because of Lemma \ref{supk} and since $k\leq
\lceil k\rceil$. When $z$ is one of $l_q$'s, one can derive the same
conclusion. Thus from \eqref{expu}, we deduce that
$||\frac{T_{s+1}}{\xi^{\lceil k\rceil+3-(s+1)}}||_{L_\xi^2}$ is also
bounded and it finishes the verification of \eqref{L^2prop}. For the
$L_\xi^\infty$ boundedness for $s+1\leq \lceil k\rceil+1$
\[
||\frac{1}{\xi^{\lceil k\rceil+2-(s+1)}}\sum_{j=2}^{(s+1)-2}
\Phi_{(s+1)-j} T_j||_{L_\xi^\infty}
\]
we employ the same counting argument except that we have to use the
fact $\frac{T_2}{\xi^k},\;\frac{T_j}{\xi^{\lceil k\rceil+2-j}}\in
L_\xi^\infty$ for $3\leq j \leq  s$. This finishes the induction
argument.

 These $T_i$'s and their properties will be  useful
 to estimate the nonlinear
terms involving $V_tV^\ast$. Based on the representation formula
\eqref{expu} and $T_i$'s, we further study the representation of
more general mixed terms $(V^\ast)^iV_t f$ for $i\leq \lceil k\rceil
+1$.
First $V_tf$ can be written as
\[
 \frac{1}{2k}V_tf=\frac{1}{\xi^k}\partial_\xi[\frac{\partial_t\phi}
{\phi}\frac{\phi^{2k}}{\xi^k}f]=\frac{\partial_t\phi} {\phi}Vf
+\frac{f}{\xi^k\phi} T_2\,.
\]
Note that the right hand side has the same structure as the last two
terms of the right hand side of \eqref{v^4u2}, in letting $f$ be
$(V)^2(\xi^k\phi)$. Thus we can apply the same technique to obtain
the following: for each $i\leq \lceil k\rceil+1$
\begin{equation}\label{gt}
 \frac{1}{2k}(V^\ast)^iV_tf=\frac{\partial_t\phi}
{\phi}(V)^{i+1}f+ \sum_{m=0}^i\{\sum_{j=2}^{m+2} \Phi_{(m+2)-j}T_j
\} (V)^{i-m}f\,.
\end{equation}

\begin{proposition} (Nonlinear estimates)
 The right hand sides in \eqref{i=2j+1} and \eqref{i=2j} are bounded
in $L_\xi^2$ by a continuous function of $\mathcal{E}^{k,\lceil
k\rceil+3}(\phi,u)$ and $C$.
\end{proposition}

\begin{proof} We only treat the most intriguing term
$(V^\ast)^iV_tV^\ast(\xi^ku)$ for $i\leq \lceil k\rceil+1$. The idea is to
estimate it in terms of $T_j$'s and to use their properties.
 First,
note that $V^\ast V_tV^\ast(\xi^ku)$ can be written in terms of
$T_2,T_3$ as in
\begin{equation}\label{V_t}
\frac{1}{2k}V^\ast
V_tV^\ast(\xi^ku)=2\frac{\partial_t\phi}{\phi}T_3-
2\frac{1}{\xi^k\phi}|T_2|^2+ V^\ast
V(\xi^k\phi)|\frac{\partial_t\phi}{\phi}|^2
\end{equation}
We would like to compute $(V)^iV^\ast V_tV^\ast(\xi^ku)$ for $1\leq
i\leq \lceil k\rceil$. Consider $i=1$.
We apply $V$ to each term in \eqref{V_t}. Since
\[
 \partial_\xi[\frac{\partial_t\phi}{\phi}]=\frac{\xi^k}{\phi^{2k+1}}T_2,
\]
we have
\[
V(\frac{\partial_t\phi}{\phi}T_3)=\frac{\partial_t\phi}{\phi}T_{4}
+\frac{\xi^k}{\phi^{2k+1}}T_2\cdot T_{3}
\]
\[
V(\frac{1}{\xi^k\phi}|T_2|^2)=\frac{2}{\xi^k\phi}T_2\cdot T_3-
\frac{4k+3}{2k+1}\frac{1}{\xi^k\phi}\frac{V(\xi^k\phi)}{\xi^k\phi}|T_2|^2
\]
\[
 V(V^\ast
V(\xi^k\phi)|\frac{\partial_t\phi}{\phi}|^2)=VV^\ast
V(\xi^k\phi)|\frac{\partial_t\phi}{\phi}|^2+ 2\frac{\xi^k}{\phi^{2k+1}}
T_2\cdot V^\ast
V(\xi^k\phi)\frac{\partial_t\phi}{\phi}
\]
Thus we deduce that $VV^\ast V_tV^\ast(\xi^ku)$ is bounded in
$L^2_\xi$ by the energy functional.
 In the same vein, by keeping applying $V^\ast$ and $V$,
for any $i\geq 2$, we can write
$(V)^iV^\ast V_tV^\ast(\xi^ku)$ as follows:
\begin{equation}\label{T}
\begin{split}
 \frac{1}{2k}(V)^iV^\ast V_tV^\ast(\xi^ku)=2\frac{\partial_t\phi}{\phi}
T_{i+3}+(V)^{i+2}(\xi^k \phi)|\frac{\partial_t\phi}{\phi}|^2
+\frac{\partial_t\phi}{\phi}\cdot\sum_{j=2}^{i+1}
\Phi_{(i+3)-j} T_j\\
+\sum_{j=2}^{i+2}C_j\frac{1}{\xi^k\phi}T_{i+4-j}T_j
+\sum_{s=4}^{i+2}\{ \Phi_{(i+4)-s}
(\sum_{j=2}^{s-2}C_{sj}\frac{1}{\xi^k\phi}T_{s-j}T_j)\}
\end{split}
\end{equation}
where $\Phi$ is given as in \eqref{Phi} with possibly different
coefficient functions and $C_j$ and $C_{sj}$ are some functions
bounded by $||\frac{\phi}{\xi}||_{L_\xi^\infty}$ and
$||\frac{\xi}{\phi}||_{L_\xi^\infty}$. This formula \eqref{T} can be also
obtained by plugging \eqref{expu} into \eqref{gt}.

 Now we claim that for each
$1\leq i\leq \lceil k\rceil$, $||(V)^iV^\ast
V_tV^\ast(\xi^ku)||_{L_\xi^2}$ is bounded by $\mathcal{E}^{k,\lceil
k\rceil+3}(\phi,u)$. The first three terms in the right hand side of
\eqref{T} are bounded since they are of the form as in \eqref{expu}
multiplied with $\frac{\partial_t\phi}{\phi}$. The rest terms
consist of quadratic or higher of energy terms or $T_j$'s. For each
term, the highest derivative with appropriate factor of
$\frac{1}{\xi}$ is considered the main $L_\xi^2$ term and other
factors are bounded by taking the sup. This can be done by employing
the counting and distributing $\frac{1}{\xi}$ argument as well as
the estimates of $T_j$'s \eqref{L^2prop} and \eqref{prop} as before.
\end{proof}

\section{Approximate Scheme}\label{5}

In this section, we implement the linear approximate scheme and
prove that the linear system is well-posed in some energy space.

 Let the initial data $\phi_{0}(\xi)$ and $u_{0}(\xi)$ of
 the Euler equations \eqref{euler} be given such that
$\frac{1}{C_0}\leq  \frac{\phi_{0}}{\xi}\leq C_0$ for a constant
$C_0>1$, and $\mathcal{E}^{k,\lceil k\rceil+3}(\phi_{0},u_{0})\leq
A$ for a constant $A>0$. Here $ \mathcal{E}^{k,\lceil
k\rceil+3}(\phi_0,u_0)$ is the energy functional \eqref{ef} induced
by $\phi_0$. Note that from the energy bound, we obtain
$\frac{\partial_\xi u_0}{\xi}\in L^\infty_\xi$.
 We will construct
 approximate solutions $\phi_n(t,\xi)$ and $u_n(t,\xi)$ for each
 $n$ by induction satisfying the following properties:
\begin{equation}\label{n}
  \phi_n|_{t=0}=\phi_{0},\;u_n|_{t=0}=u_{0};\;
 \phi_n|_{\xi=0}=u_n|_{\xi=1}=0;\;
 \frac{1}{C_n}\leq\frac{\phi_n}{\xi}\leq C_n,\text{ for }C_n>1
\end{equation}
 Note that $\phi_0,u_0$ automatically satisfy
\eqref{n}.
 Define the operators $V_n$ and $V_n^{\ast}$ as follows:
\begin{equation}
\begin{split}
V_n(f)\equiv \frac{1}{\xi^k}\partial_\xi
(\frac{\phi_n^{2k}}{\xi^k}f),\;\;\;
V_n^{\ast}(g)\equiv-\frac{\phi_n^{2k}}{\xi^k}\partial_\xi(
\frac{1}{\xi^k} g)
\end{split}\label{v_nk}
\end{equation}
 In addition,
we define commutator operators $(V_n)_t$ and $(V_n^\ast)_t$:
\[
(V_n)_t(f)\equiv
2k\frac{1}{\xi^k}\partial_\xi[\frac{\phi_n^{2k-1}\partial_t\phi_n}{\xi^k}f],\;\;
(V_n^\ast)_t(g)\equiv -2k\frac{\phi_n^{2k-1}\partial_t\phi_n}{\xi^k}
\partial_\xi[\frac{g}{\xi^k}]
\]
We define $\partial_t\phi_0$ through the equation by
\[
 \partial_t\phi_0\equiv -\frac{\phi_0^{2k}}{\xi^{2k}}\partial_\xi u_0\,.
\]
For the linear iteration scheme, we approximate $(V)^3(\xi^k\phi)\equiv G$ and
$(V^\ast)^3(\xi^ku)\equiv F$ and
 the equations
\eqref{i=2j+1} when $j=1$, instead of $\phi$ and $u$ themselves. Let
\[
D_0\equiv V_0(\xi^k\phi_0),\;H_0\equiv V^\ast_0(\xi^k u_0),\;G_0\equiv V_0V^\ast_0
V_0(\xi^k\phi_0),\;F_0\equiv V_0^\ast V_0V^\ast_0(\xi^k u_0).
\]
 For each $n\geq 0$, consider the
following approximate equations
\begin{equation}\label{FG}
\begin{split}
&\partial_tG_{n+1}
-(2k+1)V_n F_{n+1}=J_n^1\\
&\partial_tF_{n+1} +\frac{1}{2k+1}V_n^\ast G_{n+1}=J_n^2
\end{split}
\end{equation}
where  $J_n^1$ and
$J_n^2$ are given as follows:
\[
 \begin{split}
  J_n^1&\equiv 4k\frac{\partial_t\phi_n}{\phi_n}G_n+
4k\frac{\phi_n^{2k}}{\xi^k}
\partial_\xi[\frac{\partial_t\phi_n}{\phi_n}]
\frac{V_{n}^\ast D_n}{\xi^k}\\
J_n^2&\equiv 2k\frac{\partial_t\phi_n}{\phi_n}F_n+
4k|\frac{\partial_t\phi_n}{\phi_n}|^2V_{n}^\ast D_n
-8k\frac{\phi_n^{4k+1}}{\xi^{3k}}
|\partial_\xi[\frac{\partial_t\phi_n}{\phi_n}]|^2
-8k\frac{\partial_t \phi_n}{\phi_n}\frac{1}{\xi^k\phi_n}\partial_\xi
[\frac{\phi_n^{4k+2}}{\xi^{2k}}\partial_\xi[\frac{\partial_t\phi_n}{\phi_n}]]
 \end{split}
\]
 The initial and boundary conditions are inherited from the
original system: $$G_{n+1}|_{t=0}=V_0V^\ast_0 V_0(\xi^k\phi_0),\,
F_{n+1}|_{t=0}=V_0^\ast V_0V^\ast_0(\xi^k
u_0),\,G_{n+1}|_{\xi=1}=0\,.$$
Note that $F_{n+1}|_{\xi=0}$ is built in the equations \eqref{FG}.
 In turn, from these
$F_{n+1},\,G_{n+1}$, we define $D_{n+1},\;H_{n+1},\;\phi_{n+1},\;
u_{n+1}$:
\[
\begin{split}
D_{n+1}&\equiv
-\xi^k\int_1^\xi\frac{\xi'^{2k}}{\phi_n^{4k}}\int_0^{\xi'}
\xi_1^{k}G_{n+1}d\xi_1 d\xi',\;H_{n+1}\equiv
-\frac{\xi^k}{\phi_n^{2k}}
\int_0^\xi \xi'^k\int_1^{\xi'} \frac{\xi_1^k}{\phi_n^{2k}}F_{n+1}d\xi_1d\xi'\\
\phi_{n+1}&\equiv \{\int_0^\xi \xi'^k
D_{n+1}(t,\xi')d\xi'\}^{\frac{1}{2k+1}},\;
\partial_\xi
u_{n+1}\equiv-\frac{\xi^k H_{n+1}}{\phi^{2k}_{n+1}},\;u_{n+1}\equiv \int_1^\xi
\partial_\xi u_{n+1} d\xi'
\end{split}
\]
Note that we have used the boundary condition at $\xi=1$ in order to
invert $V^\ast_n$.  Also note that from the above definitions the
following identities hold
$$V_nV_n^\ast D_{n+1}=G_{n+1},\, V_n^\ast V_nH_{n+1}=F_{n+1},\,
D_{n+1}=V_{n+1}(\xi^k\phi_{n+1}),\,H_{n+1}= V_{n+1}^\ast(\xi^k
u_{n+1})\,.$$
 In view of Proposition \ref{EE}, it is easy to
deduce that $(V_n^\ast)^i D_{n+1}$ and $(V_n)^i H_{n+1}$ for $0\leq
i\leq \lceil k\rceil +2$ are well-defined, namely bounded in
$L^2_\xi$. Also $V_{n+1}$ and $V^\ast_{n+1}$ are defined as in
\eqref{v_nk} with $\phi_{n+1}$. The right hand sides $J_n^1,$
$J_n^2$ of \eqref{FG} are approximations of $2(V_n)_tV_n^\ast
V_n(\xi^k\phi_n)$ and $2V_n^\ast (V_n)_tV_n^\ast (\xi^k u_{n})
+(V_n^\ast)_tV_nV_n^\ast (\xi^k u_{n})$ in the following manner:
\[
\begin{split}
\frac{1}{2k}&(V_n)_tV_n^\ast V_n(\xi^k\phi_n)=-(2k+1)\frac{1}{\xi^k}\partial_\xi[
\frac{\phi_n^{2k-1}\partial_t\phi_n}{\xi^k}\frac{\phi_n^{2k}}{\xi^k}
\partial_\xi
[\frac{\phi_n^{2k}}{\xi^{2k}}\partial_\xi\phi_n]]\\
=&
-(2k+1)\frac{1}{\xi^k}\partial_\xi[\frac{\partial_t\phi_n}{\phi_n}\cdot
\frac{\phi_n^{4k}}{\xi^{2k}}\partial_\xi
[\frac{\phi_n^{2k}}{\xi^{2k}}\partial_\xi\phi_n]]
=\frac{\partial_t\phi_n}{\phi_n}V_nV_n^\ast V_n(\xi^k\phi_n)+
\frac{\phi_n^{2k}}{\xi^k}
\partial_\xi[\frac{\partial_t\phi_n}{\phi_n}]
\frac{V_{n}^\ast V_n(\xi^k\phi_n)}{\xi^k}\\
 \sim &\; \frac{\partial_t\phi_n}{\phi_n}G_n+ \frac{\phi_n^{2k}}{\xi^k}
\partial_\xi[\frac{\partial_t\phi_n}{\phi_n}]
\frac{V_{n}^\ast D_n}{\xi^k}
\end{split}
\]
\[
\begin{split}
\frac{1}{2k}&V_n^\ast (V_n)_tV_n^\ast (\xi^k u_{n})
\sim -2\frac{\partial_t \phi_n}{\phi_n}\frac{1}{\xi^k\phi_n}\partial_\xi
[\frac{\phi_n^{4k+2}}{\xi^{2k}}\partial_\xi[\frac{\partial_t\phi_n}{\phi_n}]]
-2\frac{\phi_n^{4k+1}}{\xi^{3k}}
|\partial_\xi[\frac{\partial_t\phi_n}{\phi_n}]|^2
+|\frac{\partial_t\phi_n}{\phi_n}|^2V_{n}^\ast D_n\\
\frac{1}{2k}&(V_n^\ast)_tV_nV_n^\ast (\xi^k u_{n})
=\frac{\partial_t\phi_n}{\phi_n}V_n^\ast V_nV_n^\ast(\xi^k u_n)\sim
\frac{\partial_t\phi_n}{\phi_n}F_n
\end{split}
\]
In particular, note that the approximation of $V_n^\ast
(V_n)_tV_n^\ast (\xi^k u_{n})$, which has the strongest nonlinearity
in view of the a priori estimates, is based on the expression
\eqref{vv_tvu}. Also note that the equation \eqref{FG} converges to
\eqref{i=2j+1} for $j=1$ in the formal limit.

 We define the approximate energy functional
$\widetilde{\mathcal{E}}^k_{n+1}$ at the $n$-th step:
\begin{equation}\label{aef}
\begin{split}
 \widetilde{\mathcal{E}}^k_{n+1}(t)&\equiv
 \sum_{i=0}^{\lceil k\rceil}\int \frac{1}{(2k+1)^2}|(V_n^\ast)^iG_{n+1}|^2
+|(V_n)^iF_{n+1}|^2d\xi\\
&\equiv ||\frac{1}{2k+1}G_{n+1}||^2_{Y_n^{k,\lceil k\rceil}}+
||F_{n+1}||^2_{X_n^{k,\lceil k\rceil}}
\end{split}
\end{equation}
where $X_n^{k,s}\text{ and } Y_n^{k,s}$ denote $X^{k,s}$ and
$Y^{k,s}$ induced by $V_n$ and $V_n^\ast$ in \eqref{XY}.

We now state and prove that the approximate system \eqref{FG} are
well-posed in the energy space generated by $X_n,$ $Y_n$ under the
following induction hypotheses:
\begin{equation*}
\begin{split}
\text{(HP1) } &\widetilde{\mathcal{E}}^k_{n}<\infty\text{ and }
\phi_n,\, u_{n}\text{ satisfy } \eqref{n}\,;\\
\text{(HP2) }&\text{when }k\leq 1,
||\frac{\partial_t\phi_n}{\phi_n}||_{L_\xi^\infty}\text{ and }
||\xi\partial_\xi[\frac{\partial_t\phi_n}{\phi_n}]||_{L_\xi^\infty}
\text{ are bounded by }\widetilde{\mathcal{E}}^k_{n},\,
\widetilde{\mathcal{E}}^k_{n-1};\\
&\text{when }k>1,
||\frac{\partial_t\phi_n}{\phi_n}||_{L_\xi^\infty},
||\frac{T_{2,n}}{\xi^k}||_{L_\xi^\infty},
||\frac{T_{i,n}}{\xi^{\lceil k\rceil +2-i}}||_{L_\xi^\infty},
||\frac{(V_n)^i(\xi^k\phi_n)}{\xi^{\lceil k\rceil
+2-i}}||_{L_\xi^\infty}\\
&\;\;\;\;\;\;\;\;\;\;\;\;\;\;\;\;\;\;\;\;\text{ for }3\leq i\leq
\lceil k\rceil+1\text{ are bounded by
}\widetilde{\mathcal{E}}^k_{n}, \,\widetilde{\mathcal{E}}^k_{n-1};\\
\text{(HP3) }& J_n^1\text{ and }J_n^2\text{ in }\eqref{FG}
 \text{ are bounded in }Y_n^{k,\lceil k\rceil},\, X_n^{k,\lceil k\rceil}
\text{ respectively.}
\end{split}
\end{equation*}
Here $T_{i,n}$ is the $T_i$ defined in \eqref{T_i} where $\phi$ is
taken as $\phi_n$.

\begin{proposition}\label{EE} (Well-posedness of approximate system and
regularity) Under the hypotheses (HP1), (HP2), and (HP3), the
 linear system \eqref{FG} admits a unique solution $(G_{n+1},H_{n+1})$ in
 $Y_n^{k,\lceil k\rceil}$, $X_n^{k,\lceil k\rceil}$ space. Furthermore, we
 obtain the following energy bounds.
\begin{equation*}
\widetilde{\mathcal{E}}^k_{n+1}(t)\leq
\widetilde{\mathcal{E}}^k_{n+1}(0)+\int_0^t
\mathcal{C}_5(\widetilde{\mathcal{E}}^k_{n-1},\widetilde{\mathcal{E}}^k_{n},
\widetilde{\mathcal{E}}^k_{n+1})
(\widetilde{\mathcal{E}}^k_{n+1})^{\frac{1}{2}}d\tau
\end{equation*}
where
$\mathcal{C}_5(\widetilde{\mathcal{E}}^k_{n-1},\widetilde{\mathcal{E}}^k_{n},
\widetilde{\mathcal{E}}^k_{n+1})$ is a continuous function of
$\widetilde{\mathcal{E}}^k_{n-1},\;
\widetilde{\mathcal{E}}^k_{n},\;\widetilde{\mathcal{E}}^k_{n+1}$ and
$C_0$.
\end{proposition}

Proposition \ref{EE} directly follows from Proposition 7.1. In the
next subsection, we verify the induction hypotheses.

\subsection{Induction procedure}

In order to finish the induction procedure of approximate schemes,
it now
 remains
to verify the induction hypotheses (HP1), (HP2), and (HP3) for $n+1$
as described in Proposition  \ref{EE}. The spirit is the same
as in the a priori estimates. However, since $n$ and $n+1$ are
mingled in the energy functional \eqref{aef}, the verification of
the induction hypotheses needs an attention. We only treat the case
when $k=1$. Other cases can be estimated similarly. First, it is
easy to see that
 $\phi_{n+1}$ and $u_{n+1}$ constructed in the above satisfy
the initial boundary conditions in \eqref{n} for $n+1$.  The
boundedness of $\frac{\phi_{n+1}}{\xi}$ will follow from  the
continuity argument by using the estimate of
$\frac{\partial_t\phi_{n+1}}{\xi}$. Write the equation for
$\partial_t(\phi_{n+1}^3)$ from the definition of $\phi_{n+1}$'s:
\begin{equation}
\begin{split}
 3\phi_{n+1}^2\partial_t\phi_{n+1}&=
4\int_0^\xi
(\xi')^2\int_1^{\xi'}\frac{\xi_2^2\partial_t\phi_{n}}{\phi_{n}^5}
\int_0^{\xi_2} \xi_1 G_{n+1}d\xi_1d\xi_2d\xi'\\
&\;\;\;-\int_0^\xi (\xi')^2\int_1^{\xi'}\frac{\xi_2^2}{\phi_{n}^4}
\int_0^{\xi_2} \xi_1 \partial_tG_{n+1}d\xi_1d\xi_2d\xi'\\
&=4\int_0^\xi
(\xi')^2\int_1^{\xi'}\frac{\xi_2^2\partial_t\phi_{n}}{\phi_{n}^5}
\int_0^{\xi_2} \xi_1 G_{n+1}d\xi_1d\xi_2d\xi'\\
&\;\;\;-3\int_0^\xi
(\xi')^2\int_1^{\xi'}\frac{\xi_2}{\phi_{n}^2}F_{n+1}
d\xi_2d\xi'-\int_0^\xi
(\xi')^2\int_1^{\xi'}\frac{\xi_2^2}{\phi_{n}^4} \int_0^{\xi_2} \xi_1
J_{n}^1d\xi_1d\xi_2d\xi'
\end{split}\label{dt}
\end{equation}
The first term can be controlled as follows: since
\begin{equation}
\int_0^{\xi_2} \xi_1 G_{n+1}d\xi_1\leq \xi_2^{\frac{5}{2}}
||\frac{G_{n+1}}{\xi}||_{L^2_\xi},\label{g1}
\end{equation}
it follows that
\[
 \int_0^\xi (\xi')^2\int_1^{\xi'}\frac{\xi_2^2\partial_t\phi_{n}}{\phi_{n}^5}
\int_0^{\xi_2} \xi_1 G_{n+1}d\xi_1d\xi_2d\xi'\leq \xi^{\frac{7}{2}}
||\frac{\xi}{\phi_{n}}||^4_{L^\infty_\xi}
||\frac{\partial_t\phi_{n}}{\phi_{n}}||_{L^\infty_\xi}
||\frac{G_{n+1}}{\xi}||_{L^2_\xi}\,.
\]
Thus $\frac{1}{\xi^3} \int_0^\xi
(\xi')^2\int_1^{\xi'}\frac{\xi_2^2\partial_t\phi_{n}}{\phi_{n}^5}
\int_0^{\xi_2} \xi_1 G_{n+1}d\xi_1d\xi_2d\xi'$ is bounded by
$\widetilde{\mathcal{E}}_{n+1}$ and the previous energies. For the
second term, note that
$$\int_1^{\xi'}\frac{\xi_2}{\phi_{n}^2}F_{n+1} d\xi_2\leq
||\frac{\xi}{\phi_{n}}||^2_{L^\infty_\xi}
||\frac{F_{n+1}}{\xi}||_{L^2_\xi}\,.$$ Hence, we obtain
\[
 \frac{1}{\xi^3}\int_0^\xi
(\xi')^2\int_1^{\xi'}\frac{\xi_2}{\phi_{n}^2}F_{n+1} d\xi_2d\xi'\leq
 ||\frac{\xi}{\phi_{n}}||^2_{L^\infty_\xi}
||\frac{F_{n+1}}{\xi}||_{L^2_\xi}\,.
\]
Finally, since
\begin{equation}
\begin{split}
\int_0^{\xi_2} \xi_1 J_{n}^1d\xi_1&\leq
\xi_2^{\frac{5}{2}}||\frac{\partial_t\phi_{n}}{\phi_{n}}||_{L_\xi^\infty}
||\frac{G_{n}}{\xi}||_{L^2_\xi}
\\&\;\;+\xi_2^{\frac{5}{2}}||\frac{\phi_n}{\xi}||^2_{L^\infty_\xi}
||\frac{\phi_n}{\phi_{n-1}}||^2_{L^\infty_\xi} ||\xi\partial_\xi
[\frac{\partial_t\phi_{n}}{\phi_{n}}]||_{L^\infty_\xi}
||\frac{V_{n-1}^\ast D_{n}}{\xi}||_{L^2_\xi}
\end{split}\label{g2}
\end{equation}
the last term is also bounded
 and therefore we conclude that
 $\partial_t(\frac{\phi_{n+1}^3}{\xi^3})$ is bounded by
 $\widetilde{\mathcal{E}}_{n+1}$
 and the previous energies. Note that the nontrivial contribution near
the boundary $\xi=0$
 is coming from the second term $F^{n+1}$. Since
\[
\frac{\phi_{n+1}^3}{\xi^3}(t)=\frac{\phi_{n+1}^3}{\xi^3}(0)+\int_0^t
\partial_t(\frac{\phi_{n+1}^3}{\xi^3})d\tau=\frac{\phi_{0}^3}{\xi^3}(0)+\int_0^t
\partial_t(\frac{\phi_{n+1}^3}{\xi^3})d\tau\,,
\]
we get
\[
 \frac{1}{C_0^3}-
 T||\partial_t(\frac{\phi_{n+1}^3}{\xi^3})||_{L^\infty_\xi}
\leq \frac{\phi_{n+1}^3}{\xi^3}\leq C_0^3+
 T||\partial_t(\frac{\phi_{n+1}^3}{\xi^3})||_{L^\infty_\xi}
\]
and in result, we also obtain  $C_{n+1}$ in \eqref{n} depending on
$T$ and energy bounds.

We move onto (HP2). Since $\frac{\partial_t\phi_{n+1}}{\phi_{n+1}}=
\frac{\partial_t\phi_{n+1}}{\xi}\frac{\xi}{\phi_{n+1}}$, we
immediately deduce that
$||\frac{\partial_t\phi_{n+1}}{\phi_{n+1}}||_{L_\xi^\infty}$ is
bounded.  Take $\partial_\xi $ of \eqref{dt}:
\begin{equation}
\begin{split}
 3\phi_{n+1}^3\partial_\xi [\frac{\partial_t\phi_{n+1}}{\phi_{n+1}}] +
9\phi_{n+1}\partial_\xi\phi_{n+1}\partial_t\phi_{n+1}=
4\xi^2\int_1^{\xi}\frac{\xi_2^2\partial_t\phi_{n}}{\phi_{n}^5}
\int_0^{\xi_2} \xi_1 G_{n+1}d\xi_1d\xi_2\\
-3\xi^2\int_1^{\xi}\frac{\xi_2}{\phi_{n}^2}F_{n+1}
d\xi_2-\xi^2\int_1^{\xi}\frac{\xi_2^2}{\phi_{n}^4} \int_0^{\xi_2}
\xi_1 J_{n}^1d\xi_1d\xi_2
\end{split}\label{ddt}
\end{equation}
Therefore the boundedness of
$||\xi\partial_\xi
[\frac{\partial_t\phi_{n+1}}{\phi_{n+1}}]||_{L_\xi^\infty}$
follows
by the same reasoning.

For (HP3), we now show that $J_{n+1}^1$ and $J_{n+1}^2$ of
\eqref{FG} at the next step are bounded in $Y_{n+1}^{1,1},$
$X_{n+1}^{1,1}$ accordingly.

\begin{claim} $V_{n+1}^\ast J_{n+1}^1$ and $V_{n+1}J_{n+1}^2$ are
bounded in $L_\xi^2$.
\end{claim}

\begin{proof} The sprit of the proof is same as in the nonlinear estimates of
 the a priori
estimates.  We present the detailed computation for $J_{n+1}^2$,
which is more complicated than $J_{n+1}^1$.
 We start with $V_{n+1}(\frac{\partial_t\phi_{n+1}}{\phi_{n+1}}F_{n+1})$.
\[
\begin{split}
V_{n+1}(\frac{\partial_t\phi_{n+1}}{\phi_{n+1}}F_{n+1})&=\frac{1}{\xi}
\partial_\xi[\frac{\phi_{n+1}^2}{\xi}\frac{\partial_t\phi_{n+1}}{\phi_{n+1}}F_{n+1}]
=\frac{1}{\xi}
\partial_\xi[\frac{\phi_{n+1}^2}{\phi_n^2}\frac{\partial_t\phi_{n+1}}{\phi_{n+1}}
\frac{\phi_{n}^2}{\xi}F_{n+1}]\\
&=\underbrace{\frac{\phi_{n+1}^2}{\phi_n^2}\frac{\partial_t\phi_{n+1}}{\phi_{n+1}}}
_{L^\infty_\xi}\underbrace{ V_nF_{n+1}}_{L^2_\xi}
+ \underbrace{\frac{\phi_{n+1}^2}{\xi^2}
\xi\partial_\xi[\frac{\partial_t\phi_{n+1}}{\phi_{n+1}}]}_{L^\infty_\xi}
\underbrace{\frac{F_{n+1}}{\xi}}_{L^2_\xi}\\
&\;\;+\underbrace{2(\frac{\phi_{n+1}}{\xi}\partial_\xi\phi_{n+1}
-\frac{\phi_{n+1}^2}{\xi\phi_n}\partial_\xi\phi_{n})
\frac{\partial_t\phi_{n+1}}{\phi_{n+1}}}_{L^\infty_\xi}
\underbrace{\frac{F_{n+1}}{\xi}}_{L^2_\xi}
\end{split}
\]
Thus
$||V_{n+1}(\frac{\partial_t\phi_{n+1}}{\phi_{n+1}}F_{n+1})||_{L^2_\xi}$
is bounded by $\widetilde{\mathcal{E}}_{n+1}$. Next we compute
$V_{n+1} (|\frac{\partial_t\phi_{n+1}}{\phi_{n+1}}|^2V_{n+1}^\ast
D_{n+1})$.
\[
 \begin{split}
 V_{n+1}
(|\frac{\partial_t\phi_{n+1}}{\phi_{n+1}}|^2V_{n+1}^\ast D_{n+1})&=
- \frac{1}{\xi}\partial_\xi[\frac{\phi_{n+1}^4}{\xi^2}\partial_\xi
[\frac{D_{n+1}}{\xi}]|\frac{\partial_t\phi_{n+1}}{\phi_{n+1}}|^2]\\
&=\underbrace{\frac{\phi_{n+1}^4}{\phi_n^4}|\frac{\partial_t\phi_{n+1}}
{\phi_{n+1}}|^2}_{L^\infty_\xi}\underbrace{G_{n+1}}_{L^2_\xi}
+\underbrace{\frac{\phi_{n+1}^4}{\xi^2\phi_n^2}
\frac{\partial_t\phi_{n+1}}{\phi_{n+1}}\xi\partial_\xi
[\frac{\partial_t\phi_{n+1}}{\phi_{n+1}}]}_{L^\infty_\xi}
\underbrace{\frac{V_n^\ast D_{n+1}}{\xi}}_{L^2_\xi}\\
&\;\;+ \underbrace{4 (\frac{\phi_{n+1}^3}{\xi\phi_n^2}
\partial_\xi\phi_{n+1}-\frac{\phi_{n+1}^4}{\xi\phi_n^3}\partial_\xi\phi_n)
|\frac{\partial_t\phi_{n+1}}{\phi_{n+1}}|^2}_{L^\infty_\xi}
 \underbrace{\frac{V_n^\ast D_{n+1}}{\xi}}_{L^2_\xi}
 \end{split}
\]
Thus $V_{n+1}
(|\frac{\partial_t\phi_{n+1}}{\phi_{n+1}}|^2V_{n+1}^\ast D_{n+1})$
is bounded in $L^2_\xi$. In order to take care of the rest of
$J_{n+1}^2$, first we claim that
\begin{equation}
||\frac{1}{\xi^3}\partial_\xi[\frac{\phi_{n+1}^6}{\xi^2}\partial_\xi
[\frac{\partial_t\phi_{n+1}}{\phi_{n+1}}]]||_{L^2_\xi}\text{ is
bounded by
}\widetilde{\mathcal{E}}_{n+1},\,\widetilde{\mathcal{E}}_{n}\,.\label{claim}
\end{equation}
In order to do so, multiply \eqref{ddt} by
$\frac{\phi_{n+1}^3}{3\xi^2}$ and take $\partial_\xi$ to get
\[
\begin{split}
&\partial_\xi[\frac{\phi_{n+1}^6}{\xi^2}\partial_\xi
[\frac{\partial_t\phi_{n+1}}{\phi_{n+1}}]]+3\partial_\xi[\phi_{n+1}^3]
\frac{\phi_{n+1}^2}{\xi^2}\partial_\xi\phi_{n+1}
\frac{\partial_t\phi_{n+1}}{\phi_{n+1}}\\&\;\;+3\phi_{n+1}^3\partial_\xi
[\frac{\phi_{n+1}^2}{\xi^2}\partial_\xi\phi_{n+1}]
\frac{\partial_t\phi_{n+1}}{\phi_{n+1}} +\underline{3\phi_{n+1}^3
\frac{\phi_{n+1}^2}{\xi^2}\partial_\xi\phi_{n+1}\partial_\xi
[\frac{\partial_t\phi_{n+1}}{\phi_{n+1}}]}_{(\star)}\\&=
\phi_{n+1}^2\partial_\xi\phi_{n+1}\{4\int_1^\xi Ad\xi'-3\int_1^\xi
Bd\xi'-\int_1^\xi Cd\xi'\}+ \frac{\phi_{n+1}^3}{3}\{4A-3B-C\}
\end{split}
\]
where $A,B,C$ are defined so that  \eqref{ddt} can be written as
$$3\phi_{n+1}^3\partial_\xi [\frac{\partial_t\phi_{n+1}}{\phi_{n+1}}] +
9\phi_{n+1}\partial_\xi\phi_{n+1}\partial_t\phi_{n+1}=\xi^2\{
4\int_1^\xi Ad\xi'-3\int_1^\xi Bd\xi'-\int_1^\xi Cd\xi'\}\,.$$ The
term $(\star)$ becomes
\[
(\star)=\frac{\phi_{n+1}^2}{\xi^2}\partial_\xi\phi_{n+1}\{ -9
\phi_{n+1}\partial_\xi\phi_{n+1}\partial_t\phi_{n+1}+\xi^2(4\int_1^\xi
Ad\xi'-3\int_1^\xi Bd\xi'-\int_1^\xi Cd\xi')\}
\]
Hence by canceling terms, we obtain
\begin{equation}
\begin{split}
\partial_\xi[\frac{\phi_{n+1}^6}{\xi^2}\partial_\xi
[\frac{\partial_t\phi_{n+1}}{\phi_{n+1}}]]&=-
3\phi_{n+1}^3\partial_\xi
[\frac{\phi_{n+1}^2}{\xi^2}\partial_\xi\phi_{n+1}]
\frac{\partial_t\phi_{n+1}}{\phi_{n+1}}+\frac{\phi_{n+1}^3}{3}\{4A-3B-C\}\\
&=\phi_{n+1}^3\frac{\xi^2}{\phi_n^2}\frac{V_{n}^\ast D_{n+1}}{\xi} +
\frac{\phi_{n+1}^3}{3}\{4A-3B-C\}
\end{split}\label{dddt}
\end{equation}
Indeed, $A,B,C$ were treated during the estimates of
$||\frac{\partial_t\phi_{n+1}}{\xi}||_{L^\infty_\xi}$. From
the same analysis,
 we deduce that $L^2_\xi$ norms of $A,B,C$ are bounded by
$\widetilde{\mathcal{E}}_{n+1},\,\widetilde{\mathcal{E}}_{n}$. Thus
the claim \eqref{claim} follows.
 Now we consider $V_{n+1}(\frac{\phi_{n+1}^5}{\xi^3}
|\partial_\xi [\frac{\partial_t\phi_{n+1}}{\phi_{n+1}}]|^2)$.
\[
\begin{split}
& V_{n+1}(\frac{\phi_{n+1}^5}{\xi^3} |\partial_\xi
[\frac{\partial_t\phi_{n+1}}{\phi_{n+1}}]|^2)=\frac{1}{\xi}\partial_\xi
[\frac{1}{\phi_{n+1}^5}|\frac{\phi_{n+1}^6}{\xi^2}\partial_\xi
[\frac{\partial_t\phi_{n+1}}{\phi_{n+1}}]|^2]\\
&=2\phi_{n+1}\partial_\xi
[\frac{\partial_t\phi_{n+1}}{\phi_{n+1}}]\cdot
\frac{1}{\xi^3}\partial_\xi[\frac{\phi_{n+1}^6}{\xi^2}\partial_\xi
[\frac{\partial_t\phi_{n+1}}{\phi_{n+1}}]] -5
\frac{\partial_\xi\phi_{n+1}}{\xi\phi_{n+1}^6}
|\frac{\phi_{n+1}^6}{\xi^2}\partial_\xi
[\frac{\partial_t\phi_{n+1}}{\phi_{n+1}}]|^2
\end{split}
\]
By \eqref{claim}, the first term in the right hand side is
controllable, and again due to \eqref{claim}, since
$|\frac{\phi_{n+1}^6}{\xi^2}\partial_\xi
[\frac{\partial_t\phi_{n+1}}{\phi_{n+1}}]|^2\leq\xi^7
||\frac{1}{\xi^3}\partial_\xi[\frac{\phi_{n+1}^6}{\xi^2}\partial_\xi
[\frac{\partial_t\phi_{n+1}}{\phi_{n+1}}]]||_{L^2_\xi}^2$, we obtain
\[
||\frac{1}{\xi^7}|\frac{\phi_{n+1}^6}{\xi^2}\partial_\xi
[\frac{\partial_t\phi_{n+1}}{\phi_{n+1}}]|^2||_{L^2_\xi}\leq
||\frac{1}{\xi^3}\partial_\xi[\frac{\phi_{n+1}^6}{\xi^2}\partial_\xi
[\frac{\partial_t\phi_{n+1}}{\phi_{n+1}}]]||_{L^2_\xi}^2
\]
and this completes the estimate. For the last term in $J_{n+1}^2$,
note that
\[
\begin{split}
&V_{n+1}(\frac{\partial_t\phi_{n+1}}{\phi_{n+1}}\frac{1}{\xi\phi_{n+1}}
\partial_\xi[\frac{\phi_{n+1}^6}{\xi^2}\partial_\xi[\frac{\partial_t\phi_{n+1}}
{\phi_{n+1}}]])\\
&=\frac{\partial_t\phi_{n+1}}{\phi_{n+1}}\frac{1}{\xi}\partial_\xi
[\frac{\phi_{n+1}}{\xi^2}\partial_\xi[\frac{\phi_{n+1}^6}{\xi^2}
\partial_\xi[\frac{\partial_t\phi_{n+1}}
{\phi_{n+1}}]]]+\phi_{n+1}\partial_\xi[\frac{\partial_t\phi_{n+1}}{\phi_{n+1}}]
\frac{1}{\xi^3}\partial_\xi[\frac{\phi_{n+1}^6}{\xi^2}
\partial_\xi[\frac{\partial_t\phi_{n+1}}
{\phi_{n+1}}]]
\end{split}
\]
Thus it remains to show that $\frac{1}{\xi}\partial_\xi
[\frac{\phi_{n+1}}{\xi^2}\partial_\xi[\frac{\phi_{n+1}^6}{\xi^2}
\partial_\xi[\frac{\partial_t\phi_{n+1}}
{\phi_{n+1}}]]]$ is bounded by
$\widetilde{\mathcal{E}}_{n+1},\,\widetilde{\mathcal{E}}_{n}$. From
\eqref{dddt}, one can write it as
\[
\begin{split}
\frac{1}{\xi}\partial_\xi
[\frac{\phi_{n+1}}{\xi^2}\partial_\xi[\frac{\phi_{n+1}^6}{\xi^2}
\partial_\xi[\frac{\partial_t\phi_{n+1}}
{\phi_{n+1}}]]]&=\frac{1}{\xi}\partial_\xi[\frac{\phi_{n+1}^4}{\phi_n^4}\cdot
\frac{\phi_n^2}{\xi}V_n^\ast
D_{n+1}]+\frac{1}{3\xi}\partial_\xi[\frac{\phi_{n+1}^4}{\xi^2}
(4A-3B-C)]\\
&\equiv (I)+(II)
\end{split}
\]
$(I)$ can be decomposed as follows:
\[
(I)=\frac{\phi_{n+1}^4}{\phi_n^4}G_{n+1}+
\xi\partial_\xi[\frac{\phi_{n+1}^4}{\phi_n^4}]\frac{\phi_n^2}{\xi^2}
\frac{V_n^\ast D_{n+1}}{\xi}
\]
hence it is bounded. For $(II)$, we put $A,B,C$ back into the
expression:
\[
\begin{split}
3\,(II)&=\frac{1}{\xi}\partial_\xi[\frac{\phi_{n+1}^4}{\phi_n^4}\{4
\frac{\partial_t\phi_n}{\phi_n}\int_0^\xi\xi_1G_{n+1}d\xi_1-3
\frac{\phi_n^2}{\xi}F_{n+1}-\int_0^\xi \xi_1J_n^1d\xi_1\}]\\
&=\frac{\phi_{n+1}^4}{\phi_n^4}\{4\frac{\partial_t\phi_n}{\phi_n}
G_{n+1}+\xi\partial_\xi[\frac{\partial_t\phi_n}{\phi_n}]\cdot\frac{1}{\xi^2}
\int_0^\xi\xi_1G_{n+1}d\xi_1-3V_nF_{n+1}-J_n^1
\}\\&\;\;+\xi\partial_\xi[\frac{\phi_{n+1}^4}{\phi_n^4}]
\{4\frac{\partial_t\phi_n}{\phi_n}\frac{1}{\xi^2}
\int_0^\xi\xi_1G_{n+1}d\xi_1-3\frac{\phi_n^2}{\xi^2}\frac{F_{n+1}}{\xi}
-\frac{1}{\xi^2}\int_0^\xi \xi_1J_n^1d\xi_1\}
\end{split}
\]
Therefore, by the estimates \eqref{g1} and  \eqref{g2} of
$\int_0^\xi\xi_1G_{n+1}d\xi_1$ and $\int_0^\xi \xi_1J_n^1d\xi_1$, we
conclude that $(II)$ is also bounded by
$\widetilde{\mathcal{E}}_{n+1},\,\widetilde{\mathcal{E}}_{n}$. This
finishes the proof of Claim.
\end{proof}

\section{Proof of Theorem \ref{thm}}\label{6}

In order to prove Theorem \ref{thm}, it now remains to show that
$\phi_n,\,u_n$ converge, the limit functions solve Euler equations
\eqref{euler}, and they are unique.

\subsection{Existence}

First, by applying Gronwall inequality to the energy inequality in
 Proposition \ref{EE}, we can deduce the
following claim.

\begin{claim}\label{claim2} Suppose that
 the initial data $\phi_{0}(\xi)$ and $u_{0}(\xi)$ of
 the Euler equations \eqref{euler} be given such that
$\frac{1}{C_0}\leq \frac{\phi_{0}}{\xi}\leq C_0$ for a constant
$C_0>1$, and $\mathcal{E}^{k,\lceil k\rceil+3}(\phi_{0},u_{0})\leq
A$ for a constant $A>0$. Then there exists $T>0$ such that if for
$m\leq n$, $\widetilde{\mathcal{E}}^k_{m}\leq \frac{3}{2}A$ for
$t\leq T$, then $\widetilde{\mathcal{E}}^k_{n+1}\leq \frac{3}{2}A$
for $t\leq T$ and in addition, for all $n$, $\frac{1}{2C_0}\leq
\frac{\phi_{n}}{\xi}\leq 2C_0$ for $t\leq T$.
\end{claim}

Thus we get the uniform bound of $\widetilde{\mathcal{E}}^k_{n}$'s
as well as the uniform bound of the upper and lower bounds of
$\frac{\phi_n}{\xi}$.  Since the approximate energy functionals
\eqref{aef} depend on not only the approximate functions
$G_{n+1},\,F_{n+1}$ but also $\phi_n$, the corresponding Banach
space changes at every step. In order to take the limit, it is
desirable to have the fixed space where $G_{n+1},\,F_{n+1}$ live.
The plan is as follows: by making use of Proposition \ref{equiv e},
we prove the equivalence between our energy functionals and the
energy functional induced by \eqref{ho} so that approximate
functions have the uniform energy bounds in the Banach space induced
by \eqref{ho} and thus we can apply the fixed point theorem.  We
present the detailed analysis for $k=1$.

 Recall the homogeneous operators $\overline{V}$
and $\overline{V}^\ast$:
\begin{equation*}
\begin{split}
\overline{V}(f)\equiv \frac{1}{\xi}\partial_\xi [\xi f],\;\;\;
\overline{V}^\ast(g)\equiv-\xi\partial_\xi[ \frac{g}{\xi}]
\end{split}
\end{equation*}
and we define the corresponding energy functional
$\overline{\mathcal{E}}_{n+1}$:
\begin{equation}
 \overline{\mathcal{E}}_{n+1}(t)\equiv
 \sum_{i=0}^1\int \frac{1}{9}|(\overline{V}^\ast)^iG_{n+1}|^2
+|(\overline{V})^iF_{n+1}|^2d\xi
\end{equation}
We claim that $\widetilde{\mathcal{E}}_{n+1}$ and
$\overline{\mathcal{E}}_{n+1}$ are equivalent in some sense.
 We compute $\widetilde{\mathcal{E}}_{n+1}$.
 $V_n$ and $V_n^\ast$ are written in terms of $\overline{V}$ and
$\overline{V}^\ast$.
\[
\begin{split}
\bullet &\;V_n^\ast G_{n+1}=\frac{\phi_n^2}{\xi^2}\overline{V}^\ast G_{n+1}\\
\bullet &\;V_nF_{n+1}=\frac{\phi_n^2}{\xi^2}\overline{V}F_{n+1}+
2(\frac{\phi_n}{\xi}\partial_\xi\phi_n-\frac{\phi_n^2}{\xi^2})
\frac{F_{n+1}}{\xi}
\end{split}
\]
Note that from the definition of $D_n$
\[
\begin{split}
|\frac{\phi_n^2}{\xi^2}\partial_\xi\phi_n|=
|\frac{1}{3}\frac{D_n}{\xi}|\leq
||\frac{\phi_{n-1}^2}{\xi^2}\partial_\xi\phi_{n-1}||_{L^\infty_\xi}
+\frac{2}{3^{\frac{3}{2}}}\xi^{\frac{1}{2}}
C_{n-1}^4||G_n||_{L^2_\xi}\leq I+
C_{n-1}^4||G_n||_{L^2_\xi}
\end{split}
\]
where
$I=||\frac{\phi_{0}^2}{\xi^2}\partial_\xi\phi_{0}||_{L^\infty_\xi}$
and $C_{n-1}$ is the bound of $\frac{\phi_{n-1}}{\xi}$ as in
\eqref{n}. Therefore, we deduce that
\[
\widetilde{\mathcal{E}}_{n+1}\leq
(1+M_n)\overline{\mathcal{E}}_{n+1}
\]
where $M_n\equiv 5
C_n^2\{C_n^2+I^2+C_{n-1}^8\overline{\mathcal{E}}_n\}$.  To show the
the converse, namely $\overline{\mathcal{E}}_{n+1}$ is bounded by
$\widetilde{\mathcal{E}}_{n+1}$, we rewrite $\overline{V}$ and
$\overline{V}^\ast$ in terms of $V^n$ and $V_n^\ast$.
\[
\begin{split}
\bullet &\;\overline{V}^\ast G_{n+1}=\frac{\xi^2}{\phi_n^2}V_n^\ast G_{n+1}\\
\bullet &\;\overline{V}F_{n+1}=\frac{\xi^2}{\phi_n^2}V_nF_{n+1}+
2(1-\frac{\xi}{\phi_n}\partial_\xi\phi_n)
\frac{F_{n+1}}{\xi}
\end{split}
\]
Thus we reach the same conclusion:
\[
\overline{\mathcal{E}}_{n+1}\leq
(1+M_n)\widetilde{\mathcal{E}}_{n+1}
\]
Note that $M_n$'s have the uniform bound over $t\leq T$ by Claim
\ref{claim2}. Therefore, there exists a sequence $n_l$ so that
$G_{n_l},$ $F_{n_l},$ $\phi_{n_l},$ $u_{n_l}$ converge strongly to
some $G,\;F,\;\phi,\;u$. Due to the uniform energy bound, we also
conclude that $G,\;F$ solve the equations \eqref{i=2j+1} for $j=1$
and moreover $\phi,\;u$ solve Euler equations \eqref{euler} with the
desired properties.

Next, we turn to the general case. Back to the approximate system \eqref{FG},
we define the corresponding
homogeneous energy
functional:
\begin{equation}
 \overline{\mathcal{E}}_{n+1}^k(t)\equiv
 ||\frac{1}{2k+1}G_{n+1}||^2_{\overline{Y}^{k,\lceil k\rceil}}
 +||F_{n+1}||^2_{\overline{X}^{k,\lceil k\rceil}}
\label{lefk}
\end{equation}
From Proposition \ref{equiv e},
 one can derive the equivalence of the associated
energy functional \eqref{lefk} and the original approximate
functional \eqref{aef}:  There exists $M_n>0$ only depending on
initial data, $C_n$, $C_{n-1}$, and $\overline{\mathcal{E}}_n$ such
that
\[
 \frac{1}{1+M_n}\overline{\mathcal{E}}_{n+1}^k\leq
\widetilde{\mathcal{E}}_{n+1}^k
 \leq (1+ M_n)
\overline{\mathcal{E}}_{n+1}^k.
\]
By the same reasoning as the case when $k=1$, the existence of
$G,\;F,\;\phi,\;u$ follows.

\subsection{Uniqueness}

In order to prove Theorem \ref{thm}, it only remains to prove the
uniqueness.  Let $(\phi_1,u_1)$ and $(\phi_2,u_2)$ be two regular
solutions to \eqref{euler} with the same initial boundary conditions
with $\mathcal{E}^{k,\lceil
k\rceil+3}(\phi_1,u_1),\,\mathcal{E}^{k,\lceil
k\rceil+3}(\phi_2,u_2)\leq 2A$. Define $\mathcal{D}(t)$ by
\[
\begin{split}
\mathcal{D}(t)\equiv& \int \frac{1}{2k+1}|\xi^k(\phi_1-\phi_2)|^2
+|\xi^k(u_1-u_2)|^2d\xi \\+&\sum_{i=1}^2\int
\frac{1}{(2k+1)^2}|(V_1)^i(\xi^k\phi_1)-(V_2)^i(\xi^k\phi_2)|^2
+|(V_1^\ast)^i(\xi^ku_1)-(V_2^\ast)^i(\xi^ku_2)|^2d\xi\,,
\end{split}
\]
where $V_j,V_j^\ast$ are the $V,V^\ast$ induced by $\phi_j$. We will
prove that $\frac{d}{dt}\mathcal{D}\leq \mathcal{C}\mathcal{D}$,
which immediately gives the uniqueness. Let us consider $i=2$ case
only in $\mathcal{D}(t)$. Recall the system \eqref{i=2}.  By
subtracting two systems from each other, we obtain the equations for
$(V_1^\ast V_1(\xi^k\phi_1)-V_2^\ast V_2(\xi^k\phi_2),\,V_1V_1^\ast
(\xi^ku_1)-V_2V_2^\ast(\xi^ku_2))$:
\[
\begin{split}
&\partial_t\{V_1^\ast V_1(\xi^k\phi_1)-V_2^\ast
V_2(\xi^k\phi_2)\}-(2k+1)V_1^\ast \{V_1V_1^\ast
(\xi^ku_1)-V_2V_2^\ast(\xi^ku_2)\}\\
&=(2k+1)\{V_1^\ast V_2V_2^\ast(\xi^ku_2)-V_2^\ast
V_2V_2^\ast(\xi^ku_2)\}+\{(V_1^\ast)_tV_1(\xi^k\phi_1)-
(V_2^\ast)_tV_2(\xi^k\phi_2)\}
\\
&\partial_t\{V_1V_1^\ast
(\xi^ku_1)-V_2V_2^\ast(\xi^ku_2)\}+\frac{1}{2k+1}V_1 \{V_1^\ast
V_1(\xi^k\phi_1)-V_2^\ast V_2(\xi^k\phi_2)\}\\
&=\frac{1}{2k+1}\{\underbrace{V_2V_2^\ast
V_2(\xi^k\phi_2)-V_1V_2^\ast
V_2(\xi^k\phi_2)}_{(I)}\}+2\{\underbrace{(V_1)_tV_1^\ast(\xi^ku_1)-
(V_2)_tV_2^\ast(\xi^ku_2)}_{(II)}\}
\end{split}
\]
We have to show that the $L^2_\xi$ norm of the right hand sides is
bounded by $\mathcal{D}^{\frac{1}{2}}$. We estimate $(I)$ and
$(II)$. Other two terms can be treated in the same way. First, we
note that
\[
\phi_1^{2k+1}-\phi_2^{2k+1}=\int_0^\xi(\xi')^k\{V_1(\xi^k\phi_1)
-V_2(\xi^k\phi_2)\} d\xi'\,.
\]
If simply applying H$\ddot{\text{o}}$lder inequality, one gets
\begin{equation}\label{b1}
||\frac{\phi_1^{2k+1}}{\xi^{k+\frac{1}{2}}}-
\frac{\phi_2^{2k+1}}{\xi^{k+\frac{1}{2}}}||_{L^\infty_\xi}\leq
||V_1(\xi^k\phi_1) -V_2(\xi^k\phi_2)||_{L^2_\xi}\,.
\end{equation}
Apply H$\ddot{\text{o}}$lder inequality once more:
\[
\begin{split}
\int_0^\xi(\xi')^{k+\frac{1}{2}}\frac{V_1(\xi^k\phi_1)
-V_2(\xi^k\phi_2)}{\sqrt{\xi'}} d\xi'\leq \xi^{k+1}&\{
||V_1(\xi^k\phi_1) -V_2(\xi^k\phi_2)||_{L^2_\xi}\\
&+\underbrace{||\sqrt{\xi}V_1^\ast(V_1(\xi^k\phi_1)
-V_2(\xi^k\phi_2))||_{L^2_\xi}}_{(\ast\ast)}\}
\end{split}
\]
Note that it is not trivial to get the boundedness of $(\ast\ast)$
in terms of $\mathcal{D}^{\frac{1}{2}}$, since it is not of the
right form yet. Here is the estimate of $(\ast\ast)$:
\[
(\ast\ast)\leq ||\sqrt{\xi}\{V_1^\ast V_1(\xi^k\phi_1)-V_2^\ast
V_2(\xi^k\phi_2)\}||_{L^2_\xi}+||\sqrt{\xi}\{V_2^\ast
V_2(\xi^k\phi_2)-V_1^\ast V_2(\xi^k\phi_2)\}||_{L^2_\xi}
\]
The second term can be written as
\[
(\frac{\phi_1^{2k}}{\xi^{k-\frac{1}{2}}}-\frac{\phi_2^{2k}}{\xi^{k-\frac{1}{2}}})
\partial_\xi[\frac{1}{\xi^k}V_2(\xi^k\phi_2)]=-
(\frac{\phi_1^{2k}}{\xi^{k-\frac{1}{2}}}-\frac{\phi_2^{2k}}{\xi^{k-\frac{1}{2}}})
\frac{\xi^{2k}}{\phi_2^{2k}} \frac{V_2^\ast
V_2(\xi^k\phi_2)}{\xi^k}\,.
\]
Hence, by using \eqref{b1}, we deduce that
\begin{equation}\label{b2}
||\frac{\phi_1^{2k+1}}{\xi^{k+1}}-\frac{\phi_2^{2k+1}}{\xi^{k+1}}||_{L^\infty_\xi}
\sim
||\frac{\phi_1^{2k}}{\xi^k}-\frac{\phi_2^{2k}}{\xi^k}||_{L^\infty_\xi}\leq
C\mathcal{D}^{\frac{1}{2}}\,.
\end{equation}
Note that one cannot hope to get the bound of
$\frac{\phi_1^{2k+1}}{\xi^{2k+1}}-\frac{\phi_2^{2k+1}}{\xi^{2k+1}}$
with the $\mathcal{D}$-regularity. Of course, it is bounded by $A$,
but for the purpose of uniqueness, $A$-bound is not useful for the
difference terms. The idea is  to rearrange $(I)$ and $(II)$ so that
$\mathcal{D}^{\frac{1}{2}}$ can be extracted from each term. We
write the $L_\xi^\infty$ factors first and then $L_\xi^2$ factor at
the last.  $\mathcal{D}^{\frac{1}{2}}$ can also come from
$L_\xi^\infty$ factor thanks to the above estimate \eqref{b2}.   We
start with $(I)$:
\[
\begin{split}
(I)=(1-\frac{\phi_1^{2k}}{\phi_2^{2k}})V_2V_2^\ast
V_2(\xi^k\phi_2)-\underline{\frac{\phi_2^{2k}}{\xi^{2k}}
\partial_\xi[\frac{\phi_1^{2k}}{\phi_2^{2k}}]V_2^\ast
V_2(\xi^k\phi_2)}_{(\star)}
\end{split}
\]
The first term is written as
$(\frac{\phi_1^{2k}}{\xi^k}-\frac{\phi_2^{2k}}{\xi^k})\frac{\xi^{2k}}{\phi_2^{2k}}
\frac{V_2V_2^\ast V_2(\xi^k\phi_2)}{\xi^k}$ and thus its $L^2_\xi$
norm is bounded by $A$ and $\mathcal{D}^{\frac{1}{2}}$. The second
term can be treated as follows:
\[
\begin{split}
(\star)&=2k
\frac{\xi}{\phi_1}[\frac{\phi_1^{2k}}{\xi^k}\partial_\xi\phi_1
-\frac{\phi_1^{2k}}{\xi^k}\frac{\phi_1}{\phi_2}\partial_\xi\phi_2]\frac{V_2^\ast
V_2(\xi^k\phi_2)}{\xi^{k+1}}\\
&=\frac{2k}{2k+1} \frac{\xi}{\phi_1}\frac{V_2^\ast
V_2(\xi^k\phi_2)}{\xi^{k}}\frac{V_1(\xi^k\phi_1)-V_2(\xi^k\phi_2)}{\xi}+
2k \frac{\xi}{\phi_1}\partial_\xi\phi_2
(\frac{\phi_2^{2k}}{\xi^k}-\frac{\phi_1^{2k}}{\xi^k}\frac{\phi_1}{\phi_2})
\frac{V_2^\ast V_2(\xi^k\phi_2)}{\xi^{k+1}}
\end{split}
\]
Thus $(\star)$ is controlled by $A$ and $\mathcal{D}^{\frac{1}{2}}$.
Next we rearrange $(II)$:
\[
\begin{split}
&\frac{1}{2k}(II)=\frac{\partial_t\phi_1}{\phi_1}\{V_1V_1^\ast(\xi^ku_1)
-V_2V_2^\ast(\xi^ku_2)\}+\underline{\frac{V_2V_2^\ast(\xi^ku_2)}{\xi^k}
\{\frac{\xi^k\partial_t\phi_1}{\phi_1}-
\frac{\xi^k\partial_t\phi_2}{\phi_2}\}}_{(\star\star)}\\
&\:+\frac{\phi_1^{2k}}{\xi^{2k}}\frac{V_1^\ast(\xi^ku_1)}{\xi^{k+1}}
\{\xi^{k+1}\partial_\xi[\frac{\partial_t\phi_1}{\phi_1}]
-\xi^{k+1}\partial_\xi[\frac{\partial_t\phi_2}{\phi_2}] \}
+\xi\partial_\xi[\frac{\partial_t\phi_2}{\phi_2}]
\frac{\phi_1^{2k}}{\xi^{2k}}\{\frac{V_1^\ast(\xi^ku_1)}{\xi}
-\frac{V_2^\ast(\xi^ku_2)}{\xi}\}\\
&\:+\xi\partial_\xi[\frac{\partial_t\phi_2}{\phi_2}]
(\frac{\phi_1^{2k}}{\xi^k}-\frac{\phi_2^{2k}}{\xi^k})
\frac{V_2^\ast(\xi^ku_2)}{\xi^{k+1}}
\end{split}
\]
For $t$ derivative difference terms, use the equation to convert
into $u$ terms and apply the same argument. For instance, the second
term can be rewritten as
\[
\begin{split}
(\star\star)&=\frac{V_2V_2^\ast(\xi^ku_2)}{\xi^k}\{\frac{V_1^\ast(\xi^ku_1)}{\phi_1}
-\frac{V_2^\ast(\xi^ku_2)}{\phi_2}\}\\
&=\frac{V_2V_2^\ast(\xi^ku_2)}{\xi^k}\frac{V_1^\ast(\xi^ku_1)
-V_2^\ast(\xi^ku_2)}{\phi_1}+
\frac{V_2V_2^\ast(\xi^ku_2)}{\xi^k}\frac{V_2^\ast(\xi^ku_2)}{\xi^{k+1}}
\{\frac{\xi^{k+1}}{\phi_1}-\frac{\xi^{k+1}}{\phi_2}\}\,.
\end{split}
\]
It is easy to deduce that $(\star\star)$ bounded by $A$ and
$\mathcal{D}^{\frac{1}{2}}$. Other terms can be similarly estimated.
This finishes the proof of the uniqueness and Theorem \ref{thm}.

\section{Duality argument}\label{7}

% {\bf Here, I am doing as if we assumed that
%   $g(\xi=1)=0$  for the domain of $V^\ast$}

Here we would like to prove the existence for the linear problem
\eqref{FG}. This is a consequence of the   following proposition.

\begin{proposition}  For $f$ and $g$ in $L^1(0,T; L^2)$, there exists a unique
solution $(F,G)   $ to the linear system
 \begin{equation}\label{FG1} \left\{
\begin{split}
\partial_tF
- V^\ast G &=f \\
\partial_tG  + V
F &=g \\
G(\xi=1) & =0, F(t=0) = G(t=0) &= 0
\end{split} \right.
\end{equation}
 on $(0,T)$  which  satisfies
\begin{equation}\label{dual-L2}
\| (F,G)  \|_{C([0,T] ;  L^2)  } \leq C  \| (f,g)  \|_{L^1 L^2 }\,.
\end{equation}
 Moreover, if $(f,g) \in L^1(0,T; X^{k,j} \times   Y^{k, j}) $
for some integer  $j \leq \lceil k\rceil $
and for   $0\leq 2i \leq j-1   $, we have $ (V^\ast)^{2i} g = 0 $ at $\xi = 1$
and for   $1\leq  2i + 1  \leq j-1   $, we have $ (V)^{2i+1} f = 0 $ at $\xi = 1$,
 then
 \begin{equation} \label{dual-reg}
\| (F,G)  \|_{C([0,T] ;   X^{k,j} \times   Y^{k, j}  )  } \leq C
  \| (f,g)  \|_{L^1  (X^{k,j} \times  Y^{k, j})  }
\end{equation}
for some constant $C$ which only depends on
$\| \xi^k  \phi \|_{X^{k,\lceil k\rceil + 3 }} $.

\end{proposition}

The proof is based on a duality argument. Let $\A$ denote the set
\[ \A = \{ {\phi \choose  \psi} \in
C^\infty((0,\infty) \times (0,1]), \quad \hbox{such that}, \
\psi(\xi=1)=0, \    (\phi,\psi)_{t=T} = 0 \} .
\]
Hence,
$(F,G )  $ solves \eqref{FG1} on a time interval $(0,T)$
  if and only if for each
test function $(\phi, \psi)  \in   \A$, we have

  \begin{equation}  \label{FG-weak}
\begin{split}
\int_0^T \int  -F \partial_t \phi -  G V \phi &= \int_0^T \int f \phi   \\
\int_0^T \int  -G \partial_t \psi +  F V^\ast \psi &= \int_0^T \int g \psi.
\end{split}
 \end{equation}
We denote $\L {F \choose G} = {\partial_t F  -V^\ast G \choose \partial_t G +   V F}   $
% with the domain
defined on the core
$$   \{ {F \choose G}, \ | \partial_t {F \choose G} \in L^2_tL^2_\xi, \
 (F,G) \in L^2_t (\mathcal{D}(V),\mathcal{D}(V^\ast)  )     \} .  $$
 Hence, $\L$ can be extended in a unique way to a closed operator.
Moreover, $\A \subset \mathcal{D} \mathcal(\L^\ast)$, the dual of $\L$ and
 $ \L^\ast   {\phi  \choose \psi} = { -\partial_t \phi + V^\ast \psi \choose
 - \partial_t \psi - V \phi}   $.
Hence, \eqref{FG-weak}  holds for each $(\phi, \psi) \in \A$ if and only if
for  each $(\phi, \psi) \in A$, we have
 \begin{equation}  \label{FG-dual}
\int_0^T \int {F \choose G} . \L^\ast {\phi  \choose \psi} =
 \int_0^T \int {f \choose g} .  {\phi  \choose \psi}.
\end{equation}

We take ${\phi \choose  \psi} \in \A$ and denote $\L^\ast {\phi
\choose  \psi}  = {\Phi \choose  \Psi}  $. The energy estimate
written for $ \L^\ast  $ yields that
\[
\sup_{0\leq t \leq T} \| \phi \|_{L^2}^2 + \| \psi \|_{L^2}^2
\leq C \int_0^T   \| \Phi \|_{L^2}^2 + \| \Psi \|_{L^2}^2.
\]

Hence, the operator $\L^\ast $ defines a bijection between
$\A$ and $\L^\ast(\A)$. Let $S_0$ be the inverse. Hence
\begin{equation*}
\begin{array}{cccc}
S_0 :&  \L^\ast(\A)   & \to &  A  \\
  &   {\Phi \choose  \Psi} & \to &  {\Phi  \choose \Psi}
\end{array}
 \end{equation*}
and we have
\begin{equation*}
\| S_0 {\Phi \choose  \Psi}  \|_{C([0,T];   L^2)} \leq C \|{\phi  \choose \psi}   \|_{L^1 L^2}.
 \end{equation*}
We extend this operator by density to $\overline{\L^\ast(A )}^{L^1 L^2}$
and to $L^1 L^2$ by Hahn-Banach. We denote this extension by $S$. Hence

\begin{equation*}
\begin{array}{cccc}
S  :&  L^1 L^2   & \to &  C([0,T]; L^2)   \\
  &   {\Phi \choose  \Psi} & \to &  {\phi  \choose \psi}
\end{array}
 \end{equation*}

Now, we want to solve \eqref{FG1}, namely   $ \L   {F \choose G} =   {f \choose g} $
with $ {F \choose G}_{t=0} = 0 $. This is, of course,  equivalent to the fact that
\eqref{FG-dual} holds for each  $(\phi, \psi) \in \A$.

 Hence it is
enough that for all $ {\Phi \choose  \Psi} \in  L^1 L^2,   $
we have
\begin{equation}  \label{FG-dual-Phi}
\int_0^T \int {F \choose G} .{\Phi \choose  \Psi}    =
 \int_0^T \int {f \choose g} .  S {\Phi \choose  \Psi} .
\end{equation}
Therefore, it is enough to take  $ {F \choose G} =   S^\ast  {f \choose g}  $
where $S^\ast $ is the dual of $S$ which satisfies
\begin{equation*}
\begin{array}{cccc}
S^\ast   :&   \mathcal{M}(0,T;  L^2)    & \to &  L^\infty(0,T; L^2).    \\
 %  &   {\Phi \choose  \Psi} & \to &  {\phi  \choose \psi}
\end{array}
 \end{equation*}

In particular it maps $L^1L^2$ into $L^\infty L^2$. Hence
\eqref{dual-L2} holds with $C([0,T] ;  L^2)$ replaced by $L^\infty
([0,T] ;  L^2) $.
 At this stage we do not
know whether $ {F \choose G} $ is continuous.
This will actually follow from the regularity.

The uniqueness of $ {F \choose G} $ follows from the fact that if
$f=g=0$ in \eqref{FG1}, then ${0 \choose 0}$ is the unique solution to
\eqref{FG1} in $L^\infty ([0,T] ;  L^2) $. To prove this consider a
solution $  {F \choose G} \in L^\infty ([0,T] ;  L^2) $   to \eqref{FG1} with
 $f=g=0$.
We will  also use  the duality argument. Indeed, arguing as above
and changing the roles of $\L$ and $\L^\ast$, we can prove the
existence of a solution  $  {\phi  \choose \psi} \in  L^\infty ([0,T] ;  L^2) $
to the dual  problem
\begin{equation}\label{FG1-dual} \left\{
\begin{split}
- \partial_t \phi
+  V^\ast \psi &= \Phi  \\
- \partial_t\psi  - V
\phi  &= \Psi \\
\psi(\xi=1) & =0,  \\
\phi(t=T) = \psi(t=0) &= 0 .
\end{split} \right.
\end{equation}
for each $ {\Phi  \choose \Psi} \in  L^1 L^2.  $

Then, for each  $ {\Phi  \choose \Psi} \in  L^1 L^2  $, we consider
$ {\phi  \choose \psi}  $ a solution to \eqref{FG1-dual}.  Hence, we
can write \eqref{FG-dual} with the solution $ {\phi  \choose \psi}
$. This yields
 \begin{equation}  \label{FG-unique}
\int_0^T \int {F \choose G} .  {\Phi  \choose \Psi} = 0,
\end{equation}
which gives rise to the fact that $ F=G=0 $.

To prove \eqref{dual-reg}, we can argue by induction of $j$. We start with the
the case $j = 1$ and we first argue formally. Applying $V$ and $V^\ast$ to
\eqref{FG1}, we get

\begin{equation}\label{FG1-V} \left\{
\begin{split}
\partial_t V^\ast G  +  V^\ast V
F &= V^\ast g + V_t^\ast G \\
\partial_t VF
- V V^\ast G &= V f + V_t F \\
V F (\xi=1) & =0, \\
V F(t=0) = V^\ast G(t=0) &= 0 .
\end{split} \right.
\end{equation}
Notice that the boundary condition $V F (\xi=1)  =0, $ comes from
the fact that $g=0$ at $\xi=1$.
Hence, we deduce formally that
\begin{equation*}
\| (V F, V^\ast G)  \|_{L^\infty ([0,T] ;  L^2)  } \leq C  \| (V^\ast g + V_t^\ast G , Vf +V_t F  )
  \|_{L^1 L^2 }\,.
\end{equation*}

To make this rigorous, we define first
${Y_0 \choose Z_0 } $, the solution of  \eqref{FG1-V}  with
the right hand side replaced by $ { V^\ast g  \choose  Vf } $. Hence,
$ \L {Y_0 \choose Z_0 } = { V^\ast g  \choose  Vf } $.
Then, we define for each  integer $i$, ${Y_i \choose Z_i } $,
the solution of \eqref{FG1-V}  with
the right hand side replaced by ${  V_t^\ast  V^{\ast-1}  Y_{i-1}  \choose  V_t V^{-1} Z_{i-1} } =  { \frac{\phi_t}{\phi} Y_{i-1}  \choose  V_t V^{-1} Z_{i-1} } $.
Hence,

\begin{equation*}
\| (Y_0, Z_0)  \|_{L^\infty ([0,T] ;  L^2)  } \leq C  \| (V^\ast g  , Vf   )
  \|_{L^1 L^2 }
\end{equation*}
and using the fact that $V_t V^{-1}$ is bounded from $L^2$ to $L^2$, we get that
\begin{equation*}
\begin{split}
\| (Y_i, Z_i)  \|_{L^\infty ([0,T] ;  L^2)  } \leq C  \| ( Y_{i-1}, Z_{i-1}   )
  \|_{L^1 L^2 }  & \leq CT  \| ( Y_{i-1}, Z_{i-1}   )
  \|_{L^\infty  L^2 }  \\
& \leq C (CT)^i   \| ( V^\ast g  , Vf   )
  \|_{L^\infty  L^2 } \,.
\end{split}
\end{equation*}
Hence, denoting
$$ {Y \choose Z} = \sum_{i=0}^\infty {Y_i \choose Z_i},   $$
we see that ${Y \choose Z} $ solves
\begin{equation}\label{FG1-YZ} \left\{
\begin{split}
\partial_t Y   +  V^\ast Z  &= V^\ast g + V_t^\ast V^{\ast-1} Y  \\
\partial_t Z
- V Y  &= V f   + V_t V^{-1} Z  \\
Z  (\xi=1) & =0, \\
Z (t=0) = Y (t=0) &= 0 .
\end{split} \right.
\end{equation}

The operator $V$ can be thought of as an operator from
$\mathcal{D}(V) \to L^2_\xi$. It has a transpose $^tV$ which
goes from $L^2_\xi \to \mathcal{D}(V)'  $ where
$\mathcal{D}(V)' $  is the dual space of $\mathcal{D}(V) $.
It satisfies : for $f\in \mathcal{D}(V)$ and $u \in L^2$,

\begin{equation*}
(V(f), u)_{L^2, L^2} = (f,^tV(u) )_{\mathcal{D}(V), \mathcal{D}(V)'
}\,.
\end{equation*}
Moreover, if we identify $L^2$ to a subspace of $\mathcal{D}(V)'  $,
then  $^tV$ extends $V^\ast$  to $L^2$.

The same argument shows that we can also extend $V$ to $L^2$ by
considering $^tV^\ast$. Hence, one can also interpret the
system \eqref{FG1} as equalities in $\mathcal{V}'$ and
$\mathcal{V^\ast}'$ and the same remark holds for
\eqref{FG1-YZ}.

Using that $V[\partial_t (V^{-1} Z ) ] = \partial_t Z - V_t V^{-1} Z
$, we deduce that  $V [\partial_t (V^{-1} Z ) -V Y - V f  ] $.
Since, the kernel of $V$ is $\{0\}$, we deduce that
\begin{equation*}
\partial_t (V^{-1} Z ) - Y = f\,.
\end{equation*}

We also have  $V^\ast[\partial_t (V^{\ast-1} Y ) ] = \partial_t Y -
V^\ast_t V^{\ast-1} Y$. Hence, the first equation of \eqref{FG1-YZ}
can be written
$$ V^\ast[\partial_t (V^{\ast-1} Y ) + Z - g  ] = 0 \,.    $$
Hence since the kernel of $V^\ast$ is $\{0\}$ when we add the
vanishing of  the boundary condition at $\xi=1$,
% there exists a function
we deduce that
\begin{equation*}
% \partial_t (V^{-1} Z ) - Y = f
\partial_t (V^{\ast-1} Y ) + Z - g = 0\,.
\end{equation*}
By uniqueness for \eqref{FG1}, we deduce that
$F = V^{-1} Z $ and $G = V^{\ast-1} Y  $. Hence, we get that
$(F,G ) \in L^\infty(0,T; X^{k,1} \times Y^{k,1}) $.
In particular, this also shows that
$(\partial_t F, \partial_t G \in L^1(0,T; L^2) $.

Hence, integrating in time, we deduce that
$(F,G) \in C([0,T); L^2)$. In particular this
shows that \eqref{dual-L2} holds if have more
regularity on $(f,g)$. By a density argument, this shows that
\eqref{dual-L2} holds even if we only know that
 $(f,g) \in L^1L^2$.

Arguing by induction on $j$, we prove that if
$(f,g) \in L^1(0,T;  X^{k,j} \times Y^{k,j}) $, then
$(F,G) \in C([0,T]  ;   X^{k,j} \times Y^{k,j}) $.

\section{Acknowledgement}
N. M was partially supported by an NSF grant DMS-0703145. J. J
 was supported by an NSF grant DMS-0635607.

% \bibliographystyle{abbrv}
% \bibliography{biblio}
%\end{document}

\end{document}